\documentclass[11pt,twoside,a4paper]{article}
\usepackage{color,amscd,amsmath,amssymb,amsthm,latexsym,stmaryrd}
\usepackage{mathabx}
\usepackage{shuffle}
\usepackage[latin1,utf8]{inputenc}
\usepackage[T1]{fontenc}   
\usepackage[english]{babel}
\usepackage{tikz,tkz-tab}
\usepackage{hyperref}
\usepackage{appendix}
\input{xy}
\xyoption{all}

\definecolor{vert}{rgb}{0.,0.5,0.}

\newcommand{\Gr}{\mathbf{Gr^{\circlearrowright}}}
\newcommand{\PGr}{\mathbf{PGr^{\circlearrowright}}}
\newcommand{\Grc}{\mathbf{Gr}^{\uparrow}}
\newcommand{\PGrc}{\mathbf{PGr}^{\uparrow}}
\newcommand{\Gri}{\mathbf{Gr}^{\circlearrowright}_\mathrm{ind}}
\newcommand{\PGri}{\mathbf{PGr}^{\circlearrowright}_\mathrm{ind}}
\newcommand{\Grci}{\mathbf{Gr}^{\uparrow}_\mathrm{ind}}
\newcommand{\PGrci}{\mathbf{PGr}^{\uparrow}_\mathrm{ind}}
\newcommand{\Gacirc}{\Gamma^{\circlearrowright}}
\newcommand{\Ga}{\Gamma^{\uparrow}}
\newcommand{\Hom}{\mathbf{Hom}}

\newcommand{\rond}[1]{*++[o][F-]{#1}}
\newcommand{\K}{\mathbb{K}}
\newcommand{\Z}{\mathbb{Z}}
\newcommand{\C}{\mathbb{C}}
\newcommand{\R}{\mathbb{R}}
\newcommand{\N}{\mathbb{N}}
\renewcommand{\geq}{\geqslant}
\renewcommand{\leq}{\leqslant}

\newcommand{\sym}{\mathfrak{S}}
\newcommand{\calO}{\mathcal{O}}
\newcommand{\grapheo}{\mathcal{O}}
\newcommand{\End}{\mathrm{End}}

\newcommand{\cat}{\mathcal{C}}
\newcommand{\catssm}{\mathbf{Mod}_\sym }
\newcommand{\Trap}{\mathbf{TraP}}
\newcommand{\Prop}{\mathbf{ProP}}

\theoremstyle{plain}
\newtheorem{defi}{Definition}[subsection]
\newtheorem{theo}[defi]{Theorem}
\newtheorem{lemma}[defi]{Lemma}
\newtheorem{cor}[defi]{Corollary}
\newtheorem{prop}[defi]{Proposition}

\newtheorem{defiprop}[defi]{Definition-Proposition}

\theoremstyle{remark}
\newtheorem{remark}{Remark}[subsection]
\newtheorem{notation}{Notations}[subsection]
\newtheorem{example}{Example}[subsection]

\begin{document}

\title{ProPs of graphs and generalised traces}

\author{Pierre J. Clavier${}^{1,2}$, Lo\"ic Foissy${}^3$, Sylvie Paycha${}^{1,4}$\\
~\\
{\small \it $^1$ Universit\"at Potsdam, Institut f\"ur Mathematik,} \\
{\small \it Campus II - Golm, Haus 9} \\
{\small \it Karl-Liebknecht-Stra\ss e 24-25} \\
{\small \it D-14476 Potsdam, Germany}\\ 
~\\
{\small \it $^2$ Technische Universit\"at, Institut f\"ur Mathematik,}\\
{\small \it Str. des 17. Juni 136,}\\
{\small \it 10587 Berlin, Germany}\\
~\\
{\small \it $^3$ Univ. Littoral Côte d'Opale, UR 2597}\\
{\small \it LMPA, Laboratoire de Mathématiques Pures et Appliquées Joseph Liouville}\\
{\small \it F-62100 Calais, France}\\ 
~\\
{\small\it $^4$ On leave from the University of Clermont-Auvergne, Clermont-Ferrand, France}\\
~\\
{\small emails: \it clavier@math.uni-potsdam.de, foissy@univ-littoral.fr, paycha@math.uni-potsdam.de}}

\date{}

\maketitle

	\begin{abstract}
We assign  generalised convolutions (resp. 
traces) to graphs  whose edges are decorated by smooth kernels (resp. smoothing operators) on a closed manifold. To do so, we introduce the concept of TraPs (Traces and Permutations), which roughly correspond to ProPs (Products and Permutations) without 
vertical concatenation and equipped with families of generalised partial traces. They can be equipped with a ProP structure in deriving vertical 
concatenation from the partial traces and we relate TraPs to wheeled ProPs first introduced by Merkulov. We further build their free object and give 
precise proofs of universal properties of ProPs and TraPs.
	\end{abstract}
\textbf{Classification:} 
18M85, 46E99, 47G30
	\tableofcontents
	
	\section*{Introduction}
	
	\addcontentsline{toc}{section}{Introduction}
	
\subsection*{State of the art}

ProPs (\underline{Pro}ducts and \underline{P}ermutations)\footnote{The traditional notation for ProP is PROP or and  more recently prop. We choose to 
use ProP as an acronym with capital letters for the first letter of the words and use the same convention for the related concept of TraP.} provide an 
algebraic structure that allows to deal with an arbitrary number of inputs and outputs. As such they generalise many other algebraic structures such as 
operads, which have one output and multiple inputs. ProPs appeared in \cite{MacLane65} and later in the book \cite{BV73} in the context of cartesian 
categories. Operads stemmed from this work in \cite{May72}, although their origin can also be traced back to the earlier work \cite{BV68}\footnote{We 
thank B. Vallette for his enlightening comments on these historical aspects.}.
	
An important asset of ProPs over operads is that they encompass algebraic structures such as bialgebras and Hopf algebras that lie outside the realm of 
operads or co-operads. We refer the reader to \cite{Pirashvili01} for the study of bialgebras in the ProPs framework and \cite{Markl} for other 
classical examples of ProPs. 
	
Our two central examples of ProPs are the ProP $\Hom_V$ of homomorphisms  of a finite dimensional vector space $V$ which we generalise 
to the ProP $\Hom^c_V$ of continuous homomorphisms  of the  nuclear Fréchet space $V$, and the  ProP  $\Gr$ of graphs\footnote{We use Merkulov's 
notations.}. In the context of deformation quantisation,  the complex of oriented graphs whether 
directed or wheeled,  plays an important role in the construction of a free ProP   generated by a $\sym\times \sym^{op}$-module  (see e.g  
\cite[Paragraph 2.1.3]{Merkulov2004}). However,  to our knowledge, the ProP of oriented graphs, briefly  mentioned  in \cite{Ionescu07}, has not  yet found concrete applications in the 
perturbative approach to quantum field theory. Filling this gap is  a long term goal we  have in mind.

\subsection*{ProPs and oriented graphs}

		In space-time variables, a Feynman rule is expected to assign to a graph $G$ with $k$ incoming and $l$ outgoing edges, a correlation function 
	(it is actually a distribution) $K_G$ in $k+l$ variables. Our long term goal is to interpret the correlation function associated with the 
	composition $G\circ G'$ of two graphs as a generalised convolution $K_G \star K_{G'}$  of the correlation functions $K_{G}$ and $K_{G'}$  
	associated with $G $ and $G'$, aiming to derive the existence and the properties of the map $G\mapsto K_G$ from a universal property 
	of the ProP  structure on graphs.
	
	ProPs entail two operations, called horizontal and vertical concatenations, which are the natural operations implemented on oriented graphs. With the goal we have in mind, ProPs are therefore natural structures to consider.  We provide a precise formulation of the well-known fact that oriented graphs can   be equipped with a ProP structure as well as a complete proof (see Theorem 
	\ref{theo:ProP_graph}) of this   statement. We also give a similar statement for (resp. planar) vertex decorated 
	graphs in Theorem \ref{theo:ProP_graph_2} (resp. Theorem \ref{thm:prop_struct_gacirc}).
	The horizontal concatenation of this ProP is the natural concatenation of 
	graphs and the  vertical concatenation is   the composition, which   to a graph $G $ with $k$ incoming and $l$ outgoing edges, and a graph 
	$G'$  with $l$ incoming and $m$ outgoing edges, assigns a a graph $G \circ G'$ with $k$ incoming and $m$ outgoing edges. Roughly speaking, 
	$G\circ G'$ is obtained by gluing together the  {outgoing} edges of $G$ and the  incoming edges of $G'$ according to their indexation.  
	
	In Theorem \ref{thm:freeness_Gr}, we show that the ProP of oriented graphs is the free ProP generated by what we call  
	indecomposable graphs (see Definition \ref{def:indecomposable}). We provide a planar version of this result in Theorem \ref{thm:freeness_PGr}.
	These universal properties are generalised to decorated graphs in Theorem 
	\ref{thm:univ_prop_deco}. Such universal properties were stated without detailed proofs in   previous work, see  e.g.
	\cite[Proposition 57]{Markl} and   \cite{Vallette1,Vallette3}. 
	
	We make use of the universal property of oriented graphs  when decorating the corresponding   ProP $\Gr$  with another ProP whose structure 
	is   compatible with that  of the one on graphs (see Subsection \ref{subsec:propdecprop}). In particular, we show in Theorem 
	\ref{thm:freeness_GaX} that $\Ga(X)$ is the free ProP generated by the $\sym\times \sym^{op}$-module $X$. The decorating set $X$
	will eventually  be a ProP of smooth kernels. Along the way, we use Theorem \ref{thm:freeness_GaX} in Corollaries \ref{cor:homGrc}, 
	\ref{cor:lift_alg_over_Prop} to build algebra over ProPs; see Definition \ref{def:Palgebra}. The same constructions and universal properties 
	hold for edge-decorated oriented graphs, i.e. Feynman graphs (see Remark \ref{rk:decoration}).
	
We have chosen to work with the ProP $\Gr$ which comprises loops, although the latter play a passive role in the presentation of a ProP. Yet,  they will be relevant   in the presentation of TraPs that come  later in Section 
	\ref{subsection:generalised_graphs}. Introducing them right at the beginning unifies the presentation, since otherwise two similar constructions over two different sets of graphs would have been necessary. 
	
	\subsection*{Correlation functions and generalised convolutions}
	
	By  means of  blow-up   methods, generalised convolutions of Green functions were built on a closed Riemannian manifold  in    \cite{DangZhang17}, with the goal of renormalising multiple loop 
	amplitudes for Euclidean QFT on Riemannian manifolds.
	We hope to be able to simplify    the intricate analytic aspects of the renormalisation procedure for   multiple loop amplitudes,  by adopting an algebraic point of view on correlation functions 
	using ProPs.  There were    earlier attempts to describe QFT theories in 
	terms of ProPs (see e.g. \cite{Ionescu,Ionescu07}), yet to our knowledge, none with  the focus we are putting on generalised convolutions to 
	describe correlation functions. 
	
	Our goal is to use the ProP (actually TraP) structure of graphs decorated by distribution (e.g. Green) kernels to build the  resulting convolutions as generalised convolutions of kernels associated with the 
	decorated graph. The  expected singularities of the resulting correlation functions are immediate obstacles in defining such generalised convolutions.  In this paper, we  focus on the smooth setup, for which the correlation functions are smooth. Our goal in the smooth case  is to 
	provide an adequate algebraic and analytic framework in which we carry out this construction for correlation functions emanating 
	from graphs decorated  with smooth kernels. 
	
	A smooth kernel $K$ on   a closed manifold $M$  gives rise to a smoothing   operator 
	\[ {\mathcal D}^\prime(M)\ni u\longmapsto \left(L_K(u): x\mapsto \int_M K(x, y)\, u(y)\, dy\right)\in C^\infty (M),\] 
	which maps the space ${\mathcal D}^\prime(M)$ of distributions on $M$ to the space $C^\infty (M)$  of smooth functions on $M$.
	So, in  generalising the convolution of smooth kernels, we generalise the composition of smoothing operators.  

	\subsection*{Graphs with oriented cycles and TraPs}
	
One challenge present  both in the smooth and non-smooth case is the treatment of oriented cycles.
A first step   is the  study of the sub-ProPs of (decorated and non decorated) graphs
without oriented cycles carried out in Subsection \ref{subsection:cycless_graph}. These structures are then used in Section 
\ref{section:Hom_dec_graphs}. Yet in order to tackle Feynman graphs, we need 
graphs that can contain  oriented cycles.

TraPs (see Definition \ref{defi:Trap}) provide a natural structure to   take into account oriented cycles in the graph. It indeed 
provides a framework to host 	(partial) traces on graphs that generalise the 
ordinary trace Tr$ (L_K)=\int_M K(x, x)\, dx$. 
The TraP structure, which we relate in Section \ref{subsec:Trap_wProPs} to Merkulov's notion of wheeled ProPs (see Corollary 
\ref{cor:lien_Trap_wProP}), encompasses families of generalised traces. In Definition \ref{defi:graph}, we introduce the set of   $\Gr$ of  
graphs which includes graphs with oriented loops. Proposition \ref{prop:TrapGG} shows that $\Gr$ can be equipped with  a TraP structure 
and  Theorem  \ref{theoFGlibre} shows that this TraP is free. This result is then generalised by Theorem \ref{theo:freeTrapplanar} 
which describes free TraP. An appendix is dedicated to the precise definition of the trace on  $\Gr$.  
Paragraph \ref{subsection:free_Traps} provides a description of a free TraP generated by a given set. 
	
We have postponed the detailed proofs of   two main results  Theorem \ref{thm:freeness_Gr} and Theorem \ref{theoFGlibre}  
to the appendix, so as not to burden the bulk of the paper with technicalities. A sketch of the proof is given straight 
after the statement so that the reader can nevertheless have an idea of the proof.


Alongside the ProP of graphs, another  guiding example throughout the paper is the ProP of 
homomorphisms, which we investigate in the infinite dimensional setup. In Theorem \ref{thm:Hom_V_generalised}, we introduce 
the ProP $\Hom_V^c$ of {\it continuous morphisms} for a topological Fr\'echet nuclear space $V$, which generalises the well-known ProP 
$\Hom_V$  (see e.g. the classical monograph \cite{Markl}) of morphisms on a finite dimensional vector space (see Definition\ref{defi:Hom_V}). 

In Proposition \ref{prop:trop_prop} we define the TraP  $(\Hom_V^c(k, l))_{k,l\geq 0}$ corresponding to the ProP $\Hom_V^c$ of continuous
morphisms on an infinite dimensional Fr\'echet nuclear space $V$. In the finite dimensional case it reduces to the TraP  
$(\Hom_V(k, l))_{k,l\geq 0}$.

\subsection*{Functorial properties: TraPs versus wheeled ProPs}
 
Much in the same way as  we build the functor (see Proposition \ref{prop:functGamma_2})  
\begin{equation*}
 \Ga:\catssm\longrightarrow \Prop
\end{equation*}
from the category $\catssm$ (Definition \ref{deficatssm}) of $\sym\times \sym^{op}$-modules to the category $\Prop$ 
(Definition \ref{deficatProP}) of ProPs, which to   
a $\sym\times \sym^{op}$-module $P$ assigns a graph-ProP $\Ga(P)$ whose vertices are decorated by $P$, following Merkulov's approach, we build a functor 
\[ \Gacirc: \catssm\longrightarrow \Trap\] which takes $\sym\times\sym^{op}$-modules
to TraPs (Proposition \ref{prop::GaTraP}). 
Combining them with forgetful functors from $\Prop$ or $\Trap$ to $\catssm$,
we can view $\Ga$ as an endofunctor of $\catssm$ or of $\Prop$, and $\Gacirc$
and an endofunctor of $\catssm$ or of $\Trap$.

In Paragraph \ref{subsec:monad}, we provide a detailed description of Merkulov's construction of the monad structure of $\Gacirc$ on the 
category $\catssm$ (Proposition \ref{prop:monad}), by means of which  (wheeled)  ProPs are defined. Our Definition 
\ref{defi:Trap} of TraPs corresponds to unital wheeled ProPs. 
Using the construction of free TraPs of Section \ref{subsection:free_Traps},  in Corollary \ref{cor:lien_Trap_wProP} we establish an 
isomorphism between the categories of wheeled TraPs on the one hand  and  of TraPs on the other hand.

Our constructions have some similarity with those underlying traced monoidal categories introduced in \cite{JSV}, yet  the  framework and the  axioms in  the two approaches   differ.

\subsection*{TraPs viewed as  ProPs: the trace and the composition}	
	
It follows from the identification between TraPs and wheeled ProPs  mentioned above, that a TraP is a ProP.  In Proposition \ref{prop:trop_prop}, we 
provide a detailed description of the ProP structure on TraPs as a result of the fact that both the trace  and composition of morphisms (see Lemma 
\ref{lem:compo_dual_pairing}) can be expressed in terms of a dual pairing. Let us illustrate this fact in the finite dimensional setup. 

	Given a finite dimensional vector space $V$ over a commutative field $\K$, both the composition and the  trace on the algebra of morphisms 
	$\Hom(V)\simeq V^*\otimes    V$  involve the dual pairing \[V^*\times V\ni  (v^*, w)\mapsto v^*(w)\in \K,\] between the algebraic dual 
	$V^*$ and the space $V$.

	Extending this to the infinite dimensional setup requires the use of a completed tensor product $\widehat \otimes$ in order to have an 
	isomorphism 
	\[\Hom_V^c(k,l)\simeq\left(V'\right)^{\widehat\otimes  k}\widehat \otimes V^{\widehat\otimes  l},\]  
	where $\Hom_V^c(k,l)$ stands for the algebra of continuous morphisms from $V^{\widehat\otimes  l}$ to $V^{\widehat\otimes  k}$ (see 
	Definition \ref{defi:Hom_V_generalised}) and
	$V'$ for the topological dual of a topological space $V$. This holds in the framework of Fr\'echet nuclear spaces which form a monoidal 
	category under
	the completed tensor product  $E\widehat\otimes  F$  (Lemma \ref{lem:prod_Frechet_nuclear}). On Fr\'echet nuclear spaces, the composition can 
	indeed be described as a dual pairing (see Lemma \ref{lem:compo_dual_pairing}) so it  comes as no surprise that (see Proposition 
	\ref{prop:TropPropHomc})
	for a Fr\'echet nuclear space $V$, the ProP built from the TraP  $(\Hom_V^c(k, l))_{k,l\geq 0}$ is 
	isomorphic, as a ProP, to the ProP $\Hom_V^c$. In  the finite dimensional setting, this induces an isomorphism of ProPs between TraP  
	$(\Hom_V(k, l))_{k,l\geq 0}$  and  $\Hom_V$. 

   In practice,  the partial trace maps $t_{i,j}$ arising in the definition of a TraP might not be defined on every operator. To circumvent this 
   difficulty, in Paragraph \ref{subsec:partial}, we introduce the notion of quasi-TraP, which we embed in a complete TraP.
   
   \subsection*{Openings}
   
	   As announced in the abstract, by means of a  (quasi-) TraP structure, we were able to build  generalised convolutions (resp. traces) associated with   graphs decorated with smooth 
	   kernels (see Remark \ref{rk:smoothkernelquasitrop}). We expect this algebraic approach  to enable us to tackle non smooth kernels and   thus  to describe  correlation functions as generalised 
	   convolutions of distribution kernels associated with graphs. At this stage these are open questions we hope to address in future work. 
	   
	 \section*{Acknowledgements} 
	 The first and last authors would like to acknowledge the Perimeter Institute for hosting them during the preparation of this paper.
	 They are also grateful for inspiring discussions with Matilde Marcolli. The three authors thank  Bruno 
	 Vallette and Dominique Manchon most warmly for pointing the wheeled ProP literature to us. The article was partially restructured in the 
	 aftermath of their suggestions and we appreciate Dominique Manchon's comments after that.

	\section*{Notation}\begin{enumerate}
		\item Any vector space in this text is taken over $\K$, chosen to be the field $\R$ or the field $\C$.	
		\item For any $k\in  \N_0=\Z_{\geq0}$, we denote by $[k]$ the set $\{1,\ldots,k\}$. In particular, $[0]=\emptyset$.
	\end{enumerate} 

\section{Two guiding examples of ProPs}

We define ProPs, the first main protagonists of the paper, and two ProPs which we shall use as a driving thread throughout the paper.

\subsection{Definition}

Following \cite{Vallette1,Markl}, a ProP is a symmetric strict monoidal category, whose objects are identified with $(\N_0)^2$
and such that the tensor product of two objects is identified with the sum of integers on each copy of $\N_0$. Here is a more detailed description.

\begin{defi} \label{def:prop}
A \textbf{ProP} is a family $P=(P(k,l))_{k,l\in  \N_0}$ of vector spaces such that:
\begin{enumerate}
\item $P$ is a $\sym\times\sym^{op}$-module, that is to say, for any $(k,l)\in  \N_0^2$, $P(k,l)$
is a $\sym_l\times \sym_k^{op}$-module. In other words, there exist maps
\begin{align*}
&\left\{\begin{array}{rcl}
\sym_l\times P(k,l)&\longrightarrow&P(k,l)\\
(\sigma,p)&\longrightarrow&\sigma\cdot p,
\end{array}\right.&
&\left\{\begin{array}{rcl}
P(k,l)\times \sym_k&\longrightarrow&P(k,l)\\
(p,\tau)&\longrightarrow&p\cdot \tau,
\end{array}\right.
\end{align*}
such that for any $(k,l)\in  \N_0^2$, for any $(\sigma,\sigma',\tau,\tau')\in \sym_l^2\times \sym_k^2$, for any $p\in P(k,l)$,
\begin{align*}
&&\mathrm{Id}_{[l]}\cdot p&=p\cdot \mathrm{Id}_{[k]}=p,\\
\sigma\cdot (\sigma'\cdot p)&=(\sigma\sigma')\cdot p,&
\sigma \cdot (p\cdot \tau)&=(\sigma\cdot p)\cdot \tau,&
(p\cdot\tau)\cdot \tau'&=p\cdot(\tau\tau').
\end{align*}
\item For any $(k,l,k',l')\in  \N_0^4$, there exists a product $*$ from $P(k,l)\otimes P(k',l')$ to $P(k+k',l+l')$
such that:
\begin{enumerate}
\item For any $(k,l,k',l',k'',l'')\in  \N_0^6$, for any $(p,p',p'')\in P(k,l)\times  P(k',l') \times P(k'',l'')$,
\[p*(p'*p'')=(p*p')*p''.\]
\item There exists $I_0\in P(0,0)$, such that for any $(k,l)\in  \N_0^2$, for any $p\in P(k,l)$,
\[p*I_0=I_0*p=p.\]
\end{enumerate}
This product $*$ is called the  \textbf{horizontal concatenation}.
\item For any $(k,l,m)\in  \N_0^3$, there exists a product $\circ$ from $P(l,m)\otimes P(k,l)$ to $P(k,m)$ such that:
\begin{enumerate}
\item For any $(k,l,m,n)\in  \N_0^4$, for any $(p,q,r)\in P(m,n)\times P(l,m)\times P(k,l)$,
\[p\circ (q\circ r)=(p\circ q)\circ r.\]
\item There exists $I_1\in P(1,1)$, such that for any $(k,l)\in  \N_0^2$, for any $p\in P(k,l)$,
\[p\circ I_k=I_l\circ p=p,\]
where we put $I_n=I_1^{*n}$ for any $n\in  \N_0$, with the convention $I_1^{*0}=I_0$.
\end{enumerate}
This product $\circ$ is called the  \textbf{vertical concatenation}.
\item The vertical and horizontal concatenations are compatible: for any $(k,k',l,l',m,m') \in  \N_0^6$, 
for any $(p,p',q,q')\in P(l,m)\times P(l',m') \times P(k,l) \times P(k',l')$,
\[(p*p')\circ(q*q')=(p\circ q)*(p'{\circ}q').\]
\item The vertical concatenation and the action of $\sym\times \sym^{op}$ are compatible:
for any $(k,l,m)\in  \N_0^3$, for any $(p,q)\in P(l,m)\times P(k,l)$, for any $(\sigma,\tau,\nu) \in \sym_m\times\sym_l\times\sym_k$,
\begin{align*}
\sigma\cdot(p\circ q)&=(\sigma\cdot p)\circ q,&
(p\circ q)\cdot \nu&=p\circ (q\cdot \nu),&
(p\cdot \tau)\circ q&=p\circ (\tau\cdot q).
\end{align*}
\item The horizontal concatenation and the action of $\sym\times \sym^{op}$ are compatible:
\begin{enumerate}
\item For any $(k,k',l,l')\in  \N_0^4$, for any $(p,p')\in P(k,l)\times P(k',l')$, 
for any $(\sigma,\sigma',\tau,\tau') \in \sym_l\times\sym_{l'}\times\sym_k\times\sym_{k'}$,
\begin{align*}
(\sigma\cdot p)*(\sigma'\cdot p')&=(\sigma\otimes \sigma')\cdot(p*p'),&
(p\cdot \tau)*(p'\cdot \tau')&=(p*p')\cdot (\tau \otimes \tau'),
\end{align*}
where for any $\alpha \in \sym_m$, $\beta \in \sym_n$, $\alpha \otimes \beta \in \sym_{m+n}$ is defined by:
\[\alpha \otimes \beta(i)=\begin{cases}
\alpha(i)\mbox{ if }i\leq m,\\
\beta(i-m)+m\mbox{ if }i>m.
\end{cases}\]
\item (Commutativity of the horizontal concatenation). For any $(k,k',l,l')\in \N_0^4$, for any $(p,p')\in P(k,l)\times P(k',l')$,
\begin{equation}\label{eqpstarpprime}c_{l,l'}\cdot (p*p')=(p'*p)\cdot c_{k,k'},\end{equation}
where for any $(m,n)\in  \N_0^2$, $c_{m,n}\in \sym_{m+n}$ is defined by:
\begin{align} 
\label{defcmn} c_{m,n}(i)&=\begin{cases}
i+n\mbox{ if }i\leq m,\\
i-m\mbox{ if }i>m.
\end{cases}
\end{align}
\end{enumerate}\end{enumerate}\end{defi}

\begin{remark}
 \begin{enumerate}
  \item Note that $c_{k, 0}= \mathrm{Id}_{[k]}=c_{0,k}$.
  \item In particular, $(P(0,0),*)$ is a unitary associative and commutative algebra, whose unit is $I_0$, which, consequently is unique.
  \item Similarly, $(P(1,1),\circ)$ is a unitary associative non commutative algebra, whose unit is $I_1$ which, consequently is unique.
  \item For any $\sigma \in \sym_k$, as a consequence of the compatibility between the vertical concatenation and
the action of $\sym\times \sym^{op}$ and the definition of $I_k\in P(k,k)$:
\[I_k\cdot \sigma=(I_k\cdot \sigma)\circ I_k=I_k\circ (\sigma\cdot I_k)=\sigma\cdot I_k.\]
Hence, $I_k\cdot \sigma=\sigma\cdot I_k$.
  \item By the commutativity axiom, if $p\in P(k,l)$ and $p_0 \in P(0,0)$, by the first item of this Remark,
  it follows   from (\ref{eqpstarpprime}) that $p*p_0=p_0*p$. So the elements of $P(0,0)$ are central for the horizontal concatenation.
If $q\in P(l,m)$, by the compatibility between the two concatenations:
\begin{align*}
(p*p_0)\circ q&=(p*p_0)\circ (q*I_0)\\
&=(p\circ q)*(p_0\circ I_0)\\
&=(p\circ q)*p_0.
\end{align*}
Similarly, $p\circ (q*p_0)=(p\circ q)*p_0$.
 \end{enumerate}
\end{remark}

We adapt the definition of morphisms of ProPs of \cite{Vallette1} in our non categorical language.
\begin{defi}      
     Let $P=(P(k,l))_{k,l\geq0}$ and $Q=(Q(k,l))_{k,l\geq0}$ be two ProPs. A \textbf{morphism of ProPs} is a family 
     $\phi=(\phi_{k,l})_{k,l\geq0}$ of linear  maps $\phi_{k,l}:P(k,l)\mapsto Q(k,l)$ which form a morphism for the horizontal 
     concatenation, the vertical concatenation and the actions of the symmetric groups. More precisely, for any $(k,l,m,n)\in \N_0^4$:
     \begin{itemize}
      \item $\forall (p,q)\in P(l,m)\times P(k,l),~\phi_{k,m}(p\circ q) = \phi_{l,m}(p)\circ \phi_{k,l}(q)$,
      \item $\forall (p,q)\in P(k,l)\times P(n,m),~\phi_{k+n,l+m}(p* q) = \phi_{k,l}(p)* \phi_{n,m}(q)$,
      \item $\forall (\sigma,p)\in\sym_l\times P(k,l),~\phi_{k,l}(\sigma.p)=\sigma.\phi_{k,l}(p)$,
      \item $\forall (p,\tau)\in P(k,l)\times\sym_k,~\phi_{k,l}(p.\tau)=\phi_{k,l}(p).\tau$.
     \end{itemize}
 By abuse of  notation, we  shall write $\phi(p)$ instead of $\phi_{k,l}(p)$ for $p\in P(k,l)$. 
    \end{defi}

    In particular, ProPs form a category. 
    \begin{defi} \label{deficatProP} 
     Let $\Prop$ be the category with objects given by $P=(P(k,l))_{(k,l)\in  \N_0^2}$ and the morphisms of which
     are morphisms $\phi:P\longrightarrow Q$ of ProPs given by families $(\phi_{k,l})_{(k,l)\in  \N_0^2}$. Here, for any $(k,l)\in  \N_0^2$,
     $\phi_{k,l}:P(k,l)\longrightarrow Q(k,l)$ is a morphism of $\sym_l\otimes \sym_k^{op}$-modules, 
  compatible with the vertical
     and horizontal concatenations, which sends the units $I_0$ and $I_1$ of $P$ to the corresponding units of $Q$.
 More explicitly, we have that
\begin{itemize}
\item For any $(k,l,k',l')\in \N_0^4$, for any $(p,p')\in P(k,l)\times P(k',l')$, 
$\phi_{k+k',l+l'}(p*p')=\phi_{k,l}(p)*\phi_{k', l'}(p')$.
\item For any $(k,l,m)\in \N_0^3$, for any $(p,p')\in P(l,m)\times P(k,l)$,
$\phi_{k,m}(p\circ p')=\phi_{l,m}(p)\circ \phi_{k,l}(p')$.
\item $\phi_{0,0}(I_0)=J_0$ and $\phi_{1,1}(I_1)=J_1$, where $I_0$,  $I_1$ are the units of $P$ and $J_0$, $J_1$ are the units of $Q$.
\end{itemize}     
    \end{defi}
    
Let $P=(P(k,l))_{k,l\geqslant 0}$ be a ProP and, for any $k,l\geqslant 0$, $Q(k,l)$ be a subspace of $P(k,l)$.
We shall say that $Q=(Q(k,l))_{k,l\geqslant 0}$ is a sub-ProP of $P$ if it is stable under the horizontal and vertical compositions,
under the action of the symmetric groups and if it contains the units  $I_0$ and $I_1$. More precisely:
\begin{itemize}
\item For any $(k,l,m)\in  \N_0^3$, $Q(l,m)\circ Q(k,l)\subseteq Q(k,m)$.
\item For any $(k,l,k',l')\in  \N_0^4$, $Q(k,l)*Q(k',l')\subseteq Q(k+k',l+l')$.
\item For any $(k,l)\in  \N_0^2$, for any $(\sigma,\tau) \in \sym_l\times \sym_k$, $\sigma.Q(k,l).\tau\subseteq Q(k,l)$.
\item $I_0\in Q(0,0)$ and $I_1\in Q(1,1)$.
\end{itemize} 

Let $P$ be a ProP.
	\begin{itemize}
\item If $Q$ is a sub-ProP of $P$, then $Q$ is also a ProP, 
and the canonical injection from $P$ to $Q$ is a ProP morphism. 
\item If $(Q_i)_{i\in I}$ is a family of sub-ProPs of $P$, then $\displaystyle \bigcap_{i\in I} Q_i$ is also a sub-ProP of $P$.
\end{itemize}
This leads to the following 
\begin{defiprop} 
Let $P$ be a ProP.
If for any $k,l\geqslant 0$, $R(k,l)$ is a subspace of $P(k,l)$, then there exists a smallest sub-ProP of $P$ containing
$R=(R(k,l))_{k,l\geqslant 0}$ 
\[\langle R\rangle:= \bigcap_{\substack{\mbox{\scriptsize $Q$ sub-ProP}\\ \mbox{\scriptsize of $P$ containing $R$}}}  Q.\] 
\end{defiprop}

\begin{remark}
Since   $Q$ contains $I_1$, by $*$-stability, $Q$ contains $\underbrace{I_1*\ldots*I_1}_{k\, \text{times}}=I_k$ and as a consequence of stability under the action of the symmetry groups,
$Q$ further contains $\sigma.I_k.\tau$.
\end{remark} 

\subsection{The ProP of  linear morphisms: $\Hom_V$} \label{sec:propHom_finiteDim}

We recall a classical example of ProP.
\begin{defiprop} \label{defi:Hom_V} 
Given a finite dimensional $\K$-vector space $V$, the ProP 
$\Hom_V$ is defined in the following way:
\begin{enumerate}
\item For any $k,l\in  \N_0$, \[\Hom_V(k,l):=\Hom(V^{\otimes k}, V^{\otimes l}).\]
\item For any $\sigma \in \sym_n$, let  $\theta_\sigma$ be the endomorphism of $V^{\otimes n}$ defined by
\[\theta_\sigma(v_1\otimes \ldots \otimes v_n):=v_{\sigma^{-1}(1)}\otimes \ldots \otimes v_{\sigma^{-1}(n)}.\] 
This defines a left action of $\sym_n$ on $V^{\otimes n}$. 
For any $(k,l)\in  \N_0^2$, for any $f\in \Hom_V (k,l)$, for any  $(\sigma,\tau) \in \sym_l\times \sym_k$, we set:
\begin{align*}
\sigma\cdot f&:=\theta_\sigma \circ f ,&f\cdot \tau&:=f\circ \theta_\tau.
\end{align*} 
\item The horizontal concatenation is the tensor product of maps and $I_0:\K\longrightarrow\K$ is the identity 
map $I_0:=\mathrm{Id}_\K$.
\item The vertical concatenation is the usual composition of maps and $I_1:V\longrightarrow V$ is the identity map $I_1:=\mathrm{Id}_V$.
\end{enumerate}
 \end{defiprop} 
\begin{remark}
 This ProP is mentioned in \cite{Vallette1} and \cite{Markl}, but without an explicit proof of its ProP structure. We add such a proof here for completeness and in 
 preparation for the
 infinite dimensional case, which will be similar in spirit.
\end{remark}

     \begin{remark}
  Following {our} convention for a  ProP  $P=(P(k,l))_{k,l\in \N_0}$, where an element in $  P(k,l)$ has  {``$k$ entries and $l$ exits''}, 
     for the ProP $\Hom_V$,  an element $f\in\Hom_V(k,l)$ has  {``$k$ entries and $l$ exits''}.
    \end{remark}

 \begin{proof}
  \begin{enumerate}
   \item The maps $\theta_\sigma$ turns $\Hom_V$ into a $\sym_l\times \sym_k^{op}$-module by associativity of the composition product.
   \item The horizontal concatenation is associative as a result of the associativity of the tensor product $\otimes$, and we trivially have that 
   $\otimes$ maps $\Hom_V(k,l)\otimes\Hom_V(k',l')$ to $\Hom_V(k+k',l+l')$.
    Furthermore, if $(k,l)\in  \N_0^2$ and  $f\in\Hom_V(k,l)$, for any $v\in V^{\otimes k}$, we have
       \begin{equation*}
    (I_0\otimes f)(v) =(I_0\otimes f)(1.v):= I_0(1)\otimes f(v) = 1_\K\otimes f(v) = f(v)
   \end{equation*}
   \item The vertical concatenation is associative as the consequence of the associativity of the composition product. We furthermore have 
   $I_n:=I_1^{\otimes n}=\mathrm{Id}_V^{\otimes n} = \mathrm{Id}_{V^{\otimes n}}$ where the last identity follows
   from the definition of the tensor product of maps.
   \item For any $f\in\Hom_V(l,m)$, $f'\in\Hom_V(l',m')$, $g\in\Hom_V(k,l)$, $g'\in\Hom_V(k',l')$, $v\in V^{\otimes k}$ and 
   $v'\in V^{\otimes k'}$ we have
   \begin{align*}
    (f\otimes f')\circ (g\otimes g')(v\otimes v') & = (f\otimes f')(g(v)\otimes g'(v')) \\
						  & = (f\circ g)(v)\otimes (f'\circ g')(v') \\
						  & = [(f\circ g)\otimes (f'\circ g')](v\otimes v').
   \end{align*}
   Thus, the horizontal and vertical concatenation are compatible.
   \item The vertical concatenation and the action of $\sym\times \sym^{op}$ are compatible by associativity of the composition product.
   \item For any $f\in \Hom_V(k,l)$, $f'\in \Hom_V(k',l')$, $\sigma \in \sym_l$, $\sigma'\in \sym_{l'}$, $v\in V^{\otimes k}$, 
   $v'\in V^{\otimes k'}$ we have 
   \begin{align*}
    (\sigma.f)\otimes(\sigma'.f')(v\otimes v') & = ({\theta_\sigma}\circ f)\otimes ({\theta_{\sigma'}}\circ f')(v\otimes v') \\
					       & = {\theta_\sigma}(f(v))\otimes ({\theta_{\sigma'}}f'(v') \\
					       & = ({\theta_\sigma}\otimes {\theta_{\sigma'}})[f(v)\otimes f'(v')] \\
					       & = (\sigma\otimes \sigma').(f\otimes f')(v\otimes v').
   \end{align*}
   Similarly, we have $(f.\tau)\otimes(f'.\tau') = (f\otimes f').(\tau\otimes\tau')$ and $c_{l,l'}\cdot (f*f')=(f'*f)\cdot c_{k,k'}$, therefore the horizontal action of $\sym\times \sym^{op}$ are compatible. \qedhere
  \end{enumerate}
 \end{proof}

\begin{remark}\label{ex:findimtensor}  
  Let  $V$ be an $n$-dimensional  $\K$-vector space equipped with a basis $(e_1,\cdots, e_n)$, and let $(e^1,\cdots, e^n)$ be the dual basis. We 
  write  $v_j=\sum_{k_j=1}^n b_j^{k_j} e_{k_j}$ the elements of $V$ and $v_i^*= \sum_{k_i=1}^n a_{k_i}^i e^{k_i}$ the elements of $V^*$. Then an element 
  $f=v_1^*\otimes \cdots \otimes v_k^*\otimes v_1\otimes \cdots \otimes v_l\in \Hom_V(k, l)$ reads 
  \[  f= \sum_{\vec I, \vec J} a_{\vec J}^{\vec I} \, e^{\vec J}  \otimes e_{ \vec I},\]
 where,  for two finite sequences $\vec I=(i_1, \cdots, i_k)$, $\vec J=(j_1, \cdots, j_l)$  of $k$ and $l$ elements of $[n]$, we have set
\begin{align*}
e_{\vec I}   &:=e_{i_1} \otimes \cdots \otimes e_{i_k}; &e^{\vec J} &:=e^{j_1} \otimes \cdots \otimes e^{j_l}
\end{align*} 
  and the $a^{\vec I}_{\vec J} \in \K$ are coefficients built from sums of products of the coefficients $a_{k_i}^i $ and $b_j^{k_j}$. In 
 particular, an element of this ProP is completely determined by this collection of numbers $a^{\vec I}_{\vec J}$. We can therefore view $f$ as a 
 map from pairs of subsets $I,J$ of $[n]$ with $k$ and $l$ elements respectively into $\K$.
 
 It follows that for any $n$-dimensional vector space $V$, $\Hom_V$ is isomorphic as a ProP to the set of maps from pairs of finite sequences of elements of $[n]$ to $\K$:
\begin{equation*}
 \Hom_V \simeq \left(\left\{a:\mathrm{Seq}_k([n])\times\mathrm{Seq}_l([n])\longrightarrow\K\right\}\right)_{k,l\geq0}.
\end{equation*}
\end{remark}

\subsection{The ProP of graphs: $\Gr$} \label{sectiongraphes}

\begin{defi} \label{defi:graph}
A \textbf{graph} is a family $G=(V(G),E(G),I(G),O(G),IO(G),L(G),s,t,\alpha,\beta)$, where:
\begin{enumerate}
\item $V(G)$ (set of vertices), $E(G)$ (set of internal edges), $I(G)$ (set of input edges),
$O(G)$ (set of output edges), $IO(G)$ (set of input-output edges) and $L(G)$ (set of loops, that is to say
edges with no endpoints) are finite (maybe empty) sets.
\item $s:E(G)\sqcup O(G)\longrightarrow V(G)$ is a map (source map).
\item $t:E(G)\sqcup I(G)\longrightarrow V(G)$ is a map (target map).
\item $\alpha:I(G)\sqcup IO(G)\longrightarrow [i(G)]$ is a bijection, with $i(G)=|I(G)|+|IO(G)|$
(indexation of the input edges).
\item $\beta:O(G)\sqcup IO(G)\longrightarrow [o(G)]$ is a bijection, with $o(G)=|O(G)|+|IO(G)|$
(indexation of the output edges).
\end{enumerate}
\end{defi}

\begin{example}\label{ex4}
Here is a graph $G$ : 
\begin{align*}
V(G)&=\{x,y\},&E(G)&=\{a,b\},&I(G)&=\{c,d\},&O(G)&=\{e,f\},&IO(G)&=\{g\}, &L(G)&=\{h,k\},
\end{align*}
and:
\begin{align*}
s&:\left\{\begin{array}{rcl}
a&\mapsto&y\\
b&\mapsto&x\\
e&\mapsto&y\\
f&\mapsto&y,
\end{array}\right.&
t&:\left\{\begin{array}{rcl}
a&\mapsto&x\\
b&\mapsto&y\\
c&\mapsto&x\\
d&\mapsto&x,
\end{array}\right.&
\alpha&:\left\{\begin{array}{rcl}
c&\mapsto&1\\
d&\mapsto&2\\
g&\mapsto&3,
\end{array}\right.&
\beta&:\left\{\begin{array}{rcl}
e&\mapsto&3\\
f&\mapsto&1\\
g&\mapsto&2.
\end{array}\right.&
\end{align*}
This is graphically represented as follows:
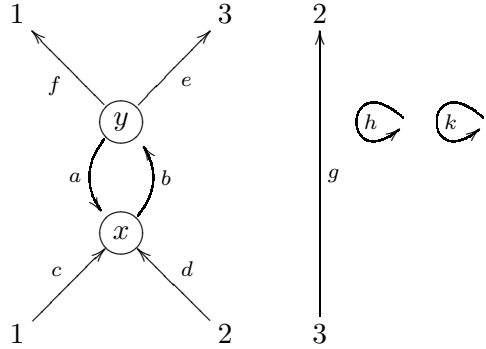
\[\xymatrix{1&&3&2&&\\
&\rond{y}\ar[ru]_e \ar[lu]^f \ar@/_1pc/[d]_a&&&\ar@(ul,dl)[]^h&\ar@(ul,dl)[]^k\\
&\rond{x}\ar@/_1pc/[u]_b&&&&\\
1\ar[ru]^c&&2\ar[lu]_d&3\ar[uuu]_g&&}\]
Note that this graph contains two loops, represented by $\xymatrix{&\ar@(ul,dl)[]^h}$
and $\xymatrix{&\ar@(ul,dl)[]^k}$.
\end{example}  

\begin{remark}
 As explained in the introduction, although loops play a passive role in the presentation of a ProP, their role  will be essential in the presentation of TraPs, see Section 
 \ref{subsection:generalised_graphs}.
\end{remark}

\begin{defi}
Let $G$ and $G'$ be two graphs. An \textbf{(resp. iso-)morphism} of graphs from $G$ to $G'$ is a family of (resp. bijections) maps $f=(f_V,f_E,f_I,f_O,f_{IO},f_L)$ with:
\begin{align*}
f_V&:V(G)\longrightarrow V(G'),&f_E&:E(G)\longrightarrow E(G'),&f_I&:I(G)\longrightarrow I(G'),\\
f_O&:O(G)\longrightarrow O(G'),&f_{IO}&:IO(G)\longrightarrow IO(G'),&f_L&:L(G)\longrightarrow L(G'),
\end{align*}
such that:
\begin{align*}
s'\circ f_E&=f_V\circ s_{\mid E(G)},&s'\circ f_O&=f_V\circ s_{\mid O(G)},\\
t'\circ f_E&=f_V\circ t_{\mid E(G)},&t'\circ f_I&=f_V\circ t_{\mid I(G)},\\
\alpha'\circ f_I&=\alpha_{\mid I(G)},&\alpha'\circ f_{IO}&=\alpha_{\mid IO(G)},\\
\beta'\circ f_O&=\beta_{\mid O(G)},&\beta'\circ f_{IO}&=\beta_{\mid IO(G)}.
\end{align*}
For any $k,l\in  \N_0$, we denote by $\Gr(k,l)$ the space generated by the isoclasses of graphs $G$ such that
$i(G)=k$ and $o(G)=l$, i.e. $\Gr(k,l)$ is the quotient space of graphs with $k$ input edges and $l$ output edges by the equivalence relation given by isomorphism.
\end{defi}
In what follows, we shall write \emph{graphs} for \emph{isoclasses of graphs}.
\begin{example}
The isomorphism class of the graph of Example \ref{ex4} is represented by:
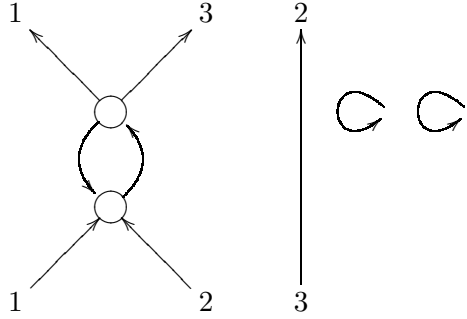
\[\xymatrix{1&&3&2&&\\
&\rond{}\ar[ru] \ar[lu] \ar@/_1pc/[d]&&&\ar@(ul,dl)[]&\ar@(ul,dl)[]\\
&\rond{}\ar@/_1pc/[u]&&&&\\
1\ar[ru]&&2\ar[lu]&3\ar[uuu]&&}\]
\end{example}
We now want  to equip the set $\Gr$ of isoclasses of graphs  with a ProP structure.
\begin{itemize}
	\item  We first define the \textbf{horizontal concatenation}. If $G$ and $G^\prime$ are two disjoint graphs,
	we define a graph $G*G'$ in the following way:
	\begin{align*}
	V(G*G')&=V(G)\sqcup V(G'),&E(G*G')&=E(G)\sqcup E(G'),&L(G*G')&=L(G)\sqcup L(G'),\\
	I(G*G')&=I(G)\sqcup I(G'),&O(G*G')&=O(G)\sqcup O(G'),&IO(G*G')&=IO(G)\sqcup IO(G').
	\end{align*}
	The source and target maps are given by:
	\begin{align*}
	s''_{\mid E(G)\sqcup O(G)}&=s,&s''_{\mid E(G')\sqcup O(G')}&=s',\\
	t''_{\mid E(G)\sqcup I(G)}&=t,&t''_{\mid E(G')\sqcup I(G')}&=t'.
	\end{align*}
	The indexations of the input and output edges are given by:
	\begin{align*}
	\alpha''_{\mid I(G)\sqcup IO(G)}&=\alpha,&\alpha''_{\mid I(G')\sqcup IO(G')}&=i(G)+\alpha',\\
	\beta''_{\mid O(G)\sqcup IO(G)}&=\beta,&\beta''_{\mid O(G')\sqcup IO(G')}&=o(G)+\beta'
	\end{align*}
	with an obvious abuse of notation in the definition of the second column.
	Notice that this product is not commutative in the usual sense for $G*G'$ and $G'*G$ might differ by the indexation of their input and output 
	edges. However, it is commutative in the sense of Axiom 6.(b) of ProPs.
	Roughly speaking, $G*G'$ is the disjoint union of $G$ and $G'$, the input and output edges of $G'$ being indexed
	after the input and output edges of $G$. 
	\begin{center}
		\begin{tikzpicture}[line cap=round,line join=round,>=triangle 45,x=0.5cm,y=0.5cm]
		\clip(-2.5,-4.) rectangle (1.,4.);
		\draw [line width=0.4pt] (-2.,1.)-- (0.5,1.);
		\draw [line width=0.4pt] (0.5,1.)-- (0.5,-1.);
		\draw [line width=0.4pt] (0.5,-1.)-- (-2.,-1.);
		\draw [line width=0.4pt] (-2.,-1.)-- (-2.,1.);
		\draw [->,line width=0.4pt] (-1.5,1.) -- (-1.5,3.);
		\draw [->,line width=0.4pt] (0.,1.) -- (0.,3.);
		\draw [->,line width=0.4pt] (-1.5,-3.) -- (-1.5,-1.);
		\draw [->,line width=0.4pt] (0.,-3.) -- (0.,-1.);
		\draw (-1.25,0.5) node[anchor=north west] {$G$};
		\draw (-1.8,-3) node[anchor=north west] {$1$};
		\draw (-0.3,-3) node[anchor=north west] {$k$};
		\draw (-1.4,-2.2) node[anchor=north west] {$\ldots$};
		\draw (-1.8,4.2) node[anchor=north west] {$1$};
		\draw (-0.3,4.2) node[anchor=north west] {$l$};
		\draw (-1.4,2.) node[anchor=north west] {$\ldots$};
		\end{tikzpicture}
		$\substack{\displaystyle *\\ \vspace{3cm}}$
		\begin{tikzpicture}[line cap=round,line join=round,>=triangle 45,x=0.5cm,y=0.5cm]
		\clip(-2.5,-4.) rectangle (0.7,4.);
		\draw [line width=0.4pt] (-2.,1.)-- (0.5,1.);
		\draw [line width=0.4pt] (0.5,1.)-- (0.5,-1.);
		\draw [line width=0.4pt] (0.5,-1.)-- (-2.,-1.);
		\draw [line width=0.4pt] (-2.,-1.)-- (-2.,1.);
		\draw [->,line width=0.4pt] (-1.5,1.) -- (-1.5,3.);
		\draw [->,line width=0.4pt] (0.,1.) -- (0.,3.);
		\draw [->,line width=0.4pt] (-1.5,-3.) -- (-1.5,-1.);
		\draw [->,line width=0.4pt] (0.,-3.) -- (0.,-1.);
		\draw (-1.25,0.5) node[anchor=north west] {$G'$};
		\draw (-1.8,-3) node[anchor=north west] {$1$};
		\draw (-0.3,-3) node[anchor=north west] {$k'$};
		\draw (-1.4,-2.2) node[anchor=north west] {$\ldots$};
		\draw (-1.8,4.2) node[anchor=north west] {$1$};
		\draw (-0.3,4.2) node[anchor=north west] {$l'$};
		\draw (-1.4,2.) node[anchor=north west] {$\ldots$};
		\end{tikzpicture}
		$\substack{\displaystyle =\\ \vspace{3cm}}$
		\begin{tikzpicture}[line cap=round,line join=round,>=triangle 45,x=0.5cm,y=0.5cm]
		\clip(-2.5,-4.) rectangle (0.5,4.);
		\draw [line width=0.4pt] (-2.,1.)-- (0.5,1.);
		\draw [line width=0.4pt] (0.5,1.)-- (0.5,-1.);
		\draw [line width=0.4pt] (0.5,-1.)-- (-2.,-1.);
		\draw [line width=0.4pt] (-2.,-1.)-- (-2.,1.);
		\draw [->,line width=0.4pt] (-1.5,1.) -- (-1.5,3.);
		\draw [->,line width=0.4pt] (0.,1.) -- (0.,3.);
		\draw [->,line width=0.4pt] (-1.5,-3.) -- (-1.5,-1.);
		\draw [->,line width=0.4pt] (0.,-3.) -- (0.,-1.);
		\draw (-1.25,0.5) node[anchor=north west] {$G$};
		\draw (-1.8,-3) node[anchor=north west] {$1$};
		\draw (-0.3,-3) node[anchor=north west] {$k$};
		\draw (-1.4,-2.2) node[anchor=north west] {$\ldots$};
		\draw (-1.8,4.2) node[anchor=north west] {$1$};
		\draw (-0.3,4.2) node[anchor=north west] {$l$};
		\draw (-1.4,2.) node[anchor=north west] {$\ldots$};
		\end{tikzpicture}
		\begin{tikzpicture}[line cap=round,line join=round,>=triangle 45,x=0.5cm,y=0.5cm]
		\clip(-2.5,-4.) rectangle (2.,4.);
		\draw [line width=0.4pt] (-2.,1.)-- (0.5,1.);
		\draw [line width=0.4pt] (0.5,1.)-- (0.5,-1.);
		\draw [line width=0.4pt] (0.5,-1.)-- (-2.,-1.);
		\draw [line width=0.4pt] (-2.,-1.)-- (-2.,1.);
		\draw [->,line width=0.4pt] (-1.5,1.) -- (-1.5,3.);
		\draw [->,line width=0.4pt] (0.,1.) -- (0.,3.);
		\draw [->,line width=0.4pt] (-1.5,-3.) -- (-1.5,-1.);
		\draw [->,line width=0.4pt] (0.,-3.) -- (0.,-1.);
		\draw (-1.25,0.5) node[anchor=north west] {$G'$};
		\draw (-2.3,-3) node[anchor=north west] {$k+1$};
		\draw (-0.3,-3) node[anchor=north west] {$k+k'$};
		\draw (-1.4,-2.2) node[anchor=north west] {$\ldots$};
		\draw (-2.3,4.2) node[anchor=north west] {$l+1$};
		\draw (-0.3,4.2) node[anchor=north west] {$l+l'$};
		\draw (-1.4,2.) node[anchor=north west] {$\ldots$};
		\end{tikzpicture}
		
		\vspace{-1.5cm}
	\end{center}
	
	\begin{example} Here is an example of horizontal concatenation :\\
	
	
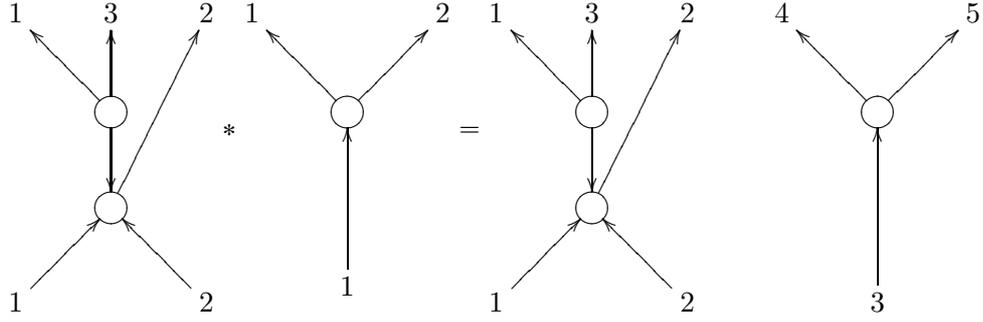
\vspace{-2cm}
	
			\[\xymatrix{1&3&2\\
			&\rond{}\ar[lu] \ar[u]\ar[d]&\\
			&\rond{}\ar[ruu]&\\
			1\ar[ru]&&2\ar[lu]} 
			\substack{\vspace{3cm}\\\displaystyle \mbox{$*$}}
		\xymatrix{1&&2\\
			&\rond{}\ar[ru] \ar[lu]&\\
			&&\\
			&1\ar[uu]&}
					\substack{\vspace{3cm}\\\displaystyle =}
					\xymatrix{1&3&2&4&&5\\
			&\rond{}\ar[lu] \ar[u]\ar[d]&&&\rond{}\ar[ru] \ar[lu]&\\
			&\rond{}\ar[ruu]&&&&\\
			1\ar[ru]&&2\ar[lu]&&3\ar[uu]&}\]
	\end{example}
	
	This product of graphs induces a product $*:\Gr(k,l)\otimes \Gr(k',l')\longrightarrow
	\Gr(k+k',l+l')$. If $G$, $G'$ and $G''$ are three graphs, clearly
	\[G*(G'*G'')=(G*G')*G''.\]
	Hence, the product $*$ is associative. Its unit $I_0$ is the unique graph such that
	$V(I_0)=E(I_0)=I(I_0)=O(I_0)=IO(I_0)=\emptyset$. 
	
	\item We now define the \textbf{ vertical concatenation}. Let $G$ and $G'$ be {disjoint} graphs such that $o(G)=i(G')$. We define a graph 
	$G''=G'\circ G$
	in the following way:
	\begin{align*}
	V(G'')&=V(G)\sqcup V(G'),\\
	E(G'')&=E(G)\sqcup E(G')\sqcup \{(f,e)\in O(G)\times I(G'):\beta(f)=\alpha'(e)\},\\
	I(G'')&=I(G)\sqcup \{(f,e)\in IO(G)\times I(G'):\beta(f)=\alpha'(e)\},\\
	O(G'')&=O(G)\sqcup \{(f,e)\in O(G)\times IO(G') : \beta(f)=\alpha'(e)\},\\
	IO(G'')&=\{(f,e)\in IO(G)\times IO(G'): \beta(f)=\alpha'(e)\},\\
	L(G'')&=L(G)\sqcup L(G').
	\end{align*}	
	Its \textbf{source} and \textbf{target} maps  are given by:
	\begin{align*}
	s''_{\mid E(G)}&=s_{\mid E(G)},&s''_{\mid E(G')}&=s'_{\mid E(G')},&
	s''_{\mid O(G')}&=s'_{\mid O(G')},&s''((f,e))&=s(f),\\
	t''_{\mid E(G)}&=t_{\mid E(G)},&t''_{\mid E(G')}&=s'_{\mid E(G')},&
	t''_{\mid I(G)}&=s_{\mid I(G)},&s''((f,e))&=t'(e).
	\end{align*}
	The indexations of its input and output edges are given by:
	\begin{align*}
	\alpha''_{\mid I(G)}&=\alpha_{\mid I(G)},& \alpha''((f,e))&=\alpha(f),\\
	\beta''_{\mid O(G')}&=\beta'_{\mid O(G')},&\beta''((f,e))&=\beta'(e).
	\end{align*}
	Roughly speaking, $G'\circ G$ is obtained by gluing together the outgoing edges of $G$ and the incoming
	edges of $G'$ according to their indexation. 
		
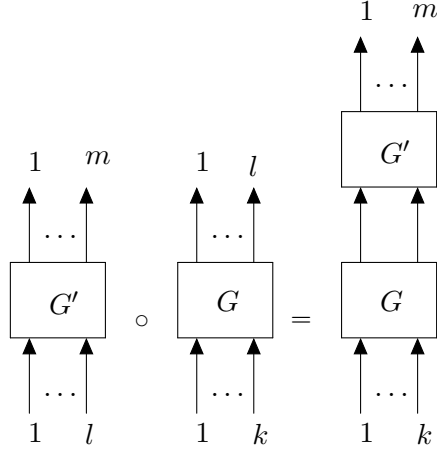
\begin{center}
		\begin{tikzpicture}[line cap=round,line join=round,>=triangle 45,x=0.5cm,y=0.5cm]
		\clip(-2.5,-4.) rectangle (1.,4.);
		\draw [line width=0.4pt] (-2.,1.)-- (0.5,1.);
		\draw [line width=0.4pt] (0.5,1.)-- (0.5,-1.);
		\draw [line width=0.4pt] (0.5,-1.)-- (-2.,-1.);
		\draw [line width=0.4pt] (-2.,-1.)-- (-2.,1.);
		\draw [->,line width=0.4pt] (-1.5,1.) -- (-1.5,3.);
		\draw [->,line width=0.4pt] (0.,1.) -- (0.,3.);
		\draw [->,line width=0.4pt] (-1.5,-3.) -- (-1.5,-1.);
		\draw [->,line width=0.4pt] (0.,-3.) -- (0.,-1.);
		\draw (-1.25,0.5) node[anchor=north west] {$G'$};
		\draw (-1.8,-3) node[anchor=north west] {$1$};
		\draw (-0.3,-3) node[anchor=north west] {$l$};
		\draw (-1.4,-2.2) node[anchor=north west] {$\ldots$};
		\draw (-1.8,4.2) node[anchor=north west] {$1$};
		\draw (-0.3,4.2) node[anchor=north west] {$m$};
		\draw (-1.4,2.) node[anchor=north west] {$\ldots$};
		\end{tikzpicture}
		$\substack{\displaystyle \circ\\ \vspace{3cm}}$		
		\begin{tikzpicture}[line cap=round,line join=round,>=triangle 45,x=0.5cm,y=0.5cm]
		\clip(-2.5,-4.) rectangle (0.7,4.);
		\draw [line width=0.4pt] (-2.,1.)-- (0.5,1.);
		\draw [line width=0.4pt] (0.5,1.)-- (0.5,-1.);
		\draw [line width=0.4pt] (0.5,-1.)-- (-2.,-1.);
		\draw [line width=0.4pt] (-2.,-1.)-- (-2.,1.);
		\draw [->,line width=0.4pt] (-1.5,1.) -- (-1.5,3.);
		\draw [->,line width=0.4pt] (0.,1.) -- (0.,3.);
		\draw [->,line width=0.4pt] (-1.5,-3.) -- (-1.5,-1.);
		\draw [->,line width=0.4pt] (0.,-3.) -- (0.,-1.);
		\draw (-1.25,0.5) node[anchor=north west] {$G$};
		\draw (-1.8,-3) node[anchor=north west] {$1$};
		\draw (-0.3,-3) node[anchor=north west] {$k$};
		\draw (-1.4,-2.2) node[anchor=north west] {$\ldots$};
		\draw (-1.8,4.2) node[anchor=north west] {$1$};
		\draw (-0.4,4.1) node[anchor=north west] {$l$};
		\draw (-1.4,2.) node[anchor=north west] {$\ldots$};
		\end{tikzpicture}
		$\substack{\displaystyle =\\ \vspace{3cm}}$
		\begin{tikzpicture}[line cap=round,line join=round,>=triangle 45,x=0.5cm,y=0.5cm]
		\clip(-2.5,-4.) rectangle (0.5,8.);
		\draw [line width=0.4pt] (-2.,1.)-- (0.5,1.);
		\draw [line width=0.4pt] (0.5,1.)-- (0.5,-1.);
		\draw [line width=0.4pt] (0.5,-1.)-- (-2.,-1.);
		\draw [line width=0.4pt] (-2.,-1.)-- (-2.,1.);
		\draw [->,line width=0.4pt] (-1.5,1.) -- (-1.5,3.);
		\draw [->,line width=0.4pt] (0.,1.) -- (0.,3.);
		\draw [->,line width=0.4pt] (-1.5,-3.) -- (-1.5,-1.);
		\draw [->,line width=0.4pt] (0.,-3.) -- (0.,-1.);
		\draw (-1.25,0.5) node[anchor=north west] {$G$};
		\draw (-1.8,-3) node[anchor=north west] {$1$};
		\draw (-0.3,-3) node[anchor=north west] {$k$};
		\draw (-1.4,-2.2) node[anchor=north west] {$\ldots$};
		\draw [line width=0.4pt] (-2.,5.)-- (0.5,5.);
		\draw [line width=0.4pt] (0.5,5.)-- (0.5,3.);
		\draw [line width=0.4pt] (0.5,3.)-- (-2.,3.);
		\draw [line width=0.4pt] (-2.,3.)-- (-2.,5.);
		\draw [->,line width=0.4pt] (-1.5,5.) -- (-1.5,7.);
		\draw [->,line width=0.4pt] (0.,5.) -- (0.,7.);
		\draw (-1.25,4.5) node[anchor=north west] {$G'$};
		\draw (-1.8,8.2) node[anchor=north west] {$1$};
		\draw (-0.4,8.1) node[anchor=north west] {$m$};
		\draw (-1.4,6.) node[anchor=north west] {$\ldots$};
		\end{tikzpicture}
		
		\vspace{-1.5cm}
	\end{center}
	\begin{example} Here is an example of vertical concatenation :\\
	
	
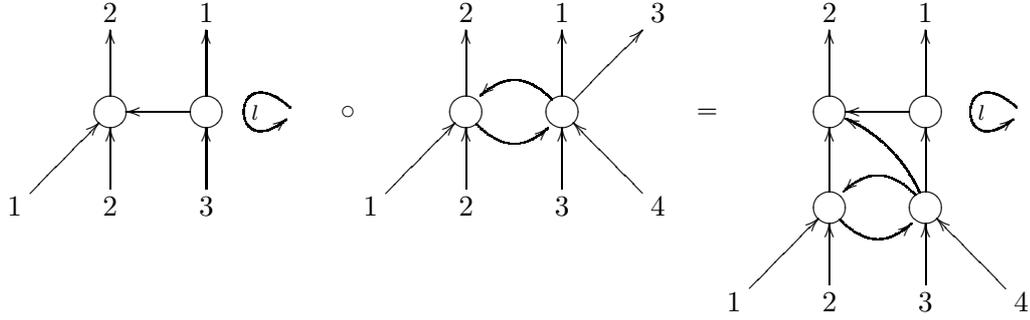
\vspace{-1.8cm}
		\[
		\xymatrix{&2&1& \\ 
			&\rond{}\ar[u]&\rond{}\ar[l] \ar[u] &\ar@(ul,dl)[]^l\\
			1\ar[ru]&2\ar[u]&3\ar[u]&}\hspace{5mm}
		\substack{\vspace{2.5cm}\\ \displaystyle \circ}
		\xymatrix{&2&1&3\\
			&\rond{}\ar[u]\ar@/_1pc/[r]&\rond{}\ar[u]\ar[ru]\ar@/_1pc/[l]&\\
			1\ar[ru]&2\ar[u]&3\ar[u]&4\ar[lu]} 
						\substack{\vspace{2.5cm}\\ \hspace{.3cm}\displaystyle=}
		\xymatrix{&2&1&\\
			&\rond{}\ar[u]&\rond{}\ar[u]\ar[l]&\ar@(ul,dl)[]^l\\
			&\rond{}\ar[u]\ar@/_1pc/[r]&\rond{}\ar@/_.5pc/[lu]\ar[u]\ar@/_1pc/[l]&\\
			1\ar[ru]&2\ar[u]&3\ar[u]&4\ar[lu]}\]
	\end{example}
	\end{itemize}
	\begin{theo} \label{theo:ProP_graph}
	The family $\Gr=(\Gr_{k,l})_{k,l\in  \N_0}$, equipped with this $\sym\times \sym^{op}$-action and these horizontal and vertical concatenations,
	is a ProP.
\end{theo}
\begin{proof}
\begin{itemize}
	\item We check the \textbf{associativity} of $\circ$. Let $G$, $G'$ and $G''$ be three graphs with $o(G)=i(G')$ and $o(G')=i(G'')$. 
	The graphs $(G''\circ G')\circ G$ and $G''\circ (G'\circ G)$ may be different, but both are isomorphic to the graph $H$
	defined by:
	\begin{align*}
	V(H)&=V(G)\sqcup V(G')\sqcup V(G''),\\
	E(H)&=E(G)\sqcup E(G')\sqcup E(G'')\\
	&\sqcup \{(f,e)\in O(G)\times I(G'): \beta(f)=\alpha'(e)\}
	\sqcup \{(f,e)\in O(G')\times I(G''): \beta'(f)=\alpha''(e)\}\\
	&\sqcup \{(f,f',e)\in O(G)\times IO(G')\times I(G''): \beta(f)=\alpha'(f'),\beta'(f')=\alpha''(e)\},\\
	I(H)&=I(G)\sqcup \{(f,e)\in IO(G)\times I(G'): \beta(f)=\alpha'(e)\}\\
	&\sqcup \{(f,f',e)\in IO(G)\times IO(G')\times I(G''): \beta(f)=\alpha'(f'),\beta'(f')=\alpha''(e)\},\\
	O(H)&=O(G'')\sqcup \{(f,e)\in O(G')\times IO(G''): \beta'(f)=\alpha''(e)\}\\
	&\sqcup \{(f,f',e)\in O(G)\times IO(G')\times IO(G''): \beta(f)=\alpha'(f'),\beta'(f')=\alpha''(e)\},\\
	IO(H)&=\{(f,f',e)\in IO(G)\times IO(G')\times IO(G''): \beta(f)=\alpha'(f'),\beta'(f')=\alpha''(e)\},\\
	L(H)&=L(G)\sqcup L(G')\sqcup L(G''),
	\end{align*}
	with immediate  source, target  and indexation maps. So $\circ$ induces an associative product  $\circ: \Gr(l,m)\otimes\Gr(k,l)\longrightarrow\Gr(k,m)$.  
	
	\item Let $I_1$ be the graph such that
	\begin{align*}
	V(I_1)&=E(I_1)=I(I_1)=O(I_1)=L(I_1)=\emptyset,&
	IO(I_1)&=[1].
	\end{align*}
	We show that $I_1$ is the \textbf{unit} for $\circ$:  
	The indexation maps are both the identity of $[1]$. 
	For any integer $n\in  \N_0$, $I_1^{*n}$ is isomorphic to the graph $I_n$ such that
	\begin{align*}
	V(I_n)&=E(I_n)=I(I_n)=O(I_n)=L(I_n)=\emptyset,&
	IO(I_n)&=[n],
	\end{align*}
	the indexation maps being both the identity of $[n]$. If $G$ is a graph and $k=i(G)$, then $H=G\circ I_k$
	is the graph such that:
	\begin{align*}
	V(H)&=V(G),&I(H)&=\{(\alpha(e),e): e\in I(G)\},\\
	E(H)&=E(G),&IO(H)&=\{(\alpha(e),e):e\in IO(G)\},\\
	O(H)&=O(G),&L(H)&=L(G),
	\end{align*}
	with immediate source, target  and indexation maps. This graph $H$ is isomorphic to $G$, via the isomorphism given by:
	\begin{align*}
	f_V&=\mathrm{Id}_{V(G)},&f_I((\alpha(e),e))&=e,\\
	f_E&=\mathrm{Id}_{E(G)},&f_{IO}((\alpha(e),e))&=e,\\
	f_O&=\mathrm{Id}_{O(G)},\\
	f_L&=\mathrm{Id}_{L(G)}.
	\end{align*}
	Similarly, $I_l\circ G$ and $G$ are isomorphic. Hence, $I_1$ is the unit of $\circ$ in $\Gr$. 
	\item We check the \textbf{compatibility} of the horizontal and vertical concatenations.
	Let $G$, $G'$, $H$ and $H'$ be graphs such that $o(G)=i(H)$ and $o(G')=i(H')$. The graphs
	$(H*H')\circ (G*G')$ and $(H\circ G)* (H'\circ G')$ are both equal to the graph $K$, such that:
	\begin{align*}
	V(K)&=V(G)\sqcup V(G')\sqcup V(H)\sqcup V(H'),\\
	E(K)&=E(G)\sqcup E(G')\sqcup E(H)\sqcup E(H')\\
	&\sqcup \{(f,e)\in O(G)\times I(H): \beta(f)=\alpha'(e)\}\\
	&\sqcup \{(f,e)\in O(G')\times I(H'); \beta(f)=\alpha'(e)\},\\
	I(K)&=I(G)\sqcup I(G')\sqcup \{(f,e)\in IO(G)\times I(H); \beta(f)=\alpha'(e)\}\\
	&\sqcup \{(f,e)\in IO(G')\times I(H'); \beta(f)=\alpha'(e)\},\\
	O(K)&=O(H)\sqcup O(H')\sqcup \{(f,e)\in O(G)\times IO(H); \beta(f)=\alpha'(e)\}\\
	&\sqcup \{(f,e)\in O(G')\times IO(H'); \beta(f)=\alpha'(e)\},\\
	IO(K)&=\sqcup \{(f,e)\in IO(G)\times IO(H); \beta(f)=\alpha'(e)\}\\
	&\sqcup \{(f,e)\in IO(G')\times IO(H'); \beta(f)=\alpha'(e)\},\\
	L(K)&\l=L(G)\sqcup L(G')\sqcup L(H)\sqcup L(H'),
	\end{align*}
	with obvious source, target  and indexation maps. Hence, the vertical and the horizontal concatenations are compatible. 
	\item We check the \textbf{module} structure of $\Gr$ over the symmetric group.
	Let $G$ be a graph, $\sigma\in \sym_{o(G)}$ and $\tau\in \sym_{i(G)}$. We set:
	\begin{align}\label{eqGrmod}
	\sigma\cdot G&=(V(G),E(G),I(G),O(G),IO(G),L(G),s,t,\alpha,\sigma \circ\beta),\nonumber\\
	G\cdot \tau&=(V(G),E(G),I(G),O(G),IO(G),L(G),s,t,\tau^{-1}\circ\alpha,\beta).
	\end{align} 
	This induces a structure of $\sym\times \sym^{op}$-module over $\Gr$.
	
	Let us prove the compatibility of this action with the vertical concatenation. Let $G$ and $G'$ be two graphs
	such that $o(G)=i(G')$, and let $\sigma \in \sym_{o(G')}$, $\tau\in \sym_{o(G)}$,
	$\nu\in \sym_{i(G)}$. Clearly, the graphs $\sigma\cdot(G'\circ G)$ and $(\sigma\cdot G')\circ G$ are equal;
	the graphs $(G'\circ G)\cdot \nu$ and $G'\circ (G\cdot \nu)$ are equal. Let us compare the graphs
	$H=(G'\cdot \tau)\circ G$ and $H'=G'\circ (\tau \cdot G)$. Their set of vertices coincide. Moreover:
	\begin{align*}
	E(H)&=E(G)\sqcup E(G')\sqcup\{(f,e)\in O(G)\times I(G'): \beta(f)=\tau^{-1}\circ \alpha'(e)\},\\
	E(H')&=E(G)\sqcup E(G')\sqcup\{(f,e)\in O(G)\times I(G'): \tau\circ \beta(f)=\alpha'(e)\},
	\end{align*}
	so $E(H)=E(H')$. Similarly, $I(H)=I(H')$, $O(H)=O(H')$,  $IO(H)=IO(H')$ and $L(H)=L(H')$.  
	Moreover, the source, target and indexation maps are the same for $H$ and $H'$, so $H=H'$.
	\item 
	We finally prove the \textbf{compatibility} of the $\sym\times \sym^{op}$-action with the horizontal composition.
	Let $G$ and $G'$ be two graphs, $\sigma \in \sym_{o(G)}$ and $\sigma'\in \sym_{o(G')}$. We put
	$H=(\sigma\cdot G)*(\sigma'\cdot G')$ and $H'=(\sigma \otimes \sigma')\cdot (G*G')$. 
	They have the same set of vertices, whether internal, input, output and input-output edges, and the source,
	target and indexation of output edges maps   for $H$ and $H'$ coincide.
	Both indexations of the set of output edges are   given by:
	\[\sigma''(e)=\begin{cases}
	\sigma \circ \beta(e)\mbox{ if }e\in O(G)\sqcup IO(G),\\
	o(G)+\sigma' \circ \beta'(e)\mbox{ if }e\in O(G')\sqcup IO(G').
	\end{cases}\]
	So $H=H'$.
	
	Let $G$ and $G'$ be graphs. We set
	$H=c_{o(G),o(G')}\cdot (G*G')$ and $H'=(G'*G)\cdot c_{i(G),i(G')}$, where $c_{m, n}\in \sym_{m+n}$  was defined in (\ref{defcmn}). They have the same sets of vertices, internal,
	input, output and input-output edges, and the same source and target maps. The indexations maps are given by:
	\begin{align*}
	\alpha_H(e)&=\begin{cases}
	\alpha(e)+i(G')\mbox{ if }e\in I(G)\sqcup IO(G),\\
	\alpha'(e)\mbox{ if }e\in I(G')\sqcup IO(G'),
	\end{cases}\\
	\beta_H(e)&=\begin{cases}
	\beta(e)\mbox{ if }e\in O(G)\sqcup IO(G),\\
	\beta'(e)+o(G)\mbox{ if }e\in O(G')\sqcup IO(G'),
	\end{cases}\\ \\
	\alpha_{H'}(e)&=\begin{cases}
	\alpha'(e)\mbox{ if }e\in I(G')\sqcup IO(G'),\\
	\alpha(e)+i(G')\mbox{ if }e\in I(G)\sqcup IO(G),
	\end{cases}\\
	\beta_{H'}(e)&=\begin{cases}
	\beta'(e)+o(G)\mbox{ if }e\in O(G')\sqcup IO(G'),\\
	\beta(e)\mbox{ if }e\in O(G)\sqcup IO(G),
	\end{cases}
	\end{align*}
	so $H=H'$. \qedhere
\end{itemize}
 
\end{proof}

\section{The ProP of continuous morphisms: $\Hom_V^c$} \label{sec:propHom}

We now generalise the ProP $\Hom_V$ of Subsection \ref{sec:propHom_finiteDim} to a ProP
$\Hom^c_V$ (the superscript "c" for continuous) for a topological vector space $V$.

We   work in the context of nuclear Fr\'echet spaces. One could relax these conditions (for example Fr\'echet could be replaced by 
barreled) yet the nuclear setup is comfortable to work in and general enough for our purposes. We refer the reader to \cite{Treves67} for the  
more general cases.

\subsection{Fr\'echet nuclear spaces} \label{subsection:Frechet_nuc} 

Nuclear spaces were defined in the seminal work \cite{Gr54}. Most of the results stated here can  be found in \cite{Gr52,Gr54}. We also 
refer to the more recent presentation \cite{Treves67}.

We recall that \begin{itemize}
                \item A topological vector space is \textbf{Fr\'echet} if it is Hausdorff, has its topology induced by a family of semi-norms and is complete with respect 
  to this family of semi-norms. 
  \item A topological vector space is called \textbf{reflexive} if $E''=(E')'=E$, where $E'$ is the topological dual of $E$.
               \end{itemize}
               In the following $E$ and $F$ are two topological vector spaces and $\Hom^c(E,F)$ is the set of continuous  linear maps from $E$ to $F$.
               \begin{remark}
                When $E$ and $F$ are finite dimensional, we have $\Hom^c(E,F)$=$\Hom(E,F)$.
               \end{remark}
In order to build the $\Hom$  ProP in the infinite dimensional case, we need Grothendieck's   completion of the tensor product, a notion we recall 
here in the setup of locally convex topological $\K$-vector spaces.

Let $E$ and $F$ be two vector spaces. Recall that there exists a   vector space $E\otimes F$, and a bilinear map 
$\phi:E\times F\longrightarrow E\otimes F$ such that for any vector space $V$ and bilinear map $f:E\times F\longrightarrow V$, there is a unique 
linear map $\tilde f:E\otimes F\to V$ satisfying  $f=\tilde f\circ \phi$. The space $E\otimes F$ is unique modulo isomorphism and is called the \textbf{tensor product} of 
$E$ and $F$.

Given two     topological vector spaces, $E$ and $F$ one can a priori equip   $E\otimes F$ with several topologies, among which the 
\textbf{$\epsilon$-topology} and the \textbf{projective topology}  whose construction are recalled in Appendix \ref{section:topologies_tens_prod}. We 
denote by $E\otimes_\epsilon F$ (resp. $E\otimes_\pi F$) the space $E\otimes F$ endowed with 
the $\epsilon$-topology (resp. the projective topology)  and by
  $E\widehat\otimes _\epsilon F$ (resp. $E\widehat\otimes _\epsilon F$) of $E\otimes_\epsilon F$ (resp. $E\otimes_\epsilon F$) their completion with respect to the 
$\epsilon$-topology (resp. projective topology). These two spaces differ in general but coincide for nuclear spaces.
\begin{defi} \cite{Gr54}
 A locally convex topological vector space $E$ is \textbf{nuclear} if, and only if, for any locally convex topological vector space $F$,
 \begin{equation*}
 E\widehat\otimes _\epsilon F = E\widehat\otimes _\pi F =: E\widehat\otimes  F
\end{equation*}
holds, in which case $E\widehat\otimes  F $ is called the \textbf{completed tensor product} of $E$ and $F$.
\end{defi}
There are other equivalent definitions of nuclearity, see for example \cite{gelfand1964,hida2008}.

Given a locally convex topological vector space $E$, its topological dual $E'$ can be endowed with various topologies. An important one for our applications will be 
the \textbf{strong topology}, generated by the family of semi-norms of $E'$ defined, on any $f\in E'$: $||f||_B:=\sup_{x\in B}|f(x)|$
for any bounded set $B$ of $E$. The topological dual $E'$ endowed with this topology is called the \textbf{strong dual}.

For Fr\'echet spaces, nuclearity is preserved under strong duality.
\begin{prop} \begin{itemize}
              \item 
\cite[Proposition 50.6]{Treves67} \label{prop:Frechet_nuclear}
 A Fr\'echet space is nuclear if and only if its strong dual is nuclear.
 \item 
 \cite[Proposition 36.5]{Treves67} A Fr\'echet nuclear space is reflexive.
             \end{itemize}
\end{prop}
Many spaces relevant to renormalisation issues are Fr\'echet and nuclear. We list here some examples.
\begin{example}\label{ex:findimtensor1}  
 Any finite dimensional vector space  can be equipped with a norm and for any of these norms, they are trivially  Banach, hence Fr\'echet and nuclear. 
 If $E$ and $F$ are finite dimensional vector spaces  we have
 $\Hom^c (E, F)=  \Hom (E, F)\simeq  E^* \otimes F$, where $  \Hom (E, F)$ stands for the space of $F$-valued linear maps on 
 $E$ and where the dual $E^*$ is the { \bf algebraic dual}.
\end{example} 
\begin{example}\label{ex:infindimtensor1} 
Let $U$  be an open subset of $\R^n$.
Take $E= C^\infty(U)=:{\mathcal E}(U)$. The topological dual 
 is the space ${E'}={\mathcal E}^\prime(U)$ of distributions 
 on $U$ with compact support. 
 
 Then $E$ is Fr\'echet (\cite{Treves67}, pp. 86-89), and $E'$ is nuclear (\cite{Treves67}, Corollary p. 530). By Proposition 
 \ref{prop:Frechet_nuclear}, $E$ is also nuclear. 
\end{example}
\begin{remark}\label{rk:dualnotfrechet}  Note that the dual $E' $ of a Fr\'echet space  $E$ is  never a Fr\'echet space (for any of the natural topologies on $E'$), unless $E$  is actually a Banach space (see for example 
\cite{kothe1969}).   In particular, ${\mathcal E}^\prime(U)$ is generally not Fr\'echet.
\end{remark} 
We now sum up various results of \cite{Treves67} of importance for  later purposes.
\begin{theo}\cite[Equations (50.17)--(50.19)]{Treves67}
 Let $E$ and $F$ be two Fr\'echet spaces, with $E$ nuclear. The following isomorphisms of topological vector spaces hold.
 \begin{align}
  & E'\widehat\otimes  F \simeq \Hom^c(E,F) \label{eq:E_prime_otimes_F} \\
  & E\widehat\otimes  F \simeq \Hom^c(E',F) \label{eq:E_otimes_F} \\
  & E'\widehat\otimes  F' \simeq (E\widehat\otimes  F)' \simeq {\mathcal B}^c(E\times F, \K). \label{eq:E_prime_otimes_F_prime}
 \end{align}
 with ${\mathcal B}^c(E\times F, \K)$ the set of continuous bilinear maps 
 $K:E\times F\longrightarrow\K$. Here the duals are endowed with the strong dual topology, 
 $\Hom^c(E,F)$ with the strong topology and ${\mathcal B}^c(E\times F, \K)$ with the topology of uniform convergence on products of bounded sets.
\end{theo}
We  also need the stability of Fr\'echet nuclear spaces under completed tensor products, for which we need the following   lemma.
\begin{lemma} \label{lem:prod_Frechet_nuclear}
 The completed tensor product  $E\widehat\otimes  F$ of two Fr\'echet nuclear spaces is a Fr\'echet nuclear space.
\end{lemma}
\begin{proof}
 If $E$ and $F$ are two nuclear spaces then $E\widehat\otimes  F$ is a nuclear space (\cite[Equation (50.9)]{Treves67}). It is moreover complete since 
 $E\widehat\otimes  F$ is obtained by completion.
\end{proof}
\begin{prop}
 Let $V$ be a Fr\'echet nuclear space. Then 
 \begin{equation} \label{eq:echange_dual_prod}
 \left(V^{ \widehat\otimes  k}\right)' \simeq\left(V'\right)^{ \widehat\otimes  k}
 \end{equation} 
 holds for any $k\geq1$, where the duals are endowed with their strong topologies.
\end{prop}
\begin{proof}
 Let $V$ be a Fr\'echet nuclear space. The case $k=1$ is trivial. Then Equation \eqref{eq:echange_dual_prod} with $k=2$ holds by Equation \eqref{eq:E_prime_otimes_F_prime} 
 with $E=F=V$. The cases $k\geq2$ are proved by induction, using $E=V^{\widehat\otimes  k-1}$ and $F=V$. The induction holds by Lemma 
 \ref{lem:prod_Frechet_nuclear}.
\end{proof}

\subsection{A ProP for Fr\'echet nuclear spaces} \label{subsection:infinite_dim_prop}

We start by recalling the definition of distributions over a    finite dimensional smooth manifold ${X}$. We quote \cite[Definition 6.3.3]{Ho89}.
 \begin{defi} 
 To every coordinate system 
  $\kappa:U_k\subset {X}\longrightarrow V_k\subset\R^n$ we associate 
  a distribution $u_k\in\mathcal{D}'(V _k)$ such that 
  \begin{equation*}
   u_{k'}=(\kappa\circ\kappa'^{-1})^*u_k
  \end{equation*}
  in $\kappa'(U_k\cap U_{k'})$; with $(\kappa\circ\kappa'^{-1})^*u_k$ the pullback of 
  $u_k$ by $\kappa\circ\kappa'^{-1}$ whose existence and uniqueness
  is given by 
  \cite[Theorem 6.1.2]{Ho89}. Then the system $u_k$ of distributions is called a distribution on ${X}$. The set of distributions on ${X}$ is written 
  $\mathcal{D}'({X})$. Similarly we define $\mathcal{E}'({X})$, the set of distributions with 
  compact support.
 \end{defi}
\begin{prop} \label{prop:fction_manifold_Frechet_nuclear}
 $\mathcal{E}({X})$ is a Fr\'echet nuclear space.
\end{prop}
It is a classical result of functional analysis that the space of functions over a smooth manifold is Fr\'echet (see for example 
\cite[Exercise 2.3.2]{BaCr13}). The fact that the same space is nuclear is a folklore result, often stated without proof nor references. A proof was 
recently given in \cite[p. 4]{BrDaLGRe17}. 

It then follows from Proposition \ref{prop:Frechet_nuclear}, that the space $\mathcal{E}'(X)$ is   also nuclear. 
\begin{remark} (Compare with Remark \ref{rk:dualnotfrechet}). Note that the 
space $\mathcal{E}'(X)$ is \emph{not} Fr\'echet  since the dual of a Fr\'echet space $F$ is Fr\'echet if and only if $F$ is Banach (see for example 
\cite{kothe1969}) which is not the 
case of $\mathcal{E}({X})$.
\end{remark}
One further useful result is
\begin{prop} \label{prop:prod_function}
 Let ${X}$ and ${Y}$ be two finite dimensional smooth manifolds. Then 
 \begin{equation*}
 \Hom^c(\mathcal{E}'({X}),\mathcal{E}({Y}))\simeq \,  \mathcal{E}({X})\,\widehat\otimes\, \mathcal{E}({Y}) \simeq \mathcal{E}({X}\times {Y})
 \end{equation*}
 holds.
\end{prop}
The second isomorphism   \cite[Chap. 5, p. 105]{Gr52} can be proved using a version of the Schwartz kernel theorem for smoothing operators 
\cite[Theorem 2.4.5]{BaCr13} by means of the identification $\Hom^c(\mathcal{E}'({X}),\mathcal{E}({Y}))\simeq \mathcal{E}({X}\times {Y})$. The result then follows from 
\eqref{eq:E_otimes_F} applied to $\mathcal{E}(X)$ and $\mathcal{E}(Y)$ which are Fr\'echet nuclear spaces.
\begin{defi} \label{defi:Hom_V_generalised}
 Let $V$ be a Fr\'echet nuclear space. For any $k,l\in  \N_0$, we set
\[\Hom_V^c(k,l)=\Hom^c(V^{\hat \otimes k}, V^{\hat \otimes l})\simeq(V')^{\widehat\otimes  k}\widehat\otimes  V^{\widehat\otimes  l},\]
where, as before $V'$ stands for the strong topological dual. Furthermore we set $\Hom^c_V:=(\Hom_V^c(k,l))_{k,l\geq0}$.

For any $\sigma \in \sym_n$, let  $\theta_\sigma$ be the endomorphism of $V^{\otimes n}$ defined by
\[\theta_\sigma(v_1\otimes \ldots \otimes v_n)=v_{\sigma^{-1}(1)}\otimes \ldots \otimes v_{\sigma^{-1}(n)}.\] It extends to a continuous linear map 
$\overline{\theta_\sigma}$ on the closure 
$V^{\widehat\otimes  n}$.   
For any $f\in \Hom_V^c (k,l)$, $\sigma \in \sym_l$, $\tau\in \sym_k$, we set:
\begin{align*}
\sigma\cdot f&=\overline{\theta_\sigma} \circ f ,&f\cdot \tau&=f\circ \overline{\theta_\tau}.
\end{align*} 
\end{defi}
In the above definition, the superscript ``c'' stands for continuous. The family $\Hom_V^c$ carries a ProP structure.
\begin{theo} \label{thm:Hom_V_generalised}
 Let $V$ be a Fr\'echet nuclear space. $\Hom_V^c$, with the action of $\sym\times\sym^\mathrm{op}$ described above, is a ProP. Its horizontal 
 concatenation is the usual (topological) tensor product of maps with $I_0:\K\longrightarrow\K$ is the constant map $I_0(x):=1_\K$  and 
 its vertical concatenation is the usual composition of maps and $I_1:V\longrightarrow V$ is the identity map.
\end{theo}
 \begin{proof}
  The proof is exactly the same as the proof of Definition-Proposition \ref{defi:Hom_V}.
 \end{proof}
 \begin{example}
      For a  finite dimensional vector space $V$ the classical ProP $\Hom_V$ of Proposition-Definition 
      \ref{defi:Hom_V} coincides with the the ProP $\Hom_V^c$.
     \end{example}
     \begin{example}
      Let $U$ be an open of $\R^n$. From Example \ref{ex:infindimtensor1} and Equation 
      \eqref{eq:echange_dual_prod}
the family  $({\mathcal K}_U(k, l))_{k,l\geq 0}$, with 
${\mathcal K}_U(k, l)= \left({\mathcal E}^\prime(U)\right)^{\widehat\otimes  k}\, \widehat\otimes  \,  
\left({\mathcal E}(U)\right)^{\widehat\otimes  l}$  
defines a ProP.
     \end{example}
     
     \begin{example}
     Let $X$ be a smooth finite dimensional manifold.
      From Proposition \ref{prop:fction_manifold_Frechet_nuclear} and Equation \eqref{eq:echange_dual_prod}
the family  $({\mathcal K}_X(k, l))_{k,l\geq 0}$,
with ${\mathcal K}_X(k, l)= \left({\mathcal E}^\prime(X)\right)^{\widehat\otimes  k}\, \widehat\otimes  \,  {\mathcal E}(X)^{\widehat\otimes  l}$  
defines a ProP.
     \end{example}

\section{Freeness of the ProP $\Gr$ of graphs} \label{sec:free_prop}

The goal of this section is to build free ProPs generated by indecomposable graphs (see Definition \ref{def:indecomposable} below).   
A free 
ProP was already described  by Hackney and Robertson in \cite{HaRo12}. Their construction is on the category of ''megagraphs", which are   
special types of graphs with decorations on their vertices and edges. Their work is categorical and not very adapted for the applications we have in mind, 
which require a more explicit description of the structures at hand. This is why we carry out the proof of the freeness of the ProP introduced in 
subsection \ref{sectiongraphes}. The complete proof of the main theorem (Theorem \ref{thm:freeness_Gr}) is postponed to Appendix 
\ref{appendix:proof_freeness_Gr}.

\subsection{Indecomposable graphs}

\begin{defi}\label{def:indecomposable}
We call a  graph $G$ \textbf{indecomposable} if the five following conditions hold:
\begin{enumerate}
\item $V(G)\neq \emptyset$.
\item $IO(G)=\emptyset$.
\item $L(G)=\emptyset$ or $G$ is reduced to a single loop.
\item If $G'$ and $G''$ are two graphs such that $G=G'\circ G''$, then $V(G')=\emptyset$ or $V(G'')=\emptyset$.
\item If $G'$ and $G''$ are two graphs and $\sigma$, $\tau$ are two permutations such that
$G=\sigma \cdot (G'*G'')\cdot \tau$, then $V(G')=\emptyset$ or $V(G'')=\emptyset$.
\end{enumerate}
For any $k,l\in  \N_0$, the subspace of $\Gr(k,l)$ generated by isoclasses of indecomposable graphs $G$ with $i(G)=k$
and $o(G)=l$ is denoted by $\Gri(k,l)$.
\end{defi}
\begin{remark}
\begin{enumerate}
\item  The permutations in the fifth item of the definition of indecomposable graphs play an 
 important role: without them, one would allow for non connected graphs to be indecomposable, which can well happen when  
 the indexations of the inputs and outputs of the various connected components do not match. For 
 example, the graph 
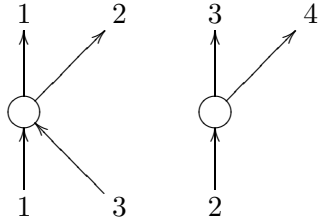
\[\xymatrix{1&2&3&4\\
\rond{}\ar[u]\ar[ru]&&\rond{}\ar[u]\ar[ru]\\
1\ar[u]&3\ar[lu]&2\ar[u]}\]
would be indecomposable. Permuting  inputs we obtain 
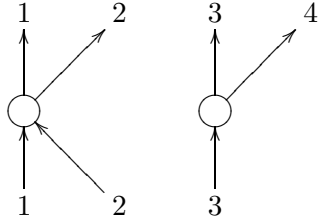
\[\xymatrix{1&2&3&4\\
\rond{}\ar[u]\ar[ru]&&\rond{}\ar[u]\ar[ru]\\
1\ar[u]&2\ar[lu]&3\ar[u]}\]
which is decomposable. The same requirement does not arise for the vertical concatenation since 
one can write $\sigma.(P\circ Q).\tau=(\sigma.P)\circ (Q.\tau)=P'\circ Q'$.
\item There is one special indecomposable graph $\grapheo$, formed by a unique loop. 
The other indecomposable graphs have no loop.
\end{enumerate}
\end{remark}

\begin{prop}
Let $G$ be a graph, $\sigma \in \sym_{o(G)}$ and $\tau\in \sym_{i(G)}$. Then
$G$ is indecomposable if, and only if, $\sigma\cdot G\cdot \tau$ is indecomposable. 
\end{prop}

\begin{proof}
Let us assume that $H=\sigma\cdot G\cdot \tau$ is indecomposable. Then $V(G)=V(H)\neq \emptyset$
and $IO(G)=IO(H)=\emptyset$.  Let us assume that $G=G'\circ G''$. Then:
\[H=\sigma\cdot(G'\circ G'')\cdot \tau=(\sigma\cdot G')\circ( G''\cdot \tau).\]
As $H$ is indecomposable, $V(G')=V(\sigma \cdot G')=\emptyset$ or $V(G'')=V(G''\cdot \tau)=\emptyset$. 
Let us assume that $G=\sigma'\cdot (G'*G'')\cdot \tau'$. Then:
\[H=\sigma\cdot(\sigma'\cdot (G'*G'')\cdot \tau')\cdot \tau=((\sigma\sigma')\cdot G')*(G''\cdot (\tau\tau')).\]

As  $H$ is indecomposable, $V(G')= V \left((\sigma\sigma')\cdot G'\right)=\emptyset$ or $V(G'')=V(G''\cdot (\tau\tau'))=\emptyset$.\\

Conversely, if $G$ is indecomposable, then $G=\sigma^{-1}\cdot H\cdot \tau^{-1}$ is indecomposable,
so $H$ is indecomposable. 
\end{proof}

\begin{notation} 
Let $G$ be a graph.
\begin{enumerate}
\item Let $J\subseteq V(G)$. We define (non uniquely due to the non uniqueness of the maps $\alpha'$ and $\beta'$) the graph $G_{\mid J}$ by:
\begin{align*}
V(G_{\mid J})&=J,\\
E(G_{\mid J})&=\{e\in E(G):  s(e)\in J, t(e)\in J\},\\
I(G_{\mid J})&=\{e\in I(G):  t(e)\in J\}\sqcup  \{e\in E(G): s(e)\notin J, t(e)\in J\},\\
O(G_{\mid J})&=\{e\in O(G):  s(e)\in J\}\sqcup  \{e\in E(G):  s(e)\in J, t(e)\notin J\},\\
IO(G_{\mid J})&=IO(G),\\
L(G_{\mid J})&=\emptyset.
\end{align*}
The source and target maps are defined by:
\begin{align*} 
&\forall e\in E(G_{\mid J})\sqcup O(G_{\mid J}),&s_{G_{\mid J}}(e)&=s(e),\\
&\forall e\in E(G_{\mid J})\sqcup I(G_{\mid J}),&t_{G_{\mid J}}(e)&=t(e),
\end{align*}
The indexation of the input edges is any indexation map $\alpha'$ such that:
\begin{align*}
&\forall e,e'\in \left(I(G)\sqcup IO(G)\right)\cap \left(I(G_{\mid J})\sqcup IO(G_{\mid J})\right),&
\alpha'(e)<\alpha'(e')&\Longleftrightarrow \alpha(e)<\alpha(e').
\end{align*}
The indexation of the output edges is any indexation map $\beta'$ such that:
\begin{align*}
&\forall f,f'\in \left(O(G)\sqcup IO(G)\right)\cap \left(O(G_{\mid J})\sqcup IO(G_{\mid J})\right),&
\beta'(f)<\beta'(f')&\Longleftrightarrow \beta(f)<\beta(f').
\end{align*}
\item We denote by $\tilde{G}$ the graph defined by:
\begin{align*}
V(\tilde{G})&=V(G),&E(\tilde{G})&=E(G),&L(\tilde{G})&=\emptyset,\\
I(\tilde{G})&=I(G),&O(\tilde{G})&=O(G),&IO(\tilde{G})&=\emptyset,\\
\tilde{s}&=s,&\tilde{t}&=t.
\end{align*}
The indexation of the input edges is the unique indexation map $\tilde{\alpha}$ such that:
\begin{align*}
&\forall e,e'\in I(G),&\tilde{\alpha}(e)<\tilde{\alpha}(e')&\Longleftrightarrow \alpha(e)<\alpha(e').
\end{align*}
The indexation of the output edges is the unique indexation map $\tilde{\beta}$ such that:
\begin{align*}
&\forall f,f''\in O(G),&\tilde{\beta}(f)<\tilde{\beta}(f')&\Longleftrightarrow \beta(f)<\beta(f').
\end{align*}
Roughly speaking, $\tilde{G}$ is obtained from $G$ by deletion of all the input-output edges and all the loops. 
\end{enumerate}
\end{notation}

\begin{defi}
Let $G$ be a graph.
\begin{enumerate}
\item A \textbf{path} in $G$ is a sequence $p=(e_1,\ldots,e_k)$ of internal edges of $G$ such that for any $i\in [k-1]$,
$t(e_i)=s(e_{i+1})$. The source of $p$ is $s(e_1)$ and its target is $t(e_k)$, and we shall say that 
$p$ is a path from $s(e_1)$ to $t(e_k)$ of length $k$. By convention, for any $x\in V(G)$, 
there exists a unique path from $x$ to $x$ of length $0$.
\item We shall say that a path $p$ is a \textbf{cycle} if its source and its target are equal 
and if its length is nonzero.
\end{enumerate}
\end{defi}
\begin{remark} A cycle of length one is  to be distinguished from a loop.
	\end{remark}
We consider oriented-pathwise connected components of graphs.
\begin{lemma} \label{lem:paths_subset}
Let $G$ be a graph such that $V(G)\neq \emptyset$. We denote by $\calO(G)$ the set of nonempty
subsets $I$ of $V(G)$ such that for any $x\in I$, for any $y\in V(G)$, if there exists a path in $G$ from $x$ to $y$,
then $y\in I$. Then:
\begin{enumerate}
\item If $I,J\in \calO(G)$, either $I\cap J=\emptyset$ or $I\cap J\in \calO(G)$.
\item For any $x \in V(G)$, there exists a unique element $\langle x\rangle \in \calO(G)$ which contains $x$
and is minimal for the inclusion. Moreover:
\[\langle x\rangle=\{y\in V(G): \mbox{ there exists a path in $G$ from $x$ to $y$}\}.\]
\end{enumerate}
\end{lemma}
Notice that, if $G_x$ is the connected component of $G$ that contains $x$, then  $\langle x\rangle\subseteq G_x$, but we do not necessarily have 
an equality, as the edges are \emph{oriented}.
\begin{proof}
1. If $I\cap J\neq \emptyset$, let $x\in I\cap J$ and $y\in V(G)$ such that there exists a path in $G$ from $x$ to $y$.
As $I,J\in \calO(G)$, $y\in I\cap J$, so $I\cap J\in \calO(G)$.

2. Note that $V(G)\in \calO(G)$. Let $x\in V(G)$; by the first item, the following element of $\calO(G)$
is the minimal (for the inclusion) element of $\calO(G)$ that contains $x$:
\[\langle x\rangle=\bigcap_{I\in \calO(G), \: x\in I}I.\]
On the one hand,   a set $I$ in $\calO(G)$ contains   $x$ if and only if any path emanating from $x$ ends at an element of $I$. So it contains all the ending vertices of such paths and hence the set
\[I_x:=\{y\in V(G): \mbox{ there exists a path in $G$ from $x$ to $y$}\}.\] Thus, $I_x\subseteq \langle x\rangle$. 
On the other hand, let $y\in I$ and $z\in V(G)$, such that there exists a path from $y$ to $z$ in $G$.
As there exists a path from $x$ to $y$ in $G$, there exists a path from $x$ to $z$, so $z\in I_x$.
Hence, $I_x$ lies in  $\calO(G)$ which in turn contains $x$, so $\langle x\rangle\subseteq I_x$. 
\end{proof}

\begin{prop}\label{prop:mindec}
Let $G$ be a graph such that $V(G)\neq \emptyset$. We denote by $J_1,\ldots,J_k$ the minimal elements
(for the inclusion) of the set  $\calO(G)$ of nonempty
subsets $I$ of $V(G)$ stable under paths as in Lemma \ref{lem:paths_subset}, and we set $G_i=\tilde{G}_{\mid J_i}$ for any $i\in [k]$. 
Then $G_1,\ldots,G_k$ are indecomposable graphs with no loop and there exists a graph $G_0$ with no loop,
integers $p$, $\ell$ and a permutation $\gamma$ such that:
\[G\approx (\gamma \cdot (G_1*\ldots *G_k*I_p)\circ G_0)*\grapheo^{*\ell},\]
where, as before  $\grapheo$ is the  indecomposable graph  formed by a unique loop.
Such a decomposition will be called  \textbf{minimal}.
\end{prop}

\begin{proof}
By definition, $V(G_i)=J_i\neq \emptyset$ and $IO(G_i)=\emptyset$ for any $i$. 
Let us assume that $G_i=G'\circ G''$. If $V(G')\neq \emptyset$, then clearly $V(G')\in \calO(G_i)$ 
and, as $J_i\in \calO(G)$, we deduce that $V(G')\in \calO(G)$. As $J_i$ is minimal in $\calO(G)$,
$V(G')=J_i=V(G_i)$, so $V(G'')=\emptyset$. Similarly, if $G_i=\sigma\cdot (G'*G'')\cdot \tau$,
then $V(G')=\emptyset$ or $V(G'')=\emptyset$: we proved that $G_i$ is indecomposable.

Let us assume that $I=V(G_i)\cap V(G_j)\neq \emptyset$. Then $I\in \calO(G)$ and, by minimality of $J_i$ and $J_j$,
$J_i=J_j=I$, so the $J_i$ are disjoint.

Let us set $K:=V(G)\setminus (J_1\cup\ldots \cup J_k)$ and $G':=G_{\mid K}$. As $J_1,\ldots, J_k$ lie in $\calO(G)$,
there is no internal edge of $G$ from a vertex of $G_i$ to a vertex of $G'$, and any outgoing edge of $G'$  is either
glued in $G$ to an incoming edge of $G_i$ or is an outgoing edge of $G$. Hence, 
there exists permutations
$\gamma$, $\sigma$ and $\tau$, and 
three integers $p:=|IO(G)|$, $q:=|\{e\in I(G):t(e)\in J_1\cup\ldots \cup J_k\}|$ and $\ell:=|L(G)|$
such that:
\[G=\gamma \cdot(G_1*\ldots*G_k*I_p)\circ (\sigma\cdot (I_q*G')\cdot \tau)*\grapheo^{*\ell}.\]
We conclude in taking $G_0=\sigma\cdot (I_q*G')\cdot \tau$. 
\end{proof}

Note that this decomposition is not unique: it depends on the indexation of the minimal elements of $\calO(G)$
and of the choice of the indexation of their input and output edges.  Importantly, it depends only on that.

\begin{prop}\label{propindecomposable}
Let $G$ be a graph such that $V(G)\neq \emptyset$ and $IO(G)=\emptyset$. The graph $G$ is indecomposable
if, and only if, $L(G)=\emptyset$ and for any $x,y\in V(G)$, there exists a path from $x$ to $y$ in $G$.
\end{prop}

\begin{proof}
First notice that if $|V(G)|=1$ the result trivially holds. In the following, we  therefore assume that $|V(G)|\geq 2$.

Let $G=\gamma \cdot (G_1*\ldots *G_k*I_p)\circ G_0*\grapheo^{*\ell}$ a minimal decomposition of $G$.

$\Longrightarrow$  Note that $V(G_1)\neq \emptyset$. As $G$ is indecomposable, 
necessarily $\ell=0$, $V(G_0)=\emptyset$, and there exists a permutation
$\tau\in \sym_p$ such that $G_0=I_q\cdot \tau$. Therefore, 
$G=\gamma \cdot (G_1*\ldots *G_k*I_p)\cdot \tau$. As $G$ is indecomposable, $k=1$ and $V(G)=V(G_1)=J_1$. 
Hence, $V(G_1)$ is both a minimal and the maximal element
of $\calO(G)$, which is consequently reduced to the singleton $\{V(G)\}$. 
Therefore, for any $x\in V(G)$, $\langle x\rangle=V(G)$, so for any $y\in V(G)$,
there exists a path from $x$ to $y$ in $G$.

$\Longleftarrow$ Firstly, note that $L(G)=\emptyset\Rightarrow \ell=0$.
 If $k\geq 2$, there is no path in $G$ from any vertex of $G_1$ to any vertex of $G_2$, so $k=1$. 
Thus, $V(G_0)=\emptyset$
and there exists a permutation $\tau$ such that $G_0=I_p\cdot \tau$. We obtain that
\[G=\gamma\cdot (G_1*I_p)\cdot \tau.\]
As $IO(G)=\emptyset$, we obtain that  
$p=0$, so $G=\gamma\cdot G\cdot \tau$ is indecomposable. \end{proof}
\begin{remark}
     Another way to formulate the above Proposition is to say that a graph $G$  is indecomposable if, and only if, 
     one (and only one) of the following conditions holds:
     \begin{itemize}
     \item $G=\grapheo$.
     \item $G$ has no loop, is connected and for any of its vertices $x$, 
     a cycle of  strictly positive length goes through $x$.
     \end{itemize}
    \end{remark}

\subsection{Freeness of $\Gr$}

We now state and give a sketch of the proof of one of the main results of this section, namely the freeness of the ProP $\Gr$. To our 
knowledge, this  result is new.
\begin{theo} \label{thm:freeness_Gr}
Let $P$ be a ProP and $\phi:\Gri\longrightarrow P$ be a morphism of $\sym\times \sym^{op}$-modules.
There exists a unique ProP morphism $\Phi:\Gr\longrightarrow P$ such that $\Phi_{\mid \Gri}=\phi$. 
In other words, $\Gr$ is the free ProP generated by $\Gri$. 
\end{theo}

\begin{proof} We provide here a sketch of the proof,  and refer the reader to Appendix \ref{appendix:proof_freeness_Gr} for a full proof.
We define $\Phi(G)$ for any graph $G$ by induction on its number $n$ of vertices. If $n=0$,  there exists a permutation $\sigma\in \sym_k$ such that $G=\sigma\cdot I_k$. We set
\[\Phi(G)=\sigma\cdot I_k.\]
If $n>0$ and $G$ is indecomposable, we  set $\Phi(G)=\phi(G)$. Otherwise, 
let \[G=\gamma \cdot (G_1*\ldots *G_k*I_p)\circ G_0*\grapheo^{*\ell}\]
be a  minimal decomposition of $G$.  As  $V(G_1)\neq \emptyset$, $|V(G_0)|<n$, we set:
\[\Phi(G)=\gamma \cdot (\phi(G_1)*\ldots *\phi(G_k)*I_p)\circ \Phi(G_0)*\phi(\grapheo)^{*\ell}.\]
One can prove that this does not depend on the choice of the minimal decomposition of $G$ with the help of the ProP axioms
applied to $P$. Using minimal decompositions of vertical or horizontal concatenations of graphs,
one can show that $\Phi$ is compatible with both concatenations. \end{proof}

\subsection{Cycleless graphs} \label{subsection:cycless_graph}

\begin{defi}
For any $k,l\in  \N_0$, we denote by $\Grc(k,l)$ the subspace of $\Gr(k,l)$ generated by the graphs which do not contain any cycle
nor any loop. Note that $\Grc$ is a $\sym\times \sym^{op}$-sub-module of $\Gr$.

As before, we write $\Grci$ for the set of indecomposable cycleless and loopless graphs in $\Grc$.
\end{defi}
A simple yet important observation is the following.
\begin{prop}
 $\Grc$ is a sub-ProP of $\Gr$.
\end{prop}
\begin{proof}
 First, notice that $I_0$ and $I_1$ are in $\Grc$.
Let us check the stability of $\Grc$ under the  horizontal and vertical concatenation.
 
 Let $G_1,G_2$ be two graphs without cycle. By construction, there is no edge $e$ of $G_1*G_2$ such that $s(e)\in V(G_1)$ and $t(e)\in V(G_2)$, or 
 such that $s(e)\in V(G_2)$ and $t(e)\in V(G_1)$. So a cycle in $G_1*G_2$ is a cycle in $G_1$ or $G_2$. Thus $\Grc$ is stable by horizontal concatenation.
 
 Similarly, let $G_1,G_2$ be two graphs without cycle such that $G_1\circ G_2$ is defined. Then using the same argument, a cycle of $G_1\circ G_2$ must either 
 be a cycle of $G_1$, a cycle of $G_2$ (both being contradictions) or contain an edge $e$ such that $s(e)\in V(G_1)$ and $t(e)\in V(G_2)$. This contradicts 
 the definition of $\circ$ for graphs.
\end{proof}
In this particular example, we recover the description of a free ProP in terms of oriented graphs \cite{Vallette1,Vallette2}:
\begin{prop} \label{prop:cycleless_graphs}
For any $k,l\in  \N_0$, we denote by $G_{k,l}$ the graph such that:
\begin{align*}
V(G_{k,l})&=\{\star\},&I(G_{k,l})&=[k],&IO(G_{k,l})&=\emptyset,\\
E(G_{k,l})&=\emptyset,&O(G_{k,l})&=[l]&L(G_{k,l})&=\emptyset.
\end{align*}
For any $i\in [k]$, for any $j\in [l]$:
\begin{align*}
\alpha(i)&=i,&\beta(j)&=j,\\
t(i)&=\star,&s(j)&=\star.
\end{align*}
These graphs generate a trivial $\sym\times \sym^{op}$-module $\Grci$, 
and $\Grc$ is the free ProP generated by $\Grci$. 
\end{prop}
Graphically, $G_{k,l}$ is represented as follows:
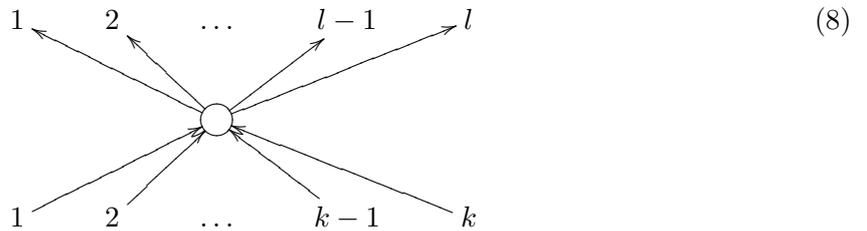
\begin{equation} \label{eq:indecompasable_graph}
\xymatrix{1&2&\dots&l-1&l\\
&&\rond{}\ar[rru]\ar[ru]\ar[lu]\ar[llu]&&\\
1\ar[rru]&2\ar[ru]&\ldots&k-1\ar[lu]&k\ar[llu]}
\end{equation}

\begin{proof}
For any permutations $\sigma\in \sym_l$, $\tau\in \sym_k$, $\sigma\cdot G_{k,l}\cdot \tau$ is isomorphic to 
$G_{k,l}$, through the isomorphism defined by:
\begin{align*}
f_V&=\mathrm{Id}_{\{\star\}},&f_I&=\tau^{-1},&f_O&=\sigma.
\end{align*}
so indeed these graphs generate a trivial $\sym\times \sym^{op}$-module.

Since sub-graphs of a graph without cycle,  are without cycle, the following is an easy consequence of Proposition \ref{prop:mindec}.

\begin{lemma}\label{lem:nocyclemindec}
Let $G$ be a graph without cycle and without loop, such that $V(G)\neq \emptyset$, then a minimal decomposition
\[G\approx \gamma \cdot (G_1*\ldots *G_k*I_p)\circ G_0\]  yields a decomposition without cycles.
Note that, as $G$ has no loop, $\ell=0$.
\end{lemma}

Let $G$ be an indecomposable graph without cycle nor loop and let us assume that $V(G)\geq 2$. Let $x\neq y$ in $V(G)$.
As $G$ is indecomposable, by Proposition \ref{propindecomposable}, there exists a path $(e_1,\ldots,e_k)$
from $x$ to $y$ in $G$ and a path $(f_1,\ldots,f_l)$ from $y$ to $x$ in $G$.
Hence, there exists a cycle $(e_1,\ldots,e_k,f_1,\ldots,f_l)$ in $G$: this is a contradiction. We obtain that
$V(G)$ is reduced to a single element. As $IO(G)=L(G)=\emptyset$, $G=G_{i(G),o(G)}$. This gives:
\[\Grc\cap \Gri=\Grci.\]
As $\Gr$ is the free ProP generated by $\Gri$, for any $\sym\times \sym^{op}$-sub-module $P$ of $\Gri$,
the sub-ProP of $\Gr$ generated by $P$ is freely generated by $P$. This holds in particular for $\Grci $.
It remains to prove that the sub-ProP $\langle \Grci \rangle$ generated by $\Grci $ is $\Grc$.

Clearly, if $G$ and $G'$ are graphs without cycles, then $G*G'$ and $G\circ G'$ are without cycles,
so $\Grc$ is a sub-ProP of $\Gr$, which contains $\Grci$. Consequently, $\langle \Grci \rangle\subseteq \Grc$.
Conversely, let $G$ be a graph without cycle and let us prove that $G\in \langle \Grci \rangle$
by induction on $n=|V(G)|$. If $n=0$, then $G=\sigma\cdot I_k$ for a certain permutation $\sigma\in \sym_k$,
so $G$ belongs to $\langle \Grci \rangle$. Otherwise, let us consider a minimal decomposition of $G$ in $\mathrm{Gr}^c$ (see Lemma 
\ref{lem:nocyclemindec}):
\[G=\gamma \cdot(G_1*\ldots*G_k*I_p)\circ G_0.\]
Since $G_1,\ldots,G_k$ are indecomposable, they lie in
$ \Grci$. Since $k\geq1$ and $V(G_i)\neq\emptyset$ we have $G_0\in \langle \Grci \rangle$ by the induction hypothesis,
so $G\in \langle \Grci \rangle$.  
\end{proof}

\begin{remark} \label{rk:decoration}
One can also work with graphs with various types of edges: each edge $e$ (internal, input, output or input-output)
of the graph under consideration has a type $\mathrm{type}(e)$, chosen in a fixed set of types $T$. 
The horizontal composition of two graphs $G$ and $G'$ exists in any case, whereas their vertical concatenation
exists if, and only if, for any $i\in [o(G)]$, the type of the output edge $\beta^{-1}(i)$ of $G$
and the type of the input edge $\alpha^{-1}(i)$ of $G'$ are the same.
One obtains a $T$-coloured ProP, and one can prove similarly a freeness result. Restricting to typed graphs
without cycles, we obtain a free $T$-coloured ProP generated by graphs with only one vertex (and no input-output edge).
\end{remark}

\subsection{Planar graphs and free ProPs}

We recall from Definition \ref{defi:graph} that $s:E(G)\sqcup O(G)\longrightarrow V(G)$ stands for the source map  
and $t:E(G)\sqcup I(G)\longrightarrow V(G)$ stands for the target map.

\begin{defi}
Let $G$ be a graph and $v\in E(G)$ be a vertex of $G$.
We put:
\begin{align*}
I(v)&=\{e\in I(G)\sqcup E(G),\: t(e)=v\},\\
O(v)&=\{e\in O(G)\sqcup E(G),\: s(e)=v\}.
\end{align*}
We also set  $i(v)=|I(v)|$ and $o(v)=|O(v)|$.
\end{defi}

The number $i(v)$ (resp.  $o(v)$) counts the number of input (resp. output) edges and ingoing (resp. outgoing) 
arrows at the vertex $v$. 

\begin{defi}
A \textbf{planar graph} is a graph $G$ such that, for any vertex $v\in V(G)$, $I(v)$ and $O(v)$ are totally ordered.
The set of planar graphs is denoted by $\PGr$ and the set of planar graphs with no cycle and no loop is denoted by
$\PGrc$.
The set of planar graphs $G$ (resp. of planar graphs $G$  with no cycle and no loop) 
with $|I(G)|+|IO(G)|=k$ and $[O(G)|+|IO(G)|=l$ is denoted by $\PGr(k,l)$ (resp. by $\PGrc(k,l)$).
    \end{defi}

Graphically, we shall represent the orders on the incoming and outgoing edges of a vertex by drawing the vertices by boxes, the incoming and outgoing edges of any vertex 
being ordered from left to right. For example, we distinguish the two following situations:
\begin{align*}
&\begin{tikzpicture}[line cap=round,line join=round,>=triangle 45,x=0.7cm,y=0.7cm]
\clip(0.8,-0.5) rectangle (2.2,2.5);
\draw [line width=.4pt] (0.8,0.)-- (2.2,0.);
\draw [line width=.4pt] (2.2,0.)-- (2.2,-0.5);
\draw [line width=.4pt] (2.2,-0.5)-- (0.8,-0.5);
\draw [line width=.4pt] (0.8,-0.5)-- (0.8,0.);
\draw [line width=.4pt] (0.8,2.)-- (2.2,2.);
\draw [line width=.4pt] (2.2,2.)-- (2.2,2.5);
\draw [line width=.4pt] (2.2,2.5)-- (0.8,2.5);
\draw [line width=.4pt] (0.8,2.5)-- (0.8,2.);
\draw [->,line width=.4pt] (1.,0.) -- (1.,2.);
\draw [->,line width=.4pt] (2.,0.) -- (2.,2.);
\end{tikzpicture}&
&\begin{tikzpicture}[line cap=round,line join=round,>=triangle 45,x=0.7cm,y=0.7cm]
\clip(0.8,-0.5) rectangle (2.2,2.5);
\draw [line width=.4pt] (0.8,0.)-- (2.2,0.);
\draw [line width=.4pt] (2.2,0.)-- (2.2,-0.5);
\draw [line width=.4pt] (2.2,-0.5)-- (0.8,-0.5);
\draw [line width=.4pt] (0.8,-0.5)-- (0.8,0.);
\draw [line width=.4pt] (0.8,2.)-- (2.2,2.);
\draw [line width=.4pt] (2.2,2.)-- (2.2,2.5);
\draw [line width=.4pt] (2.2,2.5)-- (0.8,2.5);
\draw [line width=.4pt] (0.8,2.5)-- (0.8,2.);
\draw [->,line width=.4pt] (1.,0.) -- (2.,2.);
\draw [->,line width=.4pt] (2.,0.) -- (1.,2.);
\end{tikzpicture} 
\end{align*}

\begin{remark}
This notion of planarity is not the usual one used in graph theory, as we authorise crossings of edges.
\end{remark}
Since a planar graph is a graph, the horizontal and vertical concatenation of planar graphs are defined by the concatenations of the underlying graphs, 
which preserve the orders around each of the vertices. It is a simple exercise to check that $\PGr$ is still a $\sym\times \sym^{op}$-module and we left 
it to the reader. Hence, $\PGr$ inherits a ProP structure from $\Gr$. As before, $\PGrc$ is a sub-ProP of $\PGr$.

We shall say that a planar graph is indecomposable if the underlying graph is indecomposable.
The set $\PGri$ of indecomposable planar graphs  forms a $\sym\times \sym^{op}$-module.
We then obtain a minimal decomposition of planar graphs similar to the one of Proposition \ref{prop:mindec}.
For any $k,l\in \mathbb{N}$, we denote by $PG_{k,l}$  the planar graph obtained from $G_{k,l}$ by ordering 
the sets $[k]$ and $[l]$ of incoming and outgoing edges of the unique vertex $\star$ by their usual orders. 
We obtain the planar counterpart of   Theorem \ref{thm:freeness_Gr}.

\begin{theo} \label{thm:freeness_PGr}
\begin{enumerate}
\item Let $P$ be a ProP and $\phi:\PGri\longrightarrow P$ be a morphism of $\sym\times \sym^{op}$-modules.
There exists a unique ProP morphism $\Phi:\Gr\longrightarrow P$ such that $\Phi_{\mid \PGri}=\phi$. 
In other words, $\PGr$ is the free ProP generated by $\PGri$. 
\item The planar  graphs $PG_{k,l}$ generate a free $\sym\times \sym^{op}$-module $\PGrci$, 
and $\PGrc$ is the free ProP generated by $\PGrci$. 
\end{enumerate}
\end{theo}

\section{Graphs decorated by ProPs and endofunctors  of ProPs} \label{section:Hom_dec_graphs}

This section is motivated by Feynman graphs, in which case the decorations are distribution kernels. Since we expect to be able to equip the later 
with a ProP structure, we study here graphs decorated by ProPs.
The results of Section \ref{sec:free_prop} then allow  us to build a endofunctor $\Ga$ on the category of ProPs.

\subsection{The ProP $\Gr(X)$ of decorated graphs as a free ProP}

Throughout this paragraph, $X=(X_{k,l})_{k,l\geqslant 0}$ is  a family of sets.
\begin{defi}
  A graph decorated by $X$ (or $X$-decorated graph, or simply decorated graph) 
  is a couple $(G,d_G)$ with $G$ a graph as in Definition \ref{defi:graph} and 
 $\displaystyle d_G:V(G)\longrightarrow \bigsqcup_{k,l\in  \N_0}X_{k,l}$ a map, such that for any 
vertex $v\in V(G)$, $d_G(v)\in X_{i(v),o(v)}$. We denote by $\Gr(X)$ (resp. $\Grc(X)$) the set of graphs (resp. the set 
 of cycleless graphs) decorated by $X$. 
 We define similarly $X$-decorated planar graphs and we denote by $\PGr(X)$ (resp. $\PGrc(X)$)
the set of planar graphs (resp. the set  of cycleless planar graphs) decorated by $X$. 
\end{defi}
Most of the results on graphs naturally generalise to $X$-decorated graphs. In particular, we have the horizontal  (resp. vertical) concatenation of graphs, denoted by $*$ (resp. $\circ$):
\begin{align*}
 (G,d_G)*(G',d_{G'}) &= (G*G',d_{G*G'}), &(G,d_G)\circ(G',d_{G'}) &= (G\circ G',d_{G\circ G'}).
\end{align*}
The set  of vertices of $G*G'$ and $G'\circ G'$ both being the disjoint union  $V(G)\sqcup V(G')$ of the vertices of $G$ and $G'$,  we 
define   $d_{G* G'}=d_{G\circ G'}$ on the set $V(G)\cup V(G')$ by  $d_{G* G'}|_{V(G)}:=d_G$, $d_{G*G'}|_{V(G')}:=d_{G'}$.

Furthermore, the actions on the left and on the right of the permutation group on $\Gr$ extend to actions on $\Gr(X)$ since the aforementioned actions 
leave the set of vertices of a graph invariant. Here are the decorated and cycleless versions of Theorem \ref{theo:ProP_graph}, which to our knowledge 
is new:
\begin{theo} \label{theo:ProP_graph_2}
The families $\Gr(X)$ and $\PGr(X)$, equipped the  above $\sym\times \sym^{op}$-action and the above  horizontal and 
vertical concatenations, are ProPs. The family $\Grc(X)$ is a sub-ProP of $\Gr(X)$ and
the family $\PGrc(X)$ is a sub-ProP of $\PGr(X)$.
\end{theo}

 Proposition \ref{prop:mindec} generalises to the case of decorated graphs.
\begin{prop}
Let $(G,d_G)$ be an $X$-decorated graph such that $V(G)\neq \emptyset$. We denote by $J_1,\ldots,J_k$ the minimal elements
(for the inclusion) of $\calO(G)$. As before, we set $G_i=\tilde{G}_{\mid J_i}$  and $d_i=d_G|_{J_i}$ for any $i\in [k]$. 
Then there exists an $X$-decorated graph $(G_0,d_0)$ with no loop,
 integers $p$, $\ell$ and a permutation $\gamma$ such that:
\[(G,d_G)\approx \gamma \cdot ((G_1,d_1)*\ldots *(G_k,d_k)*I_p)\circ (G_0,d_0)*\grapheo^{*\ell}.\]
As in the non decorated case, we call such a decomposition   minimal.
\end{prop}
\begin{proof}
By Proposition \ref{prop:mindec}, $G$ admits a minimal decomposition
\[G\approx \gamma \cdot (G_1*\ldots *G_k*I_p)\circ G_0*\grapheo^{*\ell}.\]
We observe that we can identify the vertices of $G_0$ with those of
$G\setminus(G_1\sqcup\cdots\sqcup G_k)$. We can therefore set
$d_0:=d_G|_{V(G)\setminus(V(G_1)\sqcup\cdots\sqcup V(G_k))}$. The result then follows from the definition of the 
actions of the permutation group on $\Gr(X)$ using the horizontal and vertical concatenations.
\end{proof}

As in the non decorated case, we  denote by $\Gri(X)$ (resp. $\Grci(X)$) the indecomposable graphs (resp. the cycleless indecomposable graphs) 
decorated by $X$. Notice that the graphs $(G_i,d_i)$; for $i\in[k]$, are indecomposable.
We define $\PGri(X)$ and $\PGrci(X)$ similarly.

The key result of this paragraph is the decorated version of the universal property (Theorem \ref{thm:freeness_Gr}).
\begin{theo} \label{thm:univ_prop_deco}
\begin{enumerate}
\item  Let $P$ be a ProP and $\phi:\Gri(X)\longrightarrow P$ be a morphism of $\sym\times \sym^{op}$-modules.
There exists a unique ProP morphism $\Phi:\Gr(X)\longrightarrow P$ such that $\Phi_{\mid \Gri(X)}=\phi$. 
In other words, $\Gr(X)$ is the free ProP generated by the $\sym\times \sym^{op}$-module $\Gri(X)$.

Furthermore, $\Grc(X)$ is the free ProP generated by the $\sym\times \sym^{op}$-module $\Grci(X)$,  which is isomorphic
to the trivial $\sym\times \sym^{op}$-module generated by $X$.

\item  Let $P$ be a ProP and $\phi:\PGri(X)\longrightarrow P$ be a morphism of $\sym\times \sym^{op}$-modules.
There exists a unique ProP morphism $\Phi:\PGr(X)\longrightarrow P$ such that $\Phi_{\mid \PGri(X)}=\phi$. 
In other words, $\PGr(X)$ is the free ProP generated by the $\sym\times \sym^{op}$-module $\PGri(X)$.

Furthermore, $\PGrc(X)$ is the free ProP generated by  the $\sym\times \sym^{op}$-module $\PGrci(X)$
generated by $X$, which is isomorphic to  the free $\sym\times \sym^{op}$-module generated by $X$. 
\end{enumerate}
\end{theo}
\begin{remark}
\begin{enumerate}
\item  This result generalises Theorem \ref{thm:freeness_Gr} and Proposition \ref{prop:cycleless_graphs}. However, it is \emph{not} a direct 
 consequence of these previous results. Given $G\in \Gr$, $d_G,d_G':V(G)\longrightarrow X$ two decoration maps of $G$, we a priori have 
 $\Phi(G,d_G)\neq\Phi(G,d_G')$.
\item The $\sym\times \sym^{op}$-modules $\Grci(X)$ and $\PGrci(X)$ differ from one another in so far as for any $k,l\in  \N_0$,
$\Grci(k,l)$ is a trivial $\sym_l\otimes \sym_k^{op}$-module,
whereas $\PGrci(k,l)$ is a free $\sym_l\otimes \sym_k^{op}$-module. They are both generated by $X_{k,l}$.
\end{enumerate}
\end{remark}
\begin{proof}
 The proofs of Theorem \ref{thm:freeness_Gr} and Proposition \ref{prop:cycleless_graphs} can be reproduced in extenso in the decorated setup, simply replacing graphs by decorated graphs and using the decorated version of the minimal 
 decomposition and will therefore not reproduce it   here. Let us however notice that
 \begin{itemize}
  \item the transformation of type A arising in the proof of Theorem \ref{thm:freeness_Gr} only concerns indexation of edges. As such, it easily 
  generalises to decorated graphs.
  \item  the transformation of type B arising in the proof of Theorem \ref{thm:freeness_Gr} exchanges two subgraphs of $G$. It  therefore extends to the 
  decorated case as a transformation exchanging two decorated graphs. The rest of the proof of Theorem \ref{thm:freeness_Gr} remains unchanged.
  \item  the cycleless indecomposable graphs are still in the decorated case the graphs with exactly one vertex, since the decorations play no role 
  in the definition of indecomposable.
  \item  the rest of the proof of \ref{prop:cycleless_graphs} also  generalises in a straightforward manner to the decorated case. \qedhere
 \end{itemize}

\end{proof}

\subsection{An endofunctor of the category of  $\sym\times \sym^{op}$-modules}

We now assume that the family $X=(X_{k,l})_{k,l\in  \N_0}$ is a $\sym\times \sym^{op}$-module.
We define another $\sym\times \sym^{op}$-module  $\Gacirc(X)$ on graphs, taking into account this module structure.

Let $G\in \PGr$. As $G$ is a planar graph, the sets $I(v)$ and $O(v)$
are canonically identified with $[i(v)]$ and $[o(v)]$ thanks to their total orders.

For any vertex $v\in V(G)$, there is a natural action of $\sym_{o(v)}\times \sym_{i(v)}^{op}$, obtained by acting on the total
orders of $O(v)$ and $I(v)$. The graph obtained from $G$ by the action of $(\sigma,\tau)$ on the vertex $v$ is denoted by
\[\sigma \cdot_v G\cdot_v \tau.\]
For example:
\begin{center} 
	$\substack{\displaystyle (12)\cdot_v\\ \vspace{1.5cm}}$
	\begin{tikzpicture}[line cap=round,line join=round,>=triangle 45,x=0.7cm,y=0.7cm]
	\clip(0.8,-0.5) rectangle (2.2,2.5);
	\draw [line width=.4pt] (0.8,0.)-- (2.2,0.);
	\draw [line width=.4pt] (2.2,0.)-- (2.2,-0.5); 
	\draw [line width=.4pt] (2.2,-0.5)-- (0.8,-0.5);
	\draw [line width=.4pt] (0.8,-0.5)-- (0.8,0.); 
	\draw [line width=.4pt] (0.8,2.)-- (2.2,2.); 
	\draw [line width=.4pt] (2.2,2.)-- (2.2,2.5);
	\draw [line width=.4pt] (2.2,2.5)-- (0.8,2.5); 
	\draw [line width=.4pt] (0.8,2.5)-- (0.8,2.);
	\draw [->,line width=.4pt] (1.,0.) -- (1.,2.);
	\draw [->,line width=.4pt] (2.,0.) -- (2.,2.);
	\draw (1.2,0.) node[anchor=north west] {\scriptsize $v$};
	\draw (1.2,2.5) node[anchor=north west] {\scriptsize $w$};
	\end{tikzpicture}\,
	$\substack{\displaystyle =\\ \vspace{1.5cm}}$\,
	\begin{tikzpicture}[line cap=round,line join=round,>=triangle 45,x=0.7cm,y=0.7cm]
	\clip(0.8,-0.5) rectangle (2.2,2.5);
	\draw [line width=.4pt] (0.8,0.)-- (2.2,0.); 
	\draw [line width=.4pt] (2.2,0.)-- (2.2,-0.5);
	\draw [line width=.4pt] (2.2,-0.5)-- (0.8,-0.5);
	\draw [line width=.4pt] (0.8,-0.5)-- (0.8,0.);
	\draw [line width=.4pt] (0.8,2.)-- (2.2,2.);
	\draw [line width=.4pt] (2.2,2.)-- (2.2,2.5);
	\draw [line width=.4pt] (2.2,2.5)-- (0.8,2.5);
	\draw [line width=.4pt] (0.8,2.5)-- (0.8,2.);
	\draw [->,line width=.4pt] (1.,0.) -- (1.,2.);
	\draw [->,line width=.4pt] (2.,0.) -- (2.,2.);
	\draw (1.2,0.) node[anchor=north west] {\scriptsize $v$};
	\draw (1.2,2.5) node[anchor=north west] {\scriptsize $w$};
	\end{tikzpicture}
	$\substack{\displaystyle \cdot_w (12)\\ \vspace{1.5cm}}$
	$\substack{\displaystyle =\\ \vspace{1.5cm}}$
	\begin{tikzpicture}[line cap=round,line join=round,>=triangle 45,x=0.7cm,y=0.7cm]
	\clip(0.8,-0.5) rectangle (2.2,2.5);
	\draw [line width=.4pt] (0.8,0.)-- (2.2,0.);
	\draw [line width=.4pt] (2.2,0.)-- (2.2,-0.5);
	\draw [line width=.4pt] (2.2,-0.5)-- (0.8,-0.5);
	\draw [line width=.4pt] (0.8,-0.5)-- (0.8,0.);
	\draw [line width=.4pt] (0.8,2.)-- (2.2,2.);
	\draw [line width=.4pt] (2.2,2.)-- (2.2,2.5);
	\draw [line width=.4pt] (2.2,2.5)-- (0.8,2.5);
	\draw [line width=.4pt] (0.8,2.5)-- (0.8,2.);
	\draw [->,line width=.4pt] (1.,0.) -- (2.,2.);
	\draw [->,line width=.4pt] (2.,0.) -- (1.,2.);
	\draw (1.2,0.) node[anchor=north west] {\scriptsize $v$};
	\draw (1.2,2.5) node[anchor=north west] {\scriptsize $w$};
	\end{tikzpicture}$\substack{\displaystyle .\\ \vspace{1.5cm}}$\\
	\vspace{-.5cm}\end{center}
Let $G\in \PGr(k,l)$ and $X$ a $\sym\times \sym^{op}$-module. We define 
\[G(X)=\bigotimes_{v\in X(G)} X(i(v),o(v)).\]
The elements of $G(P)$ will be written as linear spans of tensors
\[\bigotimes_{v\in V(G)} x_v.\]
In other words, we decorate any vertex of $G$ by an element of $X$, with respect to the number of incoming and outgoing edges
of $v$, and we take these decorations to be linear in each vertex. 

Let
\[\PGr(X)(k,l):=\bigoplus_{G\in \PGr(k,l)}  \C G\otimes G(X),\]
whose elements are linear spans of tensors
\[G\otimes \left(\bigotimes_{v\in V(G)} x_v\right). \] 

\begin{remark} \label{rk:representation_Gamma(G)}
	Graphically, this element of $\PGr(X)(k,l)$ is represented by the planar
	graph $G$ where each vertex $v\in V(G)$ is decorated by $x_v$.
\end{remark}
$\PGr(X)(k,l)$ is a $\sym_l\times \sym_k^{op}$-module, by the action on the indexation of the incoming and outgoing edges of the
graphs.

Let $I(k,l)$ be the $\sym_l\times \sym_k^{op}$-submodule of $\PGr(X)(k,l)$ generated by   elements of the form
\begin{equation} \label{eq:quotient_monad}
\sigma\cdot_{v_0}G\cdot_{v_0}\tau \otimes \left(\bigotimes_{v\in V(G)} x_v\right)
-G\otimes \left(\left(\bigotimes_{v\in V(G)\setminus\{v_0\}} x_v\right)\otimes \sigma\cdot x_{v_0}\cdot \tau\right),
\end{equation}
where $G\in \PGr(k,l)$, $v_0\in V(G)$, $\sigma \in \sym_{o(v_0)}$ and $\tau \in \sym_{i(v_0)}$. 
We further define
\[\Gacirc(X)(k,l):=\dfrac{\PGr(X)(k,l)}{I(k,l)}.\]

Here is the type of relations we obtain graphically:
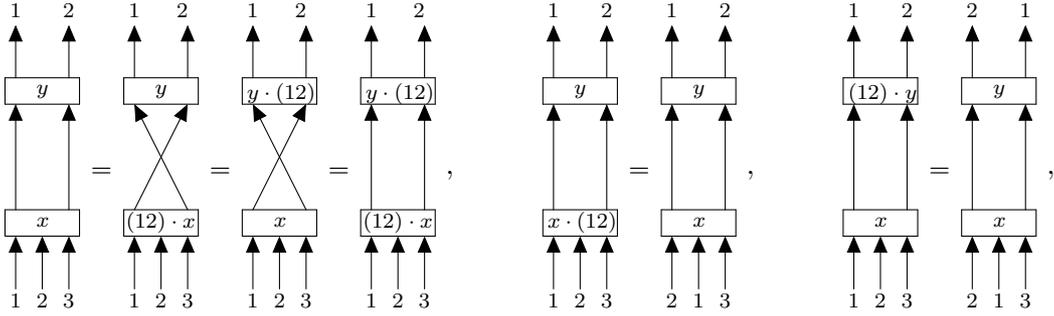
\begin{align*}
\begin{tikzpicture}[line cap=round,line join=round,>=triangle 45,x=0.7cm,y=0.7cm]
\clip(0.6,-2.1) rectangle (2.4,4.5);
\draw [line width=.4pt] (0.8,0.)-- (2.2,0.);
\draw [line width=.4pt] (2.2,0.)-- (2.2,-0.5);
\draw [line width=.4pt] (2.2,-0.5)-- (0.8,-0.5);
\draw [line width=.4pt] (0.8,-0.5)-- (0.8,0.);
\draw [line width=.4pt] (0.8,2.)-- (2.2,2.);
\draw [line width=.4pt] (2.2,2.)-- (2.2,2.5);
\draw [line width=.4pt] (2.2,2.5)-- (0.8,2.5);
\draw [line width=.4pt] (0.8,2.5)-- (0.8,2.);
\draw (1.2,0.05) node[anchor=north west] {\scriptsize $x$};
\draw (1.2,2.55) node[anchor=north west] {\scriptsize $y$};
\draw [->,line width=.4pt] (1.,0.) -- (1.,2.);
\draw [->,line width=.4pt] (2.,0.) -- (2.,2.);
\draw [->,line width=.4pt] (1.,-1.5) -- (1.,-0.5);
\draw [->,line width=.4pt] (1.5,-1.5) -- (1.5,-0.5);
\draw [->,line width=.4pt] (2.,-1.5) -- (2.,-0.5);
\draw (0.7,-1.4) node[anchor=north west] {\scriptsize $1$};
\draw (1.2,-1.4) node[anchor=north west] {\scriptsize $2$};
\draw (1.7,-1.4) node[anchor=north west] {\scriptsize $3$};
\draw [->,line width=.4pt] (1.,2.5) -- (1,3.5);
\draw [->,line width=.4pt] (2.,2.5) -- (2.,3.5);
\draw (0.7,4.1) node[anchor=north west] {\scriptsize $1$};
\draw (1.7,4.1) node[anchor=north west] {\scriptsize $2$};
\end{tikzpicture}&
\substack{\displaystyle =\\ \vspace{3.5cm}}\begin{tikzpicture}[line cap=round,line join=round,>=triangle 45,x=0.7cm,y=0.7cm]
\clip(0.6,-2.1) rectangle (2.4,4.5);
\draw [line width=.4pt] (0.8,0.)-- (2.2,0.);
\draw [line width=.4pt] (2.2,0.)-- (2.2,-0.5);
\draw [line width=.4pt] (2.2,-0.5)-- (0.8,-0.5);
\draw [line width=.4pt] (0.8,-0.5)-- (0.8,0.);
\draw [line width=.4pt] (0.8,2.)-- (2.2,2.);
\draw [line width=.4pt] (2.2,2.)-- (2.2,2.5);
\draw [line width=.4pt] (2.2,2.5)-- (0.8,2.5);
\draw [line width=.4pt] (0.8,2.5)-- (0.8,2.);
\draw (0.65,0.15) node[anchor=north west] {\scriptsize $(12)\cdot x$};
\draw (1.2,2.55) node[anchor=north west] {\scriptsize $y$};
\draw [->,line width=.4pt] (1.,0.) -- (2.,2.);
\draw [->,line width=.4pt] (2.,0.) -- (1.,2.);
\draw [->,line width=.4pt] (1.,-1.5) -- (1.,-0.5);
\draw [->,line width=.4pt] (1.5,-1.5) -- (1.5,-0.5);
\draw [->,line width=.4pt] (2.,-1.5) -- (2.,-0.5);
\draw (0.7,-1.4) node[anchor=north west] {\scriptsize $1$};
\draw (1.2,-1.4) node[anchor=north west] {\scriptsize $2$};
\draw (1.7,-1.4) node[anchor=north west] {\scriptsize $3$};
\draw [->,line width=.4pt] (1.,2.5) -- (1,3.5);
\draw [->,line width=.4pt] (2.,2.5) -- (2.,3.5);
\draw (0.7,4.1) node[anchor=north west] {\scriptsize $1$};
\draw (1.6,4.1) node[anchor=north west] {\scriptsize $2$};
\end{tikzpicture}
\substack{\displaystyle =\\ \vspace{3.5cm}}\begin{tikzpicture}[line cap=round,line join=round,>=triangle 45,x=0.7cm,y=0.7cm]
\clip(0.6,-2.1) rectangle (2.4,4.5);
\draw [line width=.4pt] (0.8,0.)-- (2.2,0.);
\draw [line width=.4pt] (2.2,0.)-- (2.2,-0.5);
\draw [line width=.4pt] (2.2,-0.5)-- (0.8,-0.5);
\draw [line width=.4pt] (0.8,-0.5)-- (0.8,0.);
\draw [line width=.4pt] (0.8,2.)-- (2.2,2.);
\draw [line width=.4pt] (2.2,2.)-- (2.2,2.5);
\draw [line width=.4pt] (2.2,2.5)-- (0.8,2.5);
\draw [line width=.4pt] (0.8,2.5)-- (0.8,2.);
\draw (1.2,0.05) node[anchor=north west] {\scriptsize $x$};
\draw (0.7,2.6) node[anchor=north west] {\scriptsize $y\cdot (12)$};
\draw [->,line width=.4pt] (1.,0.) -- (2.,2.);
\draw [->,line width=.4pt] (2.,0.) -- (1.,2.);
\draw [->,line width=.4pt] (1.,-1.5) -- (1.,-0.5);
\draw [->,line width=.4pt] (1.5,-1.5) -- (1.5,-0.5);
\draw [->,line width=.4pt] (2.,-1.5) -- (2.,-0.5);
\draw (0.7,-1.4) node[anchor=north west] {\scriptsize $1$};
\draw (1.2,-1.4) node[anchor=north west] {\scriptsize $2$};
\draw (1.7,-1.4) node[anchor=north west] {\scriptsize $3$};
\draw [->,line width=.4pt] (1.,2.5) -- (1,3.5);
\draw [->,line width=.4pt] (2.,2.5) -- (2.,3.5);
\draw (0.7,4.1) node[anchor=north west] {\scriptsize $1$};
\draw (1.6,4.1) node[anchor=north west] {\scriptsize $2$};
\end{tikzpicture}
\substack{\displaystyle =\\ \vspace{3.5cm}}\begin{tikzpicture}[line cap=round,line join=round,>=triangle 45,x=0.7cm,y=0.7cm]
\clip(0.6,-2.1) rectangle (2.4,4.5);
\draw [line width=.4pt] (0.8,0.)-- (2.2,0.);
\draw [line width=.4pt] (2.2,0.)-- (2.2,-0.5);
\draw [line width=.4pt] (2.2,-0.5)-- (0.8,-0.5);
\draw [line width=.4pt] (0.8,-0.5)-- (0.8,0.);
\draw [line width=.4pt] (0.8,2.)-- (2.2,2.);
\draw [line width=.4pt] (2.2,2.)-- (2.2,2.5);
\draw [line width=.4pt] (2.2,2.5)-- (0.8,2.5);
\draw [line width=.4pt] (0.8,2.5)-- (0.8,2.);
\draw (0.65,0.15) node[anchor=north west] {\scriptsize $(12)\cdot x$};
\draw (0.7,2.6) node[anchor=north west] {\scriptsize $y\cdot (12)$};
\draw [->,line width=.4pt] (1.,0.) -- (1.,2.);
\draw [->,line width=.4pt] (2.,0.) -- (2.,2.);
\draw [->,line width=.4pt] (1.,-1.5) -- (1.,-0.5);
\draw [->,line width=.4pt] (1.5,-1.5) -- (1.5,-0.5);
\draw [->,line width=.4pt] (2.,-1.5) -- (2.,-0.5);
\draw (0.7,-1.4) node[anchor=north west] {\scriptsize $1$};
\draw (1.2,-1.4) node[anchor=north west] {\scriptsize $2$};
\draw (1.7,-1.4) node[anchor=north west] {\scriptsize $3$};
\draw [->,line width=.4pt] (1.,2.5) -- (1,3.5);
\draw [->,line width=.4pt] (2.,2.5) -- (2.,3.5);
\draw (0.7,4.1) node[anchor=north west] {\scriptsize $1$};
\draw (1.6,4.1) node[anchor=north west] {\scriptsize $2$};
\end{tikzpicture}\substack{\displaystyle ,\\ \vspace{3.5cm}}&
\begin{tikzpicture}[line cap=round,line join=round,>=triangle 45,x=0.7cm,y=0.7cm]
\clip(0.6,-2.1) rectangle (2.4,4.5);
\draw [line width=.4pt] (0.8,0.)-- (2.2,0.);
\draw [line width=.4pt] (2.2,0.)-- (2.2,-0.5);
\draw [line width=.4pt] (2.2,-0.5)-- (0.8,-0.5);
\draw [line width=.4pt] (0.8,-0.5)-- (0.8,0.);
\draw [line width=.4pt] (0.8,2.)-- (2.2,2.);
\draw [line width=.4pt] (2.2,2.)-- (2.2,2.5);
\draw [line width=.4pt] (2.2,2.5)-- (0.8,2.5);
\draw [line width=.4pt] (0.8,2.5)-- (0.8,2.);
\draw (0.7,0.15) node[anchor=north west] {\scriptsize $x\cdot (12)$};
\draw (1.2,2.55) node[anchor=north west] {\scriptsize $y$};
\draw [->,line width=.4pt] (1.,0.) -- (1.,2.);
\draw [->,line width=.4pt] (2.,0.) -- (2.,2.);
\draw [->,line width=.4pt] (1.,-1.5) -- (1.,-0.5);
\draw [->,line width=.4pt] (1.5,-1.5) -- (1.5,-0.5);
\draw [->,line width=.4pt] (2.,-1.5) -- (2.,-0.5);
\draw (0.7,-1.4) node[anchor=north west] {\scriptsize $1$};
\draw (1.2,-1.4) node[anchor=north west] {\scriptsize $2$};
\draw (1.7,-1.4) node[anchor=north west] {\scriptsize $3$};
\draw [->,line width=.4pt] (1.,2.5) -- (1,3.5);
\draw [->,line width=.4pt] (2.,2.5) -- (2.,3.5);
\draw (0.7,4.1) node[anchor=north west] {\scriptsize $1$};
\draw (1.7,4.1) node[anchor=north west] {\scriptsize $2$};
\end{tikzpicture}&
\substack{\displaystyle =\\ \vspace{3.5cm}}\begin{tikzpicture}[line cap=round,line join=round,>=triangle 45,x=0.7cm,y=0.7cm]
\clip(0.6,-2.1) rectangle (2.4,4.5);
\draw [line width=.4pt] (0.8,0.)-- (2.2,0.);
\draw [line width=.4pt] (2.2,0.)-- (2.2,-0.5);
\draw [line width=.4pt] (2.2,-0.5)-- (0.8,-0.5);
\draw [line width=.4pt] (0.8,-0.5)-- (0.8,0.);
\draw [line width=.4pt] (0.8,2.)-- (2.2,2.);
\draw [line width=.4pt] (2.2,2.)-- (2.2,2.5);
\draw [line width=.4pt] (2.2,2.5)-- (0.8,2.5);
\draw [line width=.4pt] (0.8,2.5)-- (0.8,2.);
\draw (1.2,0.05) node[anchor=north west] {\scriptsize $x$};
\draw (1.2,2.55) node[anchor=north west] {\scriptsize $y$};
\draw [->,line width=.4pt] (1.,0.) -- (1.,2.);
\draw [->,line width=.4pt] (2.,0.) -- (2.,2.);
\draw [->,line width=.4pt] (1.,-1.5) -- (1.,-0.5);
\draw [->,line width=.4pt] (1.5,-1.5) -- (1.5,-0.5);
\draw [->,line width=.4pt] (2.,-1.5) -- (2.,-0.5);
\draw (0.7,-1.4) node[anchor=north west] {\scriptsize $2$};
\draw (1.2,-1.4) node[anchor=north west] {\scriptsize $1$};
\draw (1.7,-1.4) node[anchor=north west] {\scriptsize $3$};
\draw [->,line width=.4pt] (1.,2.5) -- (1,3.5);
\draw [->,line width=.4pt] (2.,2.5) -- (2.,3.5);
\draw (0.7,4.1) node[anchor=north west] {\scriptsize $1$};
\draw (1.7,4.1) node[anchor=north west] {\scriptsize $2$};
\end{tikzpicture}\substack{\displaystyle ,\\ \vspace{3.5cm}}&
\begin{tikzpicture}[line cap=round,line join=round,>=triangle 45,x=0.7cm,y=0.7cm]
\clip(0.6,-2.1) rectangle (2.4,4.5);
\draw [line width=.4pt] (0.8,0.)-- (2.2,0.);
\draw [line width=.4pt] (2.2,0.)-- (2.2,-0.5);
\draw [line width=.4pt] (2.2,-0.5)-- (0.8,-0.5);
\draw [line width=.4pt] (0.8,-0.5)-- (0.8,0.);
\draw [line width=.4pt] (0.8,2.)-- (2.2,2.);
\draw [line width=.4pt] (2.2,2.)-- (2.2,2.5);
\draw [line width=.4pt] (2.2,2.5)-- (0.8,2.5);
\draw [line width=.4pt] (0.8,2.5)-- (0.8,2.);
\draw (1.2,0.05) node[anchor=north west] {\scriptsize $x$};
\draw (0.7,2.6) node[anchor=north west] {\scriptsize $(12)\cdot y$};
\draw [->,line width=.4pt] (1.,0.) -- (1.,2.);
\draw [->,line width=.4pt] (2.,0.) -- (2.,2.);
\draw [->,line width=.4pt] (1.,-1.5) -- (1.,-0.5);
\draw [->,line width=.4pt] (1.5,-1.5) -- (1.5,-0.5);
\draw [->,line width=.4pt] (2.,-1.5) -- (2.,-0.5);
\draw (0.7,-1.4) node[anchor=north west] {\scriptsize $1$};
\draw (1.2,-1.4) node[anchor=north west] {\scriptsize $2$};
\draw (1.7,-1.4) node[anchor=north west] {\scriptsize $3$};
\draw [->,line width=.4pt] (1.,2.5) -- (1,3.5);
\draw [->,line width=.4pt] (2.,2.5) -- (2.,3.5);
\draw (0.7,4.1) node[anchor=north west] {\scriptsize $1$};
\draw (1.7,4.1) node[anchor=north west] {\scriptsize $2$};
\end{tikzpicture}&
\substack{\displaystyle =\\ \vspace{3.5cm}}\begin{tikzpicture}[line cap=round,line join=round,>=triangle 45,x=0.7cm,y=0.7cm]
\clip(0.6,-2.1) rectangle (2.4,4.5);
\draw [line width=.4pt] (0.8,0.)-- (2.2,0.);
\draw [line width=.4pt] (2.2,0.)-- (2.2,-0.5);
\draw [line width=.4pt] (2.2,-0.5)-- (0.8,-0.5);
\draw [line width=.4pt] (0.8,-0.5)-- (0.8,0.);
\draw [line width=.4pt] (0.8,2.)-- (2.2,2.);
\draw [line width=.4pt] (2.2,2.)-- (2.2,2.5);
\draw [line width=.4pt] (2.2,2.5)-- (0.8,2.5);
\draw [line width=.4pt] (0.8,2.5)-- (0.8,2.);
\draw (1.2,0.05) node[anchor=north west] {\scriptsize $x$};
\draw (1.2,2.55) node[anchor=north west] {\scriptsize $y$};
\draw [->,line width=.4pt] (1.,0.) -- (1.,2.);
\draw [->,line width=.4pt] (2.,0.) -- (2.,2.);
\draw [->,line width=.4pt] (1.,-1.5) -- (1.,-0.5);
\draw [->,line width=.4pt] (1.5,-1.5) -- (1.5,-0.5);
\draw [->,line width=.4pt] (2.,-1.5) -- (2.,-0.5);
\draw (0.7,-1.4) node[anchor=north west] {\scriptsize $2$};
\draw (1.2,-1.4) node[anchor=north west] {\scriptsize $1$};
\draw (1.7,-1.4) node[anchor=north west] {\scriptsize $3$};
\draw [->,line width=.4pt] (1.,2.5) -- (1,3.5);
\draw [->,line width=.4pt] (2.,2.5) -- (2.,3.5);
\draw (0.7,4.1) node[anchor=north west] {\scriptsize $2$};
\draw (1.7,4.1) node[anchor=north west] {\scriptsize $1$};
\end{tikzpicture}\substack{\displaystyle ,\\ \vspace{3.5cm}}
\end{align*}
\vspace{-2cm} 

\noindent where $x\in X_{3,2}$ and $y\in X_{2,2}$. \vskip 0,2cm

\begin{example}\label{ex:trivialmodGrcycle}
Let us assume that $X$ is a trivial $\sym\times \sym^{op}$-module: for any $k,l\in  \N_0$,
for any $x\in X_{k,l}$, for any $(\sigma,\tau)\in \sym_l\times \sym_l$,
$\sigma\cdot x\cdot \tau=x$. The relations defining $\Gacirc(X)$ which boil down to  
\begin{equation} \label{eq:quotient_monadtriv}
     \sigma\cdot_{v_0}G\cdot_{v_0}\tau \otimes \left(\bigotimes_{v\in V(G)} x_v\right)
-G\otimes \left(\bigotimes_{v\in V(G)} x_v\right),
    \end{equation} 
    amount to  the identification of
two planar $X$-decorated graphs with the same underlying $X$-decorated graph. In this case  we recover 
the $\sym\times \sym^{op}$-module $\Gr(X)$. 
\end{example}

More generally, according to the relations defining $\Gacirc(X)$, if for any graph $G$, we choose a planar graph
$\overline{G}$, the underlying graph of which is some $G\in \Gr(X)(k,l)$, then   the set of 
graphs $\overline{G}(X)(k,l)$  is a basis of $\Gacirc(X)(k,l)$. As there is no canonical way to choose the graphs $\overline{G}$,
we prefer to consider $\Gacirc(X)$ as a quotient of $\PGr(X)$. 

Alongside the category $\Prop$ introduced in Definition \ref{deficatProP}, we now introduce a second category.
\begin{defi} \label{deficatssm}
		Let $\catssm$ denote the category of $\sym\times\sym^{op}$-modules:
		its objects are families $P=(P(k,l))_{(k,l)\in \N_0^2}$, where for any $(k,l)\in \N_0^2$,
		$P(k,l)$ is a $\sym_l\otimes \sym_k^{op}$-module; a morphism $\phi:P\longrightarrow Q$
		is a family $(\phi_{k,l})_{(k,l)\in \N^2}$, where for any $(k,l)\in \N_0^2$,
		$\phi_{k,l}:P(k,l)\longrightarrow Q(k,l)$ is a morphism of $\sym_l\otimes \sym_k^{op}$-modules. 
\end{defi}

To a $\sym\times \sym^{op}$-module $X$ in $\catssm$ we have assigned another  
$\sym\times \sym^{op}$-module $\Gacirc(P)$ in $\catssm$. One easily checks that a  morphism  
$\varphi:P\longrightarrow Q$   of $\sym\times \sym^{op}$-modules induces a morphism  of 
$\sym\times \sym^{op}$-modules 
\[\Gacirc(\varphi): \Gacirc(P)\longrightarrow \Gacirc(Q)\]
	defined by pull-back  on the decorations of the vertices of the graphs: 
	\begin{equation}\label{eq:Gammapullback}\Gacirc(\varphi)
	(G,d_G):=(G,\varphi \circ d_G).
	\end{equation}
In summary, we have proven the following
\begin{prop}\label{prop:functGamma}
	The   map  $\Gacirc:\catssm\longrightarrow\catssm$
		defines an endofunctor of the category   $\catssm$.
\end{prop}

Moreover, for any $\sym\times \sym^{op}$-module, $\Gacirc(X)$ is a ProP:

\begin{theo} \label{thm:prop_struct_gacirc}
Let $X$ be a $\sym\times \sym^{op}$-module. The vertical and horizontal concatenations of the ProP
$\PGr(X)$ induce a ProP structure on $\Gacirc(X)$. 
\end{theo}

\begin{proof}
We have to prove that if $P\in I$ and $H \in \PGr(X)$, then $P*H$, $H*P$, $H \circ P$ and $P\circ H$ (if these
vertical concatenations are possible) belong to $I$. 
We can restrict ourselves to the case $P=G-G'$, where $G$ is a $X$-decorated planar graph and $G'$ is obtained from
$G$ by the action of two permutations on a vertex $v$ of $G$. It is then immediate that $G'*H$ is obtained from
$G*H$ by the action of two permutations on a vertex $v$ of $G*H$, so that $G*H-G'*H=P*H\in I$.
Similarly, $H*P\in I$, and  $H\circ P$ and $P\circ H$ belong to $I$ if these vertical concatenations are possible.
So $\Gacirc$ inherits a structure of ProP from the structure of $\PGr$.  \end{proof}

\begin{defi}
For any $\sym\times \sym^{op}$-module $X$, the ProP $\Ga(X)$ is defined by
\[\Ga(X)=\frac{\PGrc(X)}{I\cap \PGrc(X)}.\]
As $\PGrc(X)$ is a sub-ProP of $\PGr(X)$, $\Ga(X)$ is  a sub-ProP of $\Gacirc(X)$.
\end{defi}
\begin{example} As in Example \ref{ex:trivialmodGrcycle},  if $X$ is a  trivial  $\sym\times \sym^{op}$ module,  we recover the ProP $\Grc(X)$. 
	\end{example}

We have seen (Theorem \ref{thm:prop_struct_gacirc}) that for any $\sym\times \sym^{op}$-module, $\Gacirc(X)$ has a ProP structure, and that the 
same holds for $\Ga(X)$. We can then lift $\Gacirc(X)$ and $\Ga(X)$ to functors between these categories:
\begin{prop}\label{prop:functGamma_2}
	The   maps  $\Gacirc:\catssm\longrightarrow\Prop$ and $\Ga:\catssm \longrightarrow \Prop$
		define  functors from the category   $\catssm$ to the category $\Prop$ of ProPs.

\end{prop}

\begin{proof}
Let $X$ and $Y$ be two $\sym\times \sym^{op}$-modules and  $\varphi:X\longrightarrow Y$     a morphism 
of $\sym\times \sym^{op}$-modules and let
 $ \PGr(\varphi):\PGr(P)\longrightarrow \PGr(Q)$ be its pullback defined by 
 \begin{equation}\label{eq:GRcpullback}
  \PGr(\varphi)(G,d_G):=(G,\varphi \circ d_G)
 \end{equation} 
 for any $G\in \PGr(P)$. As the structure of ProP of $\Gacirc(X)$ is combinatorially given by disjoint union and grafting of graphs,
$\PGr(\varphi)$ clearly defines a morphism of ProPs from $\PGr(X)$ to $\PGr(Y)$. 
 As $\varphi$ is a morphism of $\sym\times \sym^{op}$-modules, it follows that $\PGr(\varphi)$ sends the ideal 
defining $\Gacirc(X)$ to the ideal defining $\Gacirc(Y)$, hence it induces a morphism 
$\Gacirc(\varphi): \Gacirc(X)\longrightarrow \Gacirc(Y)$ of ProPs. A similar  proof  holds for $\Ga$. 
\end{proof}

\subsection{The ProP $\Ga(P)$ of graphs decorated by another ProP}\label{subsec:propdecprop}

The ProPs $\Ga(X)$ satisfy a universal property:
\begin{theo}\label{thm:freeness_GaX}
Let $X$ be a $\sym\times \sym^{op}$-module, $P$   a ProP and $\varphi:X\longrightarrow P$   a morphism
of $\sym\times \sym^{op}$-modules. There exists a unique morphism of ProPs  $\Phi: \Ga(X)\longrightarrow P$,
extending $\varphi$  so that the following diagramme commutes:
\[\xymatrix{X\ar[d]_{\iota} \ar[r]^{\varphi}&P\\
\Ga(X)\ar[ru]_{\Phi}&}\]
where $\iota: X \hookrightarrow \Ga(X)$      is the map that sends  an element $x$ of $X$ to the   $X$-decorated graph   $G(x)= (G_{k,l},d)$ with $d$ sending  the unique vertex of $G_{k,l}$ to $x$.

In other words, $\Ga(X)$ is the free ProP generated by the $\sym\times \sym^{op}$-module $X$.
\end{theo}
\begin{proof}
\textit{Uniqueness}. As a quotient of $\PGrci(X)$, the ProP $\Ga(X)$ is generated by graphs with 
only one vertex $v$, decorated by an element of $X$ respecting $i(v)$ and $o(v)$. 
Hence, such a morphism $\Phi$ is unique.

\textit{Existence}. As $\PGrci(X)$ is the free ProP generated by the space $X$, for any linear map
$\varphi:X\longrightarrow P$, there exists a unique morphism of TraPs $\overline{\Phi}:\PGrci(X)\longrightarrow P$,
extending $\varphi$. If $\varphi$ is a morphism of $\sym\times \sym^{op}$-modules, that 
any element of the form (\ref{eq:quotient_monad}) belongs to the kernel of $\overline{\Phi}$,
thanks to the compatibility of the concatenation products of $P$ and the action of the symmetric groups.
so $\overline{\Phi}$ induces a morphism $\Phi:\Ga(X)\longrightarrow P$. 
\end{proof}

\begin{cor}\label{cor:PdecGrc}
 Given a  ProP $P$, there is a canonical morphism of ProPs \[ \alpha_P:\Ga(P)\longrightarrow P\] 
 induced by the decoration.
\end{cor}

\begin{proof} This is a straightfoward consequence of Theorem \ref{thm:freeness_GaX}, with $\varphi=\mathrm{Id}_P$. 
\end{proof}
 
\begin{example}
Let $p\in P(3,2)$ and $q\in P(2,2)$. The four following graphs (which are equal in $\Ga(P)$)
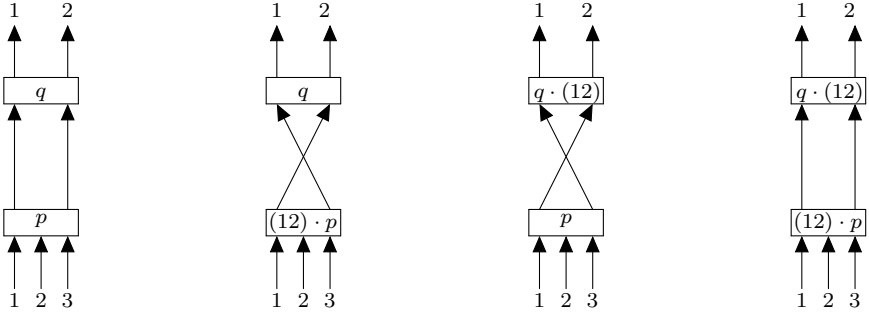
\begin{align*}
&\begin{tikzpicture}[line cap=round,line join=round,>=triangle 45,x=0.7cm,y=0.7cm]
\clip(0.6,-2.1) rectangle (2.4,4.5);
\draw [line width=.4pt] (0.8,0.)-- (2.2,0.);
\draw [line width=.4pt] (2.2,0.)-- (2.2,-0.5);
\draw [line width=.4pt] (2.2,-0.5)-- (0.8,-0.5);
\draw [line width=.4pt] (0.8,-0.5)-- (0.8,0.);
\draw [line width=.4pt] (0.8,2.)-- (2.2,2.);
\draw [line width=.4pt] (2.2,2.)-- (2.2,2.5);
\draw [line width=.4pt] (2.2,2.5)-- (0.8,2.5);
\draw [line width=.4pt] (0.8,2.5)-- (0.8,2.);
\draw (1.2,0.1) node[anchor=north west] {\scriptsize $p$};
\draw (1.2,2.5) node[anchor=north west] {\scriptsize $q$};
\draw [->,line width=.4pt] (1.,0.) -- (1.,2.);
\draw [->,line width=.4pt] (2.,0.) -- (2.,2.);
\draw [->,line width=.4pt] (1.,-1.5) -- (1.,-0.5);
\draw [->,line width=.4pt] (1.5,-1.5) -- (1.5,-0.5);
\draw [->,line width=.4pt] (2.,-1.5) -- (2.,-0.5);
\draw (0.7,-1.4) node[anchor=north west] {\scriptsize $1$};
\draw (1.2,-1.4) node[anchor=north west] {\scriptsize $2$};
\draw (1.7,-1.4) node[anchor=north west] {\scriptsize $3$};
\draw [->,line width=.4pt] (1.,2.5) -- (1,3.5);
\draw [->,line width=.4pt] (2.,2.5) -- (2.,3.5);
\draw (0.7,4.1) node[anchor=north west] {\scriptsize $1$};
\draw (1.7,4.1) node[anchor=north west] {\scriptsize $2$};
\end{tikzpicture}&
&\begin{tikzpicture}[line cap=round,line join=round,>=triangle 45,x=0.7cm,y=0.7cm]
\clip(0.6,-2.1) rectangle (2.4,4.5);
\draw [line width=.4pt] (0.8,0.)-- (2.2,0.);
\draw [line width=.4pt] (2.2,0.)-- (2.2,-0.5);
\draw [line width=.4pt] (2.2,-0.5)-- (0.8,-0.5);
\draw [line width=.4pt] (0.8,-0.5)-- (0.8,0.);
\draw [line width=.4pt] (0.8,2.)-- (2.2,2.);
\draw [line width=.4pt] (2.2,2.)-- (2.2,2.5);
\draw [line width=.4pt] (2.2,2.5)-- (0.8,2.5);
\draw [line width=.4pt] (0.8,2.5)-- (0.8,2.);
\draw (0.65,0.15) node[anchor=north west] {\scriptsize $(12)\cdot p$};
\draw (1.2,2.5) node[anchor=north west] {\scriptsize $q$};
\draw [->,line width=.4pt] (1.,0.) -- (2.,2.);
\draw [->,line width=.4pt] (2.,0.) -- (1.,2.);
\draw [->,line width=.4pt] (1.,-1.5) -- (1.,-0.5);
\draw [->,line width=.4pt] (1.5,-1.5) -- (1.5,-0.5);
\draw [->,line width=.4pt] (2.,-1.5) -- (2.,-0.5);
\draw (0.7,-1.4) node[anchor=north west] {\scriptsize $1$};
\draw (1.2,-1.4) node[anchor=north west] {\scriptsize $2$};
\draw (1.7,-1.4) node[anchor=north west] {\scriptsize $3$};
\draw [->,line width=.4pt] (1.,2.5) -- (1,3.5);
\draw [->,line width=.4pt] (2.,2.5) -- (2.,3.5);
\draw (0.7,4.1) node[anchor=north west] {\scriptsize $1$};
\draw (1.6,4.1) node[anchor=north west] {\scriptsize $2$};
\end{tikzpicture}
&\begin{tikzpicture}[line cap=round,line join=round,>=triangle 45,x=0.7cm,y=0.7cm]
\clip(0.6,-2.1) rectangle (2.4,4.5);
\draw [line width=.4pt] (0.8,0.)-- (2.2,0.);
\draw [line width=.4pt] (2.2,0.)-- (2.2,-0.5);
\draw [line width=.4pt] (2.2,-0.5)-- (0.8,-0.5);
\draw [line width=.4pt] (0.8,-0.5)-- (0.8,0.);
\draw [line width=.4pt] (0.8,2.)-- (2.2,2.);
\draw [line width=.4pt] (2.2,2.)-- (2.2,2.5);
\draw [line width=.4pt] (2.2,2.5)-- (0.8,2.5);
\draw [line width=.4pt] (0.8,2.5)-- (0.8,2.);
\draw (1.2,0.1) node[anchor=north west] {\scriptsize $p$};
\draw (0.7,2.6) node[anchor=north west] {\scriptsize $q\cdot (12)$};
\draw [->,line width=.4pt] (1.,0.) -- (2.,2.);
\draw [->,line width=.4pt] (2.,0.) -- (1.,2.);
\draw [->,line width=.4pt] (1.,-1.5) -- (1.,-0.5);
\draw [->,line width=.4pt] (1.5,-1.5) -- (1.5,-0.5);
\draw [->,line width=.4pt] (2.,-1.5) -- (2.,-0.5);
\draw (0.7,-1.4) node[anchor=north west] {\scriptsize $1$};
\draw (1.2,-1.4) node[anchor=north west] {\scriptsize $2$};
\draw (1.7,-1.4) node[anchor=north west] {\scriptsize $3$};
\draw [->,line width=.4pt] (1.,2.5) -- (1,3.5);
\draw [->,line width=.4pt] (2.,2.5) -- (2.,3.5);
\draw (0.7,4.1) node[anchor=north west] {\scriptsize $1$};
\draw (1.6,4.1) node[anchor=north west] {\scriptsize $2$};
\end{tikzpicture}&
&\begin{tikzpicture}[line cap=round,line join=round,>=triangle 45,x=0.7cm,y=0.7cm]
\clip(0.6,-2.1) rectangle (2.4,4.5);
\draw [line width=.4pt] (0.8,0.)-- (2.2,0.);
\draw [line width=.4pt] (2.2,0.)-- (2.2,-0.5);
\draw [line width=.4pt] (2.2,-0.5)-- (0.8,-0.5);
\draw [line width=.4pt] (0.8,-0.5)-- (0.8,0.);
\draw [line width=.4pt] (0.8,2.)-- (2.2,2.);
\draw [line width=.4pt] (2.2,2.)-- (2.2,2.5);
\draw [line width=.4pt] (2.2,2.5)-- (0.8,2.5);
\draw [line width=.4pt] (0.8,2.5)-- (0.8,2.);
\draw (0.65,0.15) node[anchor=north west] {\scriptsize $(12)\cdot p$};
\draw (0.7,2.6) node[anchor=north west] {\scriptsize $q\cdot (12)$};
\draw [->,line width=.4pt] (1.,0.) -- (1.,2.);
\draw [->,line width=.4pt] (2.,0.) -- (2.,2.);
\draw [->,line width=.4pt] (1.,-1.5) -- (1.,-0.5);
\draw [->,line width=.4pt] (1.5,-1.5) -- (1.5,-0.5);
\draw [->,line width=.4pt] (2.,-1.5) -- (2.,-0.5);
\draw (0.7,-1.4) node[anchor=north west] {\scriptsize $1$};
\draw (1.2,-1.4) node[anchor=north west] {\scriptsize $2$};
\draw (1.7,-1.4) node[anchor=north west] {\scriptsize $3$};
\draw [->,line width=.4pt] (1.,2.5) -- (1,3.5);
\draw [->,line width=.4pt] (2.,2.5) -- (2.,3.5);
\draw (0.7,4.1) node[anchor=north west] {\scriptsize $1$};
\draw (1.6,4.1) node[anchor=north west] {\scriptsize $2$};
\end{tikzpicture}
\end{align*}
are respectively sent to
\begin{align*}
&q\circ p,&& (q\cdot (12))\circ ((12)\cdot p),&& (q\cdot (12))\circ ((12)\cdot p),&& (q\cdot (12))\circ ((12)\cdot p),
\end{align*}
which are equal in $P$. The two  following graphs (which are equal in $\Ga(P)$)
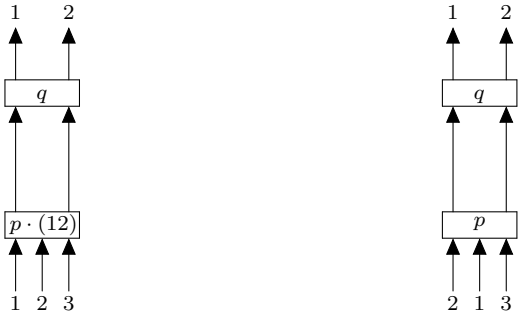
\begin{align*}
&\begin{tikzpicture}[line cap=round,line join=round,>=triangle 45,x=0.7cm,y=0.7cm]
\clip(0.6,-2.1) rectangle (2.4,4.5);
\draw [line width=.4pt] (0.8,0.)-- (2.2,0.);
\draw [line width=.4pt] (2.2,0.)-- (2.2,-0.5);
\draw [line width=.4pt] (2.2,-0.5)-- (0.8,-0.5);
\draw [line width=.4pt] (0.8,-0.5)-- (0.8,0.);
\draw [line width=.4pt] (0.8,2.)-- (2.2,2.);
\draw [line width=.4pt] (2.2,2.)-- (2.2,2.5);
\draw [line width=.4pt] (2.2,2.5)-- (0.8,2.5);
\draw [line width=.4pt] (0.8,2.5)-- (0.8,2.);
\draw (0.7,0.15) node[anchor=north west] {\scriptsize $p\cdot (12)$};
\draw (1.2,2.5) node[anchor=north west] {\scriptsize $q$};
\draw [->,line width=.4pt] (1.,0.) -- (1.,2.);
\draw [->,line width=.4pt] (2.,0.) -- (2.,2.);
\draw [->,line width=.4pt] (1.,-1.5) -- (1.,-0.5);
\draw [->,line width=.4pt] (1.5,-1.5) -- (1.5,-0.5);
\draw [->,line width=.4pt] (2.,-1.5) -- (2.,-0.5);
\draw (0.7,-1.4) node[anchor=north west] {\scriptsize $1$};
\draw (1.2,-1.4) node[anchor=north west] {\scriptsize $2$};
\draw (1.7,-1.4) node[anchor=north west] {\scriptsize $3$};
\draw [->,line width=.4pt] (1.,2.5) -- (1,3.5);
\draw [->,line width=.4pt] (2.,2.5) -- (2.,3.5);
\draw (0.7,4.1) node[anchor=north west] {\scriptsize $1$};
\draw (1.7,4.1) node[anchor=north west] {\scriptsize $2$};
\end{tikzpicture}&
&\begin{tikzpicture}[line cap=round,line join=round,>=triangle 45,x=0.7cm,y=0.7cm]
\clip(0.6,-2.1) rectangle (2.4,4.5);
\draw [line width=.4pt] (0.8,0.)-- (2.2,0.);
\draw [line width=.4pt] (2.2,0.)-- (2.2,-0.5);
\draw [line width=.4pt] (2.2,-0.5)-- (0.8,-0.5);
\draw [line width=.4pt] (0.8,-0.5)-- (0.8,0.);
\draw [line width=.4pt] (0.8,2.)-- (2.2,2.);
\draw [line width=.4pt] (2.2,2.)-- (2.2,2.5);
\draw [line width=.4pt] (2.2,2.5)-- (0.8,2.5);
\draw [line width=.4pt] (0.8,2.5)-- (0.8,2.);
\draw (1.2,0.1) node[anchor=north west] {\scriptsize $p$};
\draw (1.2,2.5) node[anchor=north west] {\scriptsize $q$};
\draw [->,line width=.4pt] (1.,0.) -- (1.,2.);
\draw [->,line width=.4pt] (2.,0.) -- (2.,2.);
\draw [->,line width=.4pt] (1.,-1.5) -- (1.,-0.5);
\draw [->,line width=.4pt] (1.5,-1.5) -- (1.5,-0.5);
\draw [->,line width=.4pt] (2.,-1.5) -- (2.,-0.5);
\draw (0.7,-1.4) node[anchor=north west] {\scriptsize $2$};
\draw (1.2,-1.4) node[anchor=north west] {\scriptsize $1$};
\draw (1.7,-1.4) node[anchor=north west] {\scriptsize $3$};
\draw [->,line width=.4pt] (1.,2.5) -- (1,3.5);
\draw [->,line width=.4pt] (2.,2.5) -- (2.,3.5);
\draw (0.7,4.1) node[anchor=north west] {\scriptsize $1$};
\draw (1.7,4.1) node[anchor=north west] {\scriptsize $2$};
\end{tikzpicture}
\end{align*}
are respectively sent to
\begin{align*}
&q\circ (p\cdot (12)),&&(q\circ p)\cdot (12),
\end{align*}
coincide in $P$. The two  following graphs (which are equal in $\Ga(P)$)
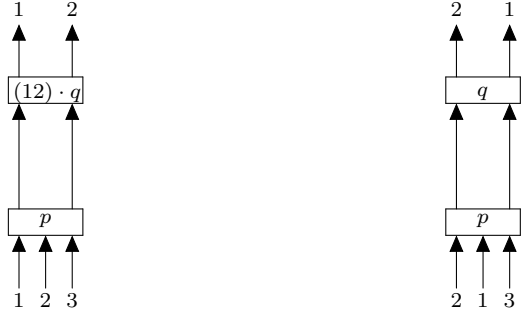
\begin{align*}
&\begin{tikzpicture}[line cap=round,line join=round,>=triangle 45,x=0.7cm,y=0.7cm]
\clip(0.6,-2.1) rectangle (2.4,4.5);
\draw [line width=.4pt] (0.8,0.)-- (2.2,0.);
\draw [line width=.4pt] (2.2,0.)-- (2.2,-0.5);
\draw [line width=.4pt] (2.2,-0.5)-- (0.8,-0.5);
\draw [line width=.4pt] (0.8,-0.5)-- (0.8,0.);
\draw [line width=.4pt] (0.8,2.)-- (2.2,2.);
\draw [line width=.4pt] (2.2,2.)-- (2.2,2.5);
\draw [line width=.4pt] (2.2,2.5)-- (0.8,2.5);
\draw [line width=.4pt] (0.8,2.5)-- (0.8,2.);
\draw (1.2,0.1) node[anchor=north west] {\scriptsize $p$};
\draw (0.7,2.6) node[anchor=north west] {\scriptsize $(12)\cdot q$};
\draw [->,line width=.4pt] (1.,0.) -- (1.,2.);
\draw [->,line width=.4pt] (2.,0.) -- (2.,2.);
\draw [->,line width=.4pt] (1.,-1.5) -- (1.,-0.5);
\draw [->,line width=.4pt] (1.5,-1.5) -- (1.5,-0.5);
\draw [->,line width=.4pt] (2.,-1.5) -- (2.,-0.5);
\draw (0.7,-1.4) node[anchor=north west] {\scriptsize $1$};
\draw (1.2,-1.4) node[anchor=north west] {\scriptsize $2$};
\draw (1.7,-1.4) node[anchor=north west] {\scriptsize $3$};
\draw [->,line width=.4pt] (1.,2.5) -- (1,3.5);
\draw [->,line width=.4pt] (2.,2.5) -- (2.,3.5);
\draw (0.7,4.1) node[anchor=north west] {\scriptsize $1$};
\draw (1.7,4.1) node[anchor=north west] {\scriptsize $2$};
\end{tikzpicture}&
&\begin{tikzpicture}[line cap=round,line join=round,>=triangle 45,x=0.7cm,y=0.7cm]
\clip(0.6,-2.1) rectangle (2.4,4.5);
\draw [line width=.4pt] (0.8,0.)-- (2.2,0.);
\draw [line width=.4pt] (2.2,0.)-- (2.2,-0.5);
\draw [line width=.4pt] (2.2,-0.5)-- (0.8,-0.5);
\draw [line width=.4pt] (0.8,-0.5)-- (0.8,0.);
\draw [line width=.4pt] (0.8,2.)-- (2.2,2.);
\draw [line width=.4pt] (2.2,2.)-- (2.2,2.5);
\draw [line width=.4pt] (2.2,2.5)-- (0.8,2.5);
\draw [line width=.4pt] (0.8,2.5)-- (0.8,2.);
\draw (1.2,0.1) node[anchor=north west] {\scriptsize $p$};
\draw (1.2,2.5) node[anchor=north west] {\scriptsize $q$};
\draw [->,line width=.4pt] (1.,0.) -- (1.,2.);
\draw [->,line width=.4pt] (2.,0.) -- (2.,2.);
\draw [->,line width=.4pt] (1.,-1.5) -- (1.,-0.5);
\draw [->,line width=.4pt] (1.5,-1.5) -- (1.5,-0.5);
\draw [->,line width=.4pt] (2.,-1.5) -- (2.,-0.5);
\draw (0.7,-1.4) node[anchor=north west] {\scriptsize $2$};
\draw (1.2,-1.4) node[anchor=north west] {\scriptsize $1$};
\draw (1.7,-1.4) node[anchor=north west] {\scriptsize $3$};
\draw [->,line width=.4pt] (1.,2.5) -- (1,3.5);
\draw [->,line width=.4pt] (2.,2.5) -- (2.,3.5);
\draw (0.7,4.1) node[anchor=north west] {\scriptsize $2$};
\draw (1.7,4.1) node[anchor=north west] {\scriptsize $1$};
\end{tikzpicture}
\end{align*}
are respectively sent to
\begin{align*}
&((12)\cdot q)\circ p,&& (12)\cdot(q\circ p),
\end{align*}
which are equal in $P$. 
\end{example}
Recall from Definition \ref{def:prop} that a ProP is a  $\sym\times \sym^{op}$-module.
Composing with the forgetful functor $F:\Prop\longrightarrow \catssm$  endofunctors
$\Gacirc\circ F$ and $\Ga\circ F$ of the category $\Prop$, which we denote also by 
$\Gacirc$ and $\Ga$ with a  slight  abuse of   notations.

\begin{prop} \label{prop:functGrc}
The maps $\alpha_P$ defined in Corollary \ref{cor:PdecGrc} give a natural transformation 
from the identity endofunctor of $\Prop$ to the endofunctor $\Ga$,
that is to say, for any morphism  $\varphi:P\longrightarrow Q$ of ProPs, the following diagram commutes:
\begin{align*}
&\xymatrix{\Ga(P)\ar[r]^{\Ga(\varphi)} \ar[d]_{\alpha_P}&\Ga(Q)\ar[d]^{\alpha_Q}\\
P\ar[r]_{\varphi}&Q}
\end{align*}
\end{prop}

\begin{proof} Since $\Ga(\varphi)$, $\alpha_P$,  $\alpha_Q$ and $\varphi$ are morphisms of ProPs,
$\alpha_Q\circ \Ga(\varphi)$ and $\varphi \circ \alpha_P$ are morphisms of ProPs.
As $\Ga(P)$ is generated by classes of graphs with only one vertex, it is enough to prove that 
$\alpha_Q\circ \Ga(\varphi)$ and $\varphi \circ \alpha_P$ coincide on such graphs.
Let us consider the planar graph $G_p=PG_{k,l}$, with its unique vertex decorated by $p\in P(k,l)$. 
Then, if $\overline{G}_p$ is the class of $G_p$ in $\Ga(X)$:
\begin{align*}
\alpha_Q\circ \Ga(\varphi)(\overline{G}_p)&=\alpha_Q(\overline{G}_{\varphi(p)})=\varphi(p)=
\varphi \circ \alpha_P(\overline{G}_p).
\end{align*}
So $\alpha_Q\circ \Ga(\varphi)=\varphi \circ \alpha_P$.
\end{proof}

\subsection{The case of $\Hom^c_V$} \label{subsec:alg_over_props}

Specialising the results of the previous Subsection to  $Q:=\Hom^c_V$ for some Fr\'echet nuclear topological vector space $V$ leads us to 
algebras over ProPs, see e.g. \cite{Markl}.
\begin{defi}
	\label{def:Palgebra}
	 A Fr\'echet nuclear topological vector space $V$ is an algebra over a ProP $P$ or a $P$-algebra if  there is a representation  
	 \[\varphi: P\longrightarrow \Hom^c_V,\]  of the ProP $P$ on the vector space $V$ i.e. if $\varphi$   is a morphism of ProPs. 
\end{defi}
\begin{remark}
In the literature of ProPs, the ${\Hom_V}$ ProP consists of the algebraic counterpart of our ${\Hom_V}^c$. 
\end{remark}
\begin{remark} Algebras over ProPs arise in Segal's axiomatic approach to conformal field theory 
	(CFT) \cite{Segal01}, by which a CFT is viewed  as  an algebra over the Segal ProP. A CFT is viewed as as  an algebra over the Segal ProP in \cite{Ionescu}, where the author claims that Feynman rules of a 
	given QFT, may be presented functorially as an algebra over the corresponding Feynman ProP.
\end{remark}
Applying Corollary \ref{cor:PdecGrc} to $P=\Hom^c_V$ and 
$\varphi =\mathrm{Id}\vert_{\Hom^c_V}$ yields:
\begin{cor}\label{cor:homGrc}
	A topological vector space $V$ has a canonical    algebra structure over $\Ga(\Hom^c_V)$ given by the canonical morphism of ProP 
	\[\alpha_V:\Ga(\Hom^c_V)\longrightarrow \Hom^c_V.\] 
\end{cor}
Proposition \ref{prop:functGrc} applied to $Q=\Hom^c_V$ yields the following statement.
\begin{cor} \label{cor:lift_alg_over_Prop}
 Let $P$ be a ProP and $V$ an algebra over $P$ given by a ProP-morphism\hfill \break\noindent $\varphi:P\longrightarrow\Hom^c_V$. Then $V$ also 
 canonically has the structure of an algebra over $\Ga(P)$ given by the map $\alpha_V\circ\Ga(\varphi) =\varphi\circ \alpha_P$.
\end{cor}

\section{Traces and Permutations (TraPs)}

This section is dedicated to TraPs, the other main protagonist of the paper. As for ProP, the main objects of interests in the category 
of TraP will be the TraP of  graphs $\Gr$ together with its variants and the TraP $\Hom_V^c$ of continuous morphisms on a Fr\'echet nuclear space $V$.

\subsection{The category of TraPs}

\begin{defi} \label{defi:Trap}
A \textbf{TraP} is a family $(P(k,l))_{k,l\geqslant 0}$ of vector spaces, equipped with the following structures:
\begin{enumerate}
\item For any $k,l\in  \N_0$, $P(k,l)$ is a $\sym_l\otimes \sym_k^{op}$-module.
\item For any $k,l,k',l'\in  \N_0$, there is a map
\begin{align*}
*:&\left\{\begin{array}{rcl}
P(k,l)\otimes P(k',l')&\longrightarrow&P(k+k',l+l')\\
p\otimes p'&\longrightarrow&p*p',
\end{array}\right.
\end{align*}
called the  \textbf{horizontal concatenation}, such that:
\begin{enumerate}
\item (Associativity). For any $(k,l,k',l',k'',l'')\in  \N_0^6$, for any $(p,p',p'')\in P(k,l)\times P(k',l')\times P(k'',l'')$,
\[(p*p')*p''=p*(p'*p'').\]
\item (Unity). There exists $I_0\in P(0,0)$ such that for any $(k,l)\in  \N_0^2$, for any $p\in P(k,l)$,
\[I_0*p=p*I_0=p.\]
\item (Compatibility with the symmetric actions).
For any $(k,l,k',l')\in  \N_0^4$, for any $(p,p')\in P(k,l)\times P(k',l')$,
for any $(\sigma,\tau,\sigma',\tau')\in \sym_l\times \sym_k\times \sym_{l'} \times \sym_{k'}$,
\[(\sigma\cdot p\cdot \tau)*(\sigma'\cdot p'\cdot \tau')
=(\sigma \otimes \sigma')\cdot (p*p')\cdot (\tau\otimes \tau').\]
\item (Commutativity). For any $(k,,k',l')\in  \N_0^4$, For any $p\in P(k,l)$, $p'\in P(k',l')$,
\[c_{l,l'}\cdot (p*p')=(p'*p)\cdot c_{k,k'},\]
where $c_{k,k'}$ and $c_{l,l'}$ are defined by (\ref{defcmn}).
\end{enumerate}
\item For any $k,l\geqslant 1$, for any $i\in [k]$, $j\in [l]$, there is a map
\begin{equation}\label{eq:partialtrace}
t_{i,j}:\left\{\begin{array}{rcl}
P(k,l)&\longrightarrow&P(k-1,l-1) \\
p&\longrightarrow&t_{i,j}(p),
\end{array}\right.
\end{equation}
called the \textbf{partial trace map}, such that:
\begin{enumerate}
\item (Commutativity). For any $k,l\geqslant 2$, for any $i\in [k]$, $j\in [l]$, $i'\in [k-1]$, $j'\in [l-1]$,
\begin{align*}
t_{i',j'}\circ t_{i,j}&=\begin{cases}
t_{i-1,j-1}\circ t_{i',j'}\mbox{ if }i'<i,\: j'<j,\\
t_{i,j-1}\circ t_{i'+1,j'}\mbox{ if }i'\geqslant i,\: j'<j,\\
t_{i-1,j}\circ t_{i',j'+1}\mbox{ if }i'<i,\: j'\geqslant j,\\
t_{i,j}\circ t_{i'+1,j'+1}\mbox{ if }i'\geqslant i,\: j'\geqslant j.
\end{cases}
\end{align*}
\item (Compatibility with the symmetric actions). For any $k,l\geqslant 1$, for any $i\in [k]$, $j\in [l]$,
$\sigma \in \sym_l$, $\tau\in \sym_k$, for any $p\in P(k,l)$,
\[t_{i,j}(\sigma\cdot p\cdot \tau)=\sigma_j\cdot (t_{\tau(i),\sigma^{-1}(j)}(p))\cdot \tau_i,\]
with the following notation: if $\alpha\in \sym_n$ and $p\in [n]$, then $\alpha_p\in \sym_{n-1}$ is defined by
\begin{equation}
\label{defalphak} \alpha_p(k)=\begin{cases}
\alpha(k)\mbox{ if }k<\alpha^{-1}(p) \mbox{ and }\alpha(k)<p,\\
\alpha(k)-1\mbox{ if }k<\alpha^{-1}(p) \mbox{ and }\alpha(k)>p,\\
\alpha(k+1)\mbox{ if }k\geqslant \alpha^{-1}(p) \mbox{ and }\alpha(k)<p,\\
\alpha(k+1)-1\mbox{ if }k\geqslant \alpha^{-1}(p) \mbox{ and }\alpha(k)>p.
\end{cases}
\end{equation}
In other words, if we represent $\alpha$ by a word $\alpha_1\ldots \alpha_n$, then $\alpha_p$ is represented
by the word obtained by suppression of the letter $p$ in $\alpha_1\ldots \alpha_n$ 
and subtraction  of $1$ to all the letters $>p$.
\item (Compatibility with the horizontal concatenation). 
For any $k,l,k',l'\geq 1$, for any $i\in [k+l]$, $j\in [k'+l']$, for any $p\in P(k,l)$, $p'\in P(k',l')$:
\[t_{i,j}(p*p')=\begin{cases}
t_{i,j }(p)*p'\mbox{ if }i\leqslant k,\: j\leqslant l,\\
p* t_{i-k,j-l}(p')\mbox{ if }i>k,\: j>l.
\end{cases}\]
\item (Unit). There exists $I\in P(1,1)$ such that for any $k,l\geqslant 1$, for any $i\in [k+1]$, $j\in [l+1]$,
for any $p\in P(k,l)$:
\begin{align*}
t_{1,j}(I*p)&= (1,2,\ldots,j-1)\cdot p \mbox{ if }j\geqslant 2,\\
t_{i,1}(I*p)&= p\cdot (1,2,\ldots, i-1)^{-1}\mbox{ if }i\geqslant 2,\\
t_{k+1,j}(p*I)&=(j,j+1,\ldots,l)^{-1}\cdot p\mbox{ if }j\leqslant l,\\
t_{i,l+1}(p*I)&=p\cdot (i,i+1,\ldots,k)\mbox{ if }i\leqslant k.
\end{align*}
\end{enumerate}\end{enumerate}\end{defi}

\begin{remark}
\begin{enumerate}
\item We do not require that $t_{1,1}(I)=I_0$, hence the terminology partial trace map.
\item By commutativity of $*$, for any $p\in P(0,0)$, for any $(k,l)\in  \N_0^2$, for any $p'\in P(k,l)$:
\[p*p'=p'*p,\]
since $c_{0,k}=\mathrm{Id}_{[k]}$.
\end{enumerate}
\end{remark}

\begin{remark}
     Our notion of TraP is an axiomatised version of Merkulov's notion of \textbf{wheeled ProPs} introduced in \cite{Merkulov2006}. The link between 
     TraPs and wheeled ProPs will be made in Section \ref{subsec:Trap_wProPs} , Corollary \ref{proptroptogamma}.
     
     Our approach mainly differs from  Merkulov's categorical approach in that it comprises  units. Units of wheeled ProPs 
     are mentioned in \cite[Remark 2.3.1]{Merkulov2010} but their axioms are not explicitly written down in the literature. Our axiomatic approach 
     is tailored to address analytic issues  regarding products of singularities. This axiomatic approach  allows us to give a simple 
     definition of quasi-TraPs in Section \ref{subsec:partial}, a notion that seem absent in previous works on wheeled ProPs. However,
     the categorical approach seems better suited for classification problems, e.g. regarding the solutions of the master equation in the BV formalism 
     \cite{Merkulov2009,Merkulov2010}.
    \end{remark}

\begin{lemma}\label{lemmeaxiomessimples}
Let $P=(P(k,l))_{k,l\in  \N_0}$ be a $\sym\otimes \sym^{op}$-module, equipped with a horizontal concatenation $*$
satisfying axioms 2. (a)-(d), and with maps $t_{i,j}$ satisfying axioms 3. (a)-(b). 
\begin{enumerate}
\item We assume that for any $k,l,k',l'\geqslant 1$, for any $p\in P(k,l)$, $p'\in P(k',l')$,
\[t_{1,1}(p*p')=t_{1,1}(p)*p'.\]
Then axiom 3.(c) is satisfied.
\item We assume for for any $k,l\geqslant 1$, for any $p\in P(k,l)$,
\[t_{1,2}(I*p)=p.\]
Then axiom 3.(d) is satisfied.
\end{enumerate}\end{lemma}

\begin{proof} 
1. Let $p\in P(k,l)$ and $p'\in P(k',l')$. Let us take $i\in [k+l]$, $j\in [k'+l']$, consider the transpositions $\sigma=(1,j)$ and $\tau=(1,i)$, with the convention $(1,1)=\mathrm{Id}$.
If $i\leqslant k$ and $j\leqslant l$,  then:
\begin{align*}
t_{i,j}(p*p')&=t_{i,j}(\sigma^2\cdot( p*p')\cdot \tau^2)\\
&=\sigma_j\cdot t_{1,1}(\sigma\cdot (p*p')\cdot \tau)\cdot \tau_i\\
&=\sigma_j\cdot(t_{1,1}((\sigma\cdot p\cdot \tau)*p')\cdot \tau_i\\
&=\sigma_j \cdot (t_{1,1}(\sigma\cdot p\cdot \tau)*p')\cdot \tau_i\\
&=(\sigma_j\cdot (t_{1,1}(\sigma\cdot p\cdot \tau)\cdot \tau_i)*p'\\
&=t_{i,j}(p)*p'.
\end{align*}
If $i>k$ and $j>l$, using $c_{m,n}^{-1}=c_{n,m}$:
\begin{align*}
t_{i,j}(p*p')&=t_{i,j}(c_{l',l}\cdot (p'*p)\cdot c_{k,k'})\\
&=(c_{l',l})_j\cdot t_{i-k,j-l}(p'*p)\cdot (c_{k,k'})_i\\
&=c_{l'-1,l}\cdot (t_{i-k,j-l}(p')*p)\cdot c_{k,k'-1}\\
&=p*t_{i-k,j-l}(p').
\end{align*}

2. Let us take $j\geqslant 2$.
\begin{align*}
t_{1,j}(I*p)&=t_{1,j}((2,j)^2\cdot(I*p))\\
&=(2,\ldots,j-1)\cdot t_{1,2}((2,j)\cdot (I*p))\\
&=(2,\ldots,j-1)\cdot t_{1,2}(I*(1,j-1)\cdot p))\\
&=(2,\ldots,j-1)\cdot((1,j-1)\cdot p)\\
&=(2,\ldots,j-1)(1,j-1)\cdot p\\
&=(1,\ldots,j-1)\cdot p.
\end{align*}
The three other relations are proved in the same way. 
\end{proof}
\begin{defi}      
     Let $P=(P(k,l))_{k,l\geq0}$ and $Q=(Q(k,l))_{k,l\geq0}$ be two TraPs with partial trace maps $(t_{i,j}^P)_{i,j\geq0}$ and $(t_{i,j}^Q)_{i,j\geq0}$ 
     respectively. A \textbf{morphism of TraPs} is a family $\phi=(\phi_{k,l})_{k,l\geq0}$ of linear 
     maps $\phi_{k,l}:P(k,l)\mapsto Q(k,l)$ which are morphism for the horizontal 
     concatenation, the actions of the symmetric groups and the partial trace maps. More precisely, for any $(k,l,m,n)\in \N_0^4$:
     \begin{itemize}
      \item $\forall (p,q)\in P(k,l)\times P(n,m),~\phi_{k+n,l+m}(p* q) = \phi_{k,l}(p)* \phi_{n,m}(q)$,
      \item $\forall (\sigma,p)\in\sym_l\times P(k,l),~\phi_{k,l}(\sigma.p)=\sigma.\phi_{k,l}(p)$,
      \item $\forall (p,\tau)\in P(k,l)\times\sym_k,~\phi_{k,l}(p.\tau)=\phi_{k,l}(p).\tau$.
      \item $\forall (p,i,j)\in P(k,l)\times [k]\times [l],~\phi_{k-1,l-1}(t^P_{i,j}(p))=t^Q_{i,j}(\phi_{k,l}(p))$.
     \end{itemize}
     With a slight  abuse of notations, we write $\phi(p)$ instead of $\phi_{k,l}(p)$ for $p\in P(k,l)$.
     In particular, TraPs form a category, which we denote by $\Trap$.
    \end{defi}

 \begin{remark}
The abuse of notation $t_{i,j}$ is legitimate since a full notation such as  $t_{i,j}^{k,l}$ is not necessary in practice. Indeed the indices  $
k$ and $l$ in $t_{i,j}(p)$ are entirely determined by  $p$ to which  $t_{i,j}$ is applied.

 More so,  $t_{i,j}$ does not strongly depend on $k$ and $l$: indeed, let  $f:P(k,l)\longrightarrow P(k+1,l+1)$ be the map that sends $p$ to  $p*I$ (for 
 the TraP of linear morphisms, this is the tensorisation by    $\mathrm{Id}$), then for  $i\in [k]$ and $j\in [l]$, we have
 \[t_{i,j} \circ f(p)=f\circ t_{i,j}(p),\] which is the axiom 3.(c).
\end{remark}

\begin{lemma}\label{lemmemorphismes}
Let $P$ and $Q$ be two TraPs and $\phi:P\longrightarrow Q$ be a map. We assume that:
\begin{enumerate}
\item For any $(k,l)\in  \N_0^2$, for any $(\sigma,\tau) \in \sym_l\times\sym_k$, for any $x\in P(k,l)$,
\[\phi(\sigma\cdot x\cdot \tau)=\sigma\cdot \phi(x)\cdot \tau.\]
\item For any $k,l\geqslant 1$, for any $x\in P(k,l)$,
\[t_{1,1}\circ \phi(x)=\phi\circ t_{1,1}(x).\]
\end{enumerate}
Then for any $k,l\geqslant 1$, for any $(i,j)\in [k]\times [l]$, for any $x\in P(k,l)$,
\[t_{i,j}\circ \phi(x)=\phi\circ t_{i,j}(x).\]
\end{lemma}

\begin{proof} If $i\in [k]$, $j\in [l]$, and $x\in P(k,l)$:
\begin{align*}
\phi\circ t_{i,j}(x)&=\phi\circ t_{i,j}((1,j)^2\cdot x\cdot (1,i)^2)\\
&=\phi((1,j)\cdot t_{1,1}((1,j)\cdot x\cdot (1,i))\cdot (1,i))\\
&=(1,j)\cdot  \phi\circ t_{1,1}((1,j)\cdot x\cdot (1,i))\cdot (1,i)\\
&=(1,j)\cdot  t_{1,1}\circ \phi ((1,j)\cdot x\cdot (1,i))\cdot (1,i)\\
&=t_{i,j}((1,j)\cdot \phi((1,j)\cdot x\cdot (1,i))\cdot (1,i))\\
&=t_{i,j}\circ \phi(x),
\end{align*}
with the convention $(1,k)=\mathrm{Id}$ if $k=1$. \end{proof}
In particular, to show that a collection of linear maps between two TraPs preserving 
the horizontal concatenation and the actions of the symmetry group is a morphism of TraP, it is enough to check 
the properties of Lemma \ref{lemmemorphismes}.

\subsection{The TraP $\Hom^c_V$}

We start with the TraP version of the ProP of linear morphisms of section \ref{sec:propHom_finiteDim}.
\begin{prop} \label{prop:Trap_fin_dim_VS}
	Let $V$ be a finite dimensional vector space and $V^*$ its algebraic dual. Then for any $(k,l)\in  \N_0^2$:
	\[\Hom_V(k,l)=\Hom(V^{\otimes k},V^{\otimes l})\simeq V^{*\otimes k}\otimes V^{\otimes l}.\]
\ $\sym_l\otimes \sym_k^{op}$   acts on the ProP $\Hom_V$   as readily described in   Proposition-Definition \ref{defi:Hom_V}.
We shall make some abuse of notation setting $f_1\cdots f_k:= f_1\otimes\cdots \otimes f_k\in V^{*\otimes k}$ and $v_1\cdots v_l:= v_1\otimes \cdots \otimes v_l\in V^l$. We equip $V^{*\otimes k}\otimes V^{\otimes l}$   with a horizontal concatenation:
	\[(f_1\ldots f_k\otimes v_1\ldots v_l)*(f'_1\ldots f'_{k'}\otimes v'_1\ldots v'_{l'})
	=f_1\ldots f_k f'_1\ldots f'_{k'}\otimes v_1\ldots v_l v'_1\ldots v'_{l'},\]
	and partial trace maps:
	\[t_{i,j}(f_1\ldots f_k\otimes v_1\ldots v_l)=f_i(v_j)f_1\ldots f_{i-1}f_{i+1}\ldots f_k
	\otimes v_1\ldots v_{j-1}v_{j+1}\ldots v_l\]
	(with obvious abuses of notations). These make $\Hom_V$ a TraP.
\end{prop}

\begin{proof}
	Properties 2.(a)-(d) are trivially satisfied, with $I_0=1\in \K=V^{\otimes 0}\otimes V^{*\otimes 0}$.
	Property 3.(a) is direct.  Let us prove Property 3. (b). 
	\begin{align*}
	t_{i,j}(\sigma\cdot f_1\ldots f_k\otimes v_1\ldots v_l \cdot \tau)
	&=t_{i,j}(f_{\tau(1)}\ldots f_{\tau(k)}\otimes v_{\sigma^{-1}(1)}\ldots v_{\sigma^{-1}(l)})\\
	&=f_{\tau(i)}(v_{\sigma^{-1}(j)})
	f_{\tau(1)}\ldots f_{\tau(i-1)} f_{\tau(i+1)}\ldots f_{\tau(k)}\\
	&\otimes v_{\sigma^{-1}(1)}\ldots  v_{\sigma^{-1}(j-1)} v_{\sigma^{-1}(j+1)}\ldots v_{\sigma^{-1}(l)}\\
	&=\sigma_j\cdot t_{\tau(i),\sigma^{-1}(j)}(f_1\ldots f_k\otimes v_1\ldots v_l)\cdot \tau_i.
	\end{align*}
	Property 3.(c) is straightforward.  Let us prove property 3.(d) with the help of Lemma \ref{lemmeaxiomessimples}.
	Let us fix $(e_i)_{i\in I}$ a basis of $V$, then $(e_i^*)_{i\in I}$ is a basis of $V^*$ and 
 the identity map $I=\sum_{i\in I} e_i^* \otimes e_i,  $  acts as follows, $ I(v)= \sum_{i\in I} e_i^*(v)e_i=v$  for all $v\in V$.
  Then:
	\begin{align*}
	t_{1,2}(I*f_1\ldots f_k\otimes v_1\ldots v_l)
	&=\sum_{i\in I} t_{1,2}(e_i^*f_1\ldots f_k\otimes e_iv_1\ldots v_l)\\
	&=\sum_{i\in I} f_1\ldots f_k\otimes e_ie_i^*(v_1)v_2\ldots v_l\\
	&= f_1\ldots f_k\otimes I(v_1)v_2\ldots v_l\\
	&=f_1\ldots f_k\otimes v_1\ldots v_l. 
	\end{align*}
	So $\Hom_V$ is a TraP.
\end{proof}

\begin{remark}
	In this example of TraP, $t_{1,1}(I)=\dim(V)=\dim(V)I_0$.
\end{remark}
In order to generalise this construction to nuclear Fr\'echet spaces, we need to characterise the composition of linear morphisms of such 
	spaces.
\begin{lemma} \label{lem:compo_dual_pairing}
	Let $E_1,E_2$ be two Fr\'echet nuclear spaces and $E_3$ a Fr\'echet space. Then
	the composition of continuous morphisms $L_1:E_1\longrightarrow E_2$, $L_2:E_2\longrightarrow E_3$ amounts to a dual pairing.
\end{lemma} 
\begin{proof} Let $E_1, E_2, E_3$ be three topological spaces as in the statement. Then by \eqref{eq:E_prime_otimes_F}
	the identifications $\Hom^c (E_1, E_2)\simeq E_1'\hat \otimes E_2$ and 
	$\Hom^c (E_2, E_3)\simeq E_2'\hat \otimes E_3$ hold. For  $L_1=\sum_{i, j} u_i^{1*}\otimes u_j^2\in   \Hom^c(E_1,E_2)$, $L_2=\sum_{k, l} u_k^{2*}\otimes u_l^3 \in  \Hom(E_2,E_3)$ and $u\in E_1$, we have 
	\[L_2\circ L_1 (u)=L_2\left(\sum_{i, j} u_i^{1*}(u) u_j^2\right)= \sum_{k, l} \sum_{i, j} \,u_i^{1*}(u)  \,u_k^{2*}(u_j^2)\,  u_l^3 \] so that
	\[\Hom^c( E_1,E_3)\ni L_2\circ L_1 = \sum_{i, l} \left(\sum_{k, j}  u_k^{2*}(u_j^2)\right)\,   u_i^{1*}\otimes u_l^3 \in E_1^* \hat \otimes E_3. \qedhere\]
\end{proof}

Recall that, for a Fr\'echet nuclear space $V$, the ProP    $(\Hom_V^c(k, l))_{k,l\geq 0}$ introduced  in Subsection 
\ref{subsection:infinite_dim_prop}, Theorem \ref{thm:Hom_V_generalised}, 
reads:
\[\Hom_V^c(k,l)\simeq\left(V'\right)^{\widehat\otimes  k}\hat \otimes V^{\widehat\otimes  l}.\]
\begin{prop}\label{lem:TrHom}
	Let $V$ be a Fr\'echet nuclear space. The family  $(\Hom_V^c(k, l))_{k,l\geq 0}$   equipped with the 
	partial trace maps  in the sense of \eqref{eq:partialtrace} defined by
	\begin{align*}
	\mathrm{tr}_{i,j}&:\left\{\begin{array}{rcl}
	\Hom^c_V(k,l)&\longrightarrow&\Hom^c_V(k-1,l-1)\\
	(v_1^*\otimes \cdots \otimes v_k^*)\otimes (w_1\otimes  \cdots \otimes w_l ) &\longmapsto & \mathrm{tr}_{i,j}\left((v_1^*\otimes \cdots \otimes v_k^*)\otimes (w_1\otimes  \cdots \otimes w_l) \right)
	\end{array}\right.
	\end{align*}
	with $\mathrm{tr}_{i,j}\left(v_1^*\otimes \cdots \otimes v_k^*)\otimes (w_1\otimes  \cdots \otimes w_l \right)$ defined as
	\begin{equation*}
	v_i^*(w_j)\,  (v_1^* \otimes \cdots \otimes\,  {\widehat{v_i^*}}\,\otimes\cdots \otimes v_k^*)\, \otimes \, (w_1 \otimes \cdots \otimes\,  {\widehat{w_j}}\, \otimes \cdots  \otimes w_l) 
	\end{equation*}
	for  any $k,l\geqslant 1$, for any $i\in [k]$, $j\in [l]$, where $v_i^*(w_j)$ is the dual pairing, defines a TraP, with the topological tensor product as 
	horizontal concatenation.
\end{prop}
\begin{proof}Commutativity   {follows from the commutativity of the field $\K$}, compatibility with the symmetric actions and compatibility with the 
	horizontal concatenation are   shown as for the ProP $\Hom_V^c$. 
	The unit is the identity map $I\in V^*\widehat{\otimes} V$.
\end{proof}
\begin{example}
	With the notations of Remark \ref{ex:findimtensor}, $ \mathrm{tr}_{i,j}\left( \sum_{\vec I, \vec J} a_{\vec J}^{\vec I} \, e^{\vec J}  \otimes e_{ \vec I}\right)$ is of the form 
	$ \sum_{\vec I_i, \vec J_j} b_{\vec J_j}^{\vec I_i} \, e^{\vec J_j}  \otimes e_{ \vec I_i}$, where $\vec I_i= (i_1, \cdots, \hat i, \cdots i_k)$, $\vec J_j= (j_1, \cdots, \hat j, \cdots, j_l)$ and $b_{\vec J_j}^{\vec I_i}$ corresponds to the trace of the $n\times n$ matrix in the $(i, j)$ entries of $a_{\vec J}^{\vec I} $ with the other indices frozen.
\end{example}
\begin{example}
	Let $U$ be an open of $\R^n$. Example \ref{ex:infindimtensor1} and Equation 
	\eqref{eq:echange_dual_prod} imply that
	the family  $({\mathcal K}_U(k, l))_{k,l\geq 0}$, with 
	${\mathcal K}_U(k, l)= \left({\mathcal E}^\prime(U)\right)^{\hat \otimes k}\, \hat \otimes \,  {\mathcal E}(U)^{\hat \otimes l}$  defines a TraP.
\end{example}
\begin{example} \label{ex:KM}
	Let $X$ be a finite dimensional smooth manifold. Proposition 
	\ref{prop:fction_manifold_Frechet_nuclear} and Equation \eqref{eq:echange_dual_prod}
	imply that
	the family  $({\mathcal K}_X(k, l))_{k,l\geq 0}$,
	with ${\mathcal K}_X(k, l)= \left({\mathcal E}^\prime(X)\right)^{\hat \otimes k}\, \hat \otimes \,  {\mathcal E}(X)^{\hat \otimes l}$  defines a TraP.
\end{example}

\subsection{The TraP $\Gr$ of graphs} \label{subsection:generalised_graphs}

We now equip graphs and planar graphs with a TraP structure. We have already equipped
$\Gr$ and $\PGr$ with a structure of $\sym\times \sym^{op}$-modules and a horizontal concatenation,
which we leave untouched. Let us now define partial trace maps. 
Let $G\in \Gr(k,l)$, $1\leqslant i\leqslant k$ and $1\leqslant j\leqslant l$. 
We set $e_i=\alpha_G^{-1}(i)$,  $f_j=\beta_G^{-1}(j)$ and define $t_{i,j}(G)$ 
as the graph obtained by identifying the input of $e_i$ with the output $j$ of $f_j$. 
If $e_i\in I(G)$ and $f_j\in O(G)$, this creates
an edge in $E(G)$. This case is illustrated in the figure below. Otherwise, we create an edge in $I(G)$, or $O(G)$ or $IO(G)$ or in $L(G)$. In all these cases, we then reindex increasingly the inputs and the outputs of the obtained graph.

Graphically:
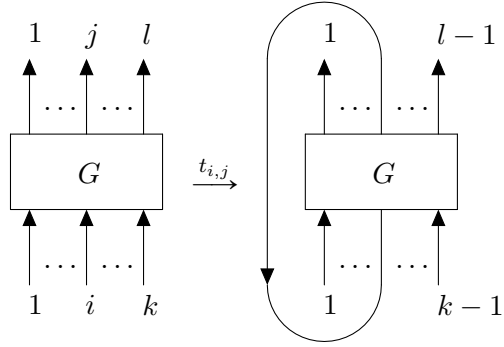
\begin{center}
\begin{tikzpicture}[line cap=round,line join=round,>=triangle 45,x=0.5cm,y=0.5cm]
\clip(-2.5,-4.5) rectangle (2.5,4.);
\draw [line width=0.4pt] (-2.,1.)-- (2.,1.);
\draw [line width=0.4pt] (2.,1.)-- (2.,-1.);
\draw [line width=0.4pt] (2.,-1.)-- (-2.,-1.);
\draw [line width=0.4pt] (-2.,-1.)-- (-2.,1.);
\draw [->,line width=0.4pt] (-1.5,1.) -- (-1.5,3.);
\draw [->,line width=0.4pt] (0.,1.) -- (0.,3.);
\draw [->,line width=0.4pt] (1.5,1.) -- (1.5,3.);
\draw [->,line width=0.4pt] (-1.5,-3.) -- (-1.5,-1.);
\draw [->,line width=0.4pt] (0.,-3.) -- (0.,-1.);
\draw [->,line width=0.4pt] (1.5,-3.) -- (1.5,-1.);
\draw (-0.5,0.5) node[anchor=north west] {$G$};
\draw (-1.8,-3) node[anchor=north west] {$1$};
\draw (-0.3,-3) node[anchor=north west] {$i$};
\draw (1.2,-3) node[anchor=north west] {$k$};
\draw (-1.4,-2.2) node[anchor=north west] {$\ldots$};
\draw (0.1,-2.2) node[anchor=north west] {$\ldots$};
\draw (-1.8,4.2) node[anchor=north west] {$1$};
\draw (-0.3,4.2) node[anchor=north west] {$j$};
\draw (1.2,4.2) node[anchor=north west] {$l$};
\draw (-1.4,2.) node[anchor=north west] {$\ldots$};
\draw (0.1,2.) node[anchor=north west] {$\ldots$};
\end{tikzpicture}
$\substack{\displaystyle \stackrel{t_{i,j}}{\longrightarrow}\\ \vspace{4cm}}$
\begin{tikzpicture}[line cap=round,line join=round,>=triangle 45,x=0.5cm,y=0.5cm]
\clip(-3.5,-4.5) rectangle (3.5,5.);
\draw [line width=0.4pt] (-2.,1.)-- (2.,1.);
\draw [line width=0.4pt] (2.,1.)-- (2.,-1.);
\draw [line width=0.4pt] (2.,-1.)-- (-2.,-1.);
\draw [line width=0.4pt] (-2.,-1.)-- (-2.,1.);
\draw [->,line width=0.4pt] (-1.5,1.) -- (-1.5,3.);
\draw [line width=0.4pt] (0.,1.) -- (0.,3.);
\draw [->,line width=0.4pt] (1.5,1.) -- (1.5,3.);
\draw [->,line width=0.4pt] (-1.5,-3.) -- (-1.5,-1.);
\draw [line width=0.4pt] (0.,-3.) -- (0.,-1.);
\draw [->,line width=0.4pt] (1.5,-3.) -- (1.5,-1.);
\draw (-0.5,0.5) node[anchor=north west] {$G$};
\draw (-1.8,-3) node[anchor=north west] {$1$};
\draw (1.2,-3) node[anchor=north west] {$k-1$};
\draw (-1.4,-2.2) node[anchor=north west] {$\ldots$};
\draw (0.1,-2.2) node[anchor=north west] {$\ldots$};
\draw (-1.8,4.2) node[anchor=north west] {$1$};
\draw (1.2,4.2) node[anchor=north west] {$l-1$};
\draw (-1.4,2.) node[anchor=north west] {$\ldots$};
\draw (0.1,2.) node[anchor=north west] {$\ldots$};
\draw [shift={(-1.5,-3.)},line width=0.4pt]  plot[domain=3.141592653589793:6.283185307179586,variable=\t]({1.*1.5*cos(\t r)+0.*1.5*sin(\t r)},{0.*1.5*cos(\t r)+1.*1.5*sin(\t r)});
\draw [shift={(-1.5,3.)},line width=0.4pt]  plot[domain=0.:3.141592653589793,variable=\t]({1.*1.5*cos(\t r)+0.*1.5*sin(\t r)},{0.*1.5*cos(\t r)+1.*1.5*sin(\t r)});
\draw [->,line width=0.4pt] (-3.,3.) -- (-3.,-3.);
\end{tikzpicture}
\vspace{-2cm}
\end{center}
A more rigorous definition is given in the appendix. A similar definition can be given for planar graphs,
by preserving the orders on incoming and outgoing edges of any vertex.

\begin{example} Let $G$ be the following graph:
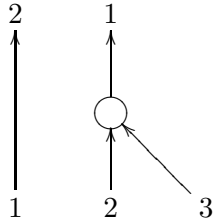
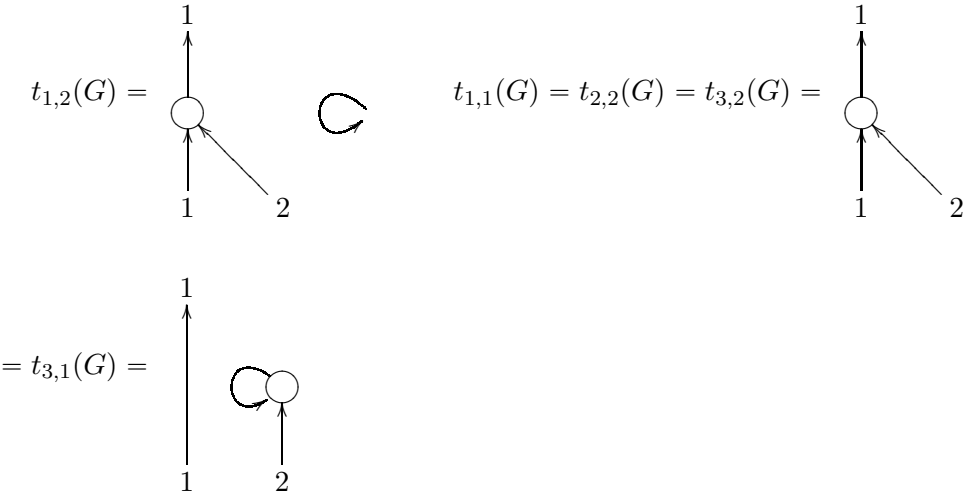
\[\xymatrix{2&1&\\
&\rond{}\ar[u]&\\
1\ar[uu]&2\ar[u]&3\ar[lu]}\]
Then:

\vspace{-2cm}

\begin{align*}
\substack{\vspace{2cm}\\ \displaystyle t_{1,2}(G)=\hspace{.3cm}}&\xymatrix{1&&\\
\rond{}\ar[u]&&\ar@(ul,dl)[]\\
1\ar[u]&2\ar[lu]&}&
\substack{\vspace{2cm}\\ \displaystyle t_{1,1}(G)=t_{2,2}(G)=t_{3,2}(G)=\hspace{.3cm}}
&\xymatrix{1&\\
\rond{}\ar[u]&\\
1\ar[u]&2\ar[lu]}\\[-5mm]
\substack{\vspace{2cm}\\ \displaystyle t_{2,1}(G)=t_{3,1}(G)=\hspace{.3cm}}
&\xymatrix{1&\\
&\rond{}\ar@(ul,dl)[]\\
1\ar[uu]&2\ar[u]}
\end{align*}
Note that $t_{1,2}$ creates a loop when applied on $G$. \end{example}

\begin{remark}
In particular, $t_{1,1}(I)$ is the graph $\grapheo$, which is essential for TraPs.
\end{remark}

\begin{prop}\label{prop:TrapGG}
$\Gr$ and $\PGr$, with the usual horizontal concatenation and this partial trace map, are TraPs.
\end{prop}

\begin{proof}
Properties 2.(a)-(d) are trivial. Let us give a graphical indication of the proof of Property 3.(a), when $i'<i$ and $j'<j$.

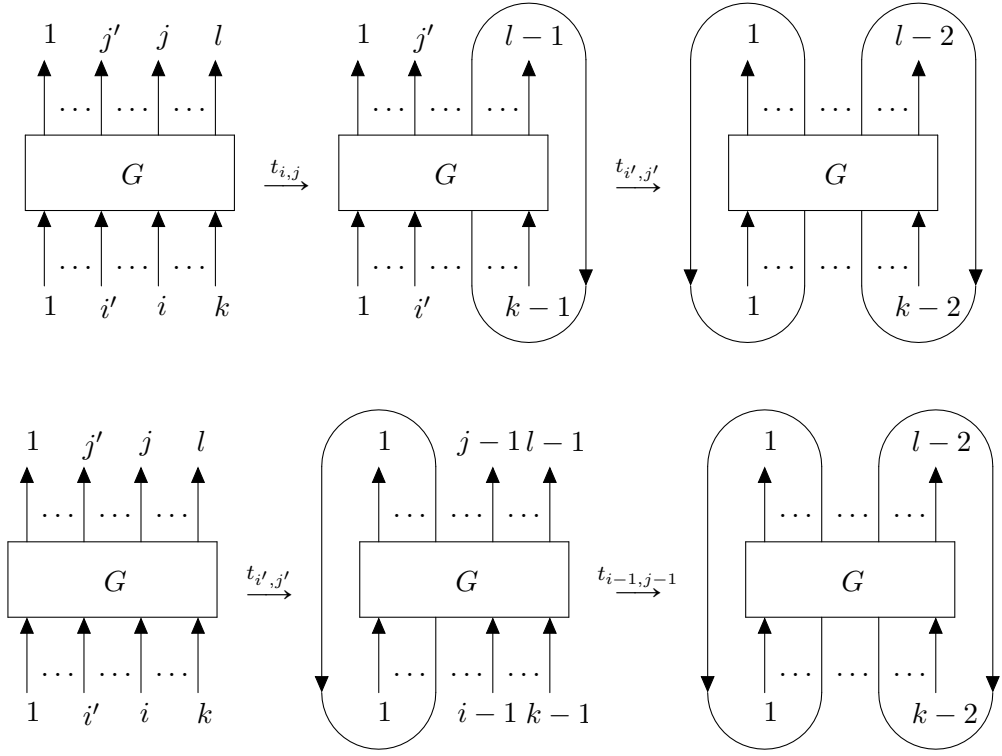
\begin{center}
\begin{tikzpicture}[line cap=round,line join=round,>=triangle 45,x=0.5cm,y=0.5cm]
\clip(-2.5,-4.5) rectangle (4,4.);
\draw [line width=0.4pt] (-2.,1.)-- (3.5,1.);
\draw [line width=0.4pt] (3.5,1.)-- (3.5,-1.);
\draw [line width=0.4pt] (3.5,-1.)-- (-2.,-1.);
\draw [line width=0.4pt] (-2.,-1.)-- (-2.,1.);
\draw [->,line width=0.4pt] (-1.5,1.) -- (-1.5,3.);
\draw [->,line width=0.4pt] (0.,1.) -- (0.,3.);
\draw [->,line width=0.4pt] (1.5,1.) -- (1.5,3.);
\draw [->,line width=0.4pt] (3.,1.) -- (3.,3.);
\draw [->,line width=0.4pt] (-1.5,-3.) -- (-1.5,-1.);
\draw [->,line width=0.4pt] (0.,-3.) -- (0.,-1.);
\draw [->,line width=0.4pt] (1.5,-3.) -- (1.5,-1.);
\draw [->,line width=0.4pt] (3.,-3.) -- (3.,-1.);
\draw (0.25,0.5) node[anchor=north west] {$G$};
\draw (-1.8,-3) node[anchor=north west] {$1$};
\draw (-0.3,-3) node[anchor=north west] {$i'$};
\draw (1.2,-3) node[anchor=north west] {$i$};
\draw (2.7,-3) node[anchor=north west] {$k$};
\draw (-1.4,-2.2) node[anchor=north west] {$\ldots$};
\draw (0.1,-2.2) node[anchor=north west] {$\ldots$};
\draw (1.6,-2.2) node[anchor=north west] {$\ldots$};
\draw (-1.8,4.2) node[anchor=north west] {$1$};
\draw (-0.3,4.2) node[anchor=north west] {$j'$};
\draw (1.2,4.2) node[anchor=north west] {$j$};
\draw (2.7,4.2) node[anchor=north west] {$l$};
\draw (-1.4,2.) node[anchor=north west] {$\ldots$};
\draw (0.1,2.) node[anchor=north west] {$\ldots$};
\draw (1.6,2.) node[anchor=north west] {$\ldots$};
\end{tikzpicture}
$\substack{\displaystyle \stackrel{t_{i,j}}{\longrightarrow}\\ \vspace{4cm}}$
\begin{tikzpicture}[line cap=round,line join=round,>=triangle 45,x=0.5cm,y=0.5cm]
\clip(-2.5,-4.5) rectangle (5.,5.);
\draw [line width=0.4pt] (-2.,1.)-- (3.5,1.);
\draw [line width=0.4pt] (3.5,1.)-- (3.5,-1.);
\draw [line width=0.4pt] (3.5,-1.)-- (-2.,-1.);
\draw [line width=0.4pt] (-2.,-1.)-- (-2.,1.);
\draw [->,line width=0.4pt] (-1.5,1.) -- (-1.5,3.);
\draw [->,line width=0.4pt] (0.,1.) -- (0.,3.);
\draw [line width=0.4pt] (1.5,1.) -- (1.5,3.);
\draw [->,line width=0.4pt] (3.,1.) -- (3.,3.);
\draw [->,line width=0.4pt] (-1.5,-3.) -- (-1.5,-1.);
\draw [->,line width=0.4pt] (0.,-3.) -- (0.,-1.);
\draw [,line width=0.4pt] (1.5,-3.) -- (1.5,-1.);
\draw [->,line width=0.4pt] (3.,-3.) -- (3.,-1.);
\draw (0.25,0.5) node[anchor=north west] {$G$};
\draw (-1.8,-3) node[anchor=north west] {$1$};
\draw (-0.3,-3) node[anchor=north west] {$i'$};
\draw (2.1,-3) node[anchor=north west] {$k-1$};
\draw (-1.4,-2.2) node[anchor=north west] {$\ldots$};
\draw (0.1,-2.2) node[anchor=north west] {$\ldots$};
\draw (1.6,-2.2) node[anchor=north west] {$\ldots$};
\draw (-1.8,4.2) node[anchor=north west] {$1$};
\draw (-0.3,4.2) node[anchor=north west] {$j'$};
\draw (2.1,4.2) node[anchor=north west] {$l-1$};
\draw (-1.4,2.) node[anchor=north west] {$\ldots$};
\draw (0.1,2.) node[anchor=north west] {$\ldots$};
\draw (1.6,2.) node[anchor=north west] {$\ldots$};
\draw [shift={(3.,-3.)},line width=0.4pt]  plot[domain=3.141592653589793:6.283185307179586,variable=\t]({1.*1.5*cos(\t r)+0.*1.5*sin(\t r)},{0.*1.5*cos(\t r)+1.*1.5*sin(\t r)});
\draw [shift={(3.,3.)},line width=0.4pt]  plot[domain=0.:3.141592653589793,variable=\t]({1.*1.5*cos(\t r)+0.*1.5*sin(\t r)},{0.*1.5*cos(\t r)+1.*1.5*sin(\t r)});
\draw [->,line width=0.4pt] (4.5,3.) -- (4.5,-3.);
\end{tikzpicture}
$\substack{\displaystyle \stackrel{t_{i',j'}}{\longrightarrow}\\ \vspace{4cm}}$
\begin{tikzpicture}[line cap=round,line join=round,>=triangle 45,x=0.5cm,y=0.5cm]
\clip(-3.5,-4.5) rectangle (5.,5.);
\draw [line width=0.4pt] (-2.,1.)-- (3.5,1.);
\draw [line width=0.4pt] (3.5,1.)-- (3.5,-1.);
\draw [line width=0.4pt] (3.5,-1.)-- (-2.,-1.);
\draw [line width=0.4pt] (-2.,-1.)-- (-2.,1.);
\draw [->,line width=0.4pt] (-1.5,1.) -- (-1.5,3.);
\draw [line width=0.4pt] (0.,1.) -- (0.,3.);
\draw [line width=0.4pt] (1.5,1.) -- (1.5,3.);
\draw [->,line width=0.4pt] (3.,1.) -- (3.,3.);
\draw [->,line width=0.4pt] (-1.5,-3.) -- (-1.5,-1.);
\draw [line width=0.4pt] (0.,-3.) -- (0.,-1.);
\draw [line width=0.4pt] (1.5,-3.) -- (1.5,-1.);
\draw [->,line width=0.4pt] (3.,-3.) -- (3.,-1.);
\draw (0.25,0.5) node[anchor=north west] {$G$};
\draw (-1.8,-3) node[anchor=north west] {$1$};
\draw (2.1,-3) node[anchor=north west] {$k-2$};
\draw (-1.4,-2.2) node[anchor=north west] {$\ldots$};
\draw (0.1,-2.2) node[anchor=north west] {$\ldots$};
\draw (1.6,-2.2) node[anchor=north west] {$\ldots$};
\draw (-1.8,4.2) node[anchor=north west] {$1$};
\draw (2.1,4.2) node[anchor=north west] {$l-2$};
\draw (-1.4,2.) node[anchor=north west] {$\ldots$};
\draw (0.1,2.) node[anchor=north west] {$\ldots$};
\draw (1.6,2.) node[anchor=north west] {$\ldots$};
\draw [shift={(-1.5,-3.)},line width=0.4pt]  plot[domain=3.141592653589793:6.283185307179586,variable=\t]({1.*1.5*cos(\t r)+0.*1.5*sin(\t r)},{0.*1.5*cos(\t r)+1.*1.5*sin(\t r)});
\draw [shift={(-1.5,3.)},line width=0.4pt]  plot[domain=0.:3.141592653589793,variable=\t]({1.*1.5*cos(\t r)+0.*1.5*sin(\t r)},{0.*1.5*cos(\t r)+1.*1.5*sin(\t r)});
\draw [->,line width=0.4pt] (-3.,3.) -- (-3.,-3.);
\draw [shift={(3.,-3.)},line width=0.4pt]  plot[domain=3.141592653589793:6.283185307179586,variable=\t]({1.*1.5*cos(\t r)+0.*1.5*sin(\t r)},{0.*1.5*cos(\t r)+1.*1.5*sin(\t r)});
\draw [shift={(3.,3.)},line width=0.4pt]  plot[domain=0.:3.141592653589793,variable=\t]({1.*1.5*cos(\t r)+0.*1.5*sin(\t r)},{0.*1.5*cos(\t r)+1.*1.5*sin(\t r)});
\draw [->,line width=0.4pt] (4.5,3.) -- (4.5,-3.);
\end{tikzpicture}

\vspace{-2cm}
\end{center}

\begin{center}
\begin{tikzpicture}[line cap=round,line join=round,>=triangle 45,x=0.5cm,y=0.5cm]
\clip(-2.5,-4.5) rectangle (4,4.);
\draw [line width=0.4pt] (-2.,1.)-- (3.5,1.);
\draw [line width=0.4pt] (3.5,1.)-- (3.5,-1.);
\draw [line width=0.4pt] (3.5,-1.)-- (-2.,-1.);
\draw [line width=0.4pt] (-2.,-1.)-- (-2.,1.);
\draw [->,line width=0.4pt] (-1.5,1.) -- (-1.5,3.);
\draw [->,line width=0.4pt] (0.,1.) -- (0.,3.);
\draw [->,line width=0.4pt] (1.5,1.) -- (1.5,3.);
\draw [->,line width=0.4pt] (3.,1.) -- (3.,3.);
\draw [->,line width=0.4pt] (-1.5,-3.) -- (-1.5,-1.);
\draw [->,line width=0.4pt] (0.,-3.) -- (0.,-1.);
\draw [->,line width=0.4pt] (1.5,-3.) -- (1.5,-1.);
\draw [->,line width=0.4pt] (3.,-3.) -- (3.,-1.);
\draw (0.25,0.5) node[anchor=north west] {$G$};
\draw (-1.8,-3) node[anchor=north west] {$1$};
\draw (-0.3,-3) node[anchor=north west] {$i'$};
\draw (1.2,-3) node[anchor=north west] {$i$};
\draw (2.7,-3) node[anchor=north west] {$k$};
\draw (-1.4,-2.2) node[anchor=north west] {$\ldots$};
\draw (0.1,-2.2) node[anchor=north west] {$\ldots$};
\draw (1.6,-2.2) node[anchor=north west] {$\ldots$};
\draw (-1.8,4.2) node[anchor=north west] {$1$};
\draw (-0.3,4.2) node[anchor=north west] {$j'$};
\draw (1.2,4.2) node[anchor=north west] {$j$};
\draw (2.7,4.2) node[anchor=north west] {$l$};
\draw (-1.4,2.) node[anchor=north west] {$\ldots$};
\draw (0.1,2.) node[anchor=north west] {$\ldots$};
\draw (1.6,2.) node[anchor=north west] {$\ldots$};
\end{tikzpicture}
$\substack{\displaystyle \stackrel{t_{i',j'}}{\longrightarrow}\\ \vspace{4cm}}$
\begin{tikzpicture}[line cap=round,line join=round,>=triangle 45,x=0.5cm,y=0.5cm]
\clip(-3.5,-4.5) rectangle (4.,5.);
\draw [line width=0.4pt] (-2.,1.)-- (3.5,1.);
\draw [line width=0.4pt] (3.5,1.)-- (3.5,-1.);
\draw [line width=0.4pt] (3.5,-1.)-- (-2.,-1.);
\draw [line width=0.4pt] (-2.,-1.)-- (-2.,1.);
\draw [->,line width=0.4pt] (-1.5,1.) -- (-1.5,3.);
\draw [line width=0.4pt] (0.,1.) -- (0.,3.);
\draw [->,line width=0.4pt] (1.5,1.) -- (1.5,3.);
\draw [->,line width=0.4pt] (3.,1.) -- (3.,3.);
\draw [->,line width=0.4pt] (-1.5,-3.) -- (-1.5,-1.);
\draw [line width=0.4pt] (0.,-3.) -- (0.,-1.);
\draw [->,line width=0.4pt] (1.5,-3.) -- (1.5,-1.);
\draw [->,line width=0.4pt] (3.,-3.) -- (3.,-1.);
\draw (0.25,0.5) node[anchor=north west] {$G$};
\draw (-1.8,-3) node[anchor=north west] {$1$};
\draw (0.3,-3) node[anchor=north west] {$i-1$};
\draw (2.1,-3) node[anchor=north west] {$k-1$};
\draw (-1.4,-2.2) node[anchor=north west] {$\ldots$};
\draw (0.1,-2.2) node[anchor=north west] {$\ldots$};
\draw (1.6,-2.2) node[anchor=north west] {$\ldots$};
\draw (-1.8,4.2) node[anchor=north west] {$1$};
\draw (0.3,4.2) node[anchor=north west] {$j-1$};
\draw (2.1,4.2) node[anchor=north west] {$l-1$};
\draw (-1.4,2.) node[anchor=north west] {$\ldots$};
\draw (0.1,2.) node[anchor=north west] {$\ldots$};
\draw (1.6,2.) node[anchor=north west] {$\ldots$};
\draw [shift={(-1.5,-3.)},line width=0.4pt]  plot[domain=3.141592653589793:6.283185307179586,variable=\t]({1.*1.5*cos(\t r)+0.*1.5*sin(\t r)},{0.*1.5*cos(\t r)+1.*1.5*sin(\t r)});
\draw [shift={(-1.5,3.)},line width=0.4pt]  plot[domain=0.:3.141592653589793,variable=\t]({1.*1.5*cos(\t r)+0.*1.5*sin(\t r)},{0.*1.5*cos(\t r)+1.*1.5*sin(\t r)});
\draw [->,line width=0.4pt] (-3.,3.) -- (-3.,-3.);
\end{tikzpicture}
$\substack{\displaystyle \stackrel{t_{i-1,j-1}}{\longrightarrow}\\ \vspace{4cm}}$
\begin{tikzpicture}[line cap=round,line join=round,>=triangle 45,x=0.5cm,y=0.5cm]
\clip(-3.5,-4.5) rectangle (5.,5.);
\draw [line width=0.4pt] (-2.,1.)-- (3.5,1.);
\draw [line width=0.4pt] (3.5,1.)-- (3.5,-1.);
\draw [line width=0.4pt] (3.5,-1.)-- (-2.,-1.);
\draw [line width=0.4pt] (-2.,-1.)-- (-2.,1.);
\draw [->,line width=0.4pt] (-1.5,1.) -- (-1.5,3.);
\draw [line width=0.4pt] (0.,1.) -- (0.,3.);
\draw [line width=0.4pt] (1.5,1.) -- (1.5,3.);
\draw [->,line width=0.4pt] (3.,1.) -- (3.,3.);
\draw [->,line width=0.4pt] (-1.5,-3.) -- (-1.5,-1.);
\draw [line width=0.4pt] (0.,-3.) -- (0.,-1.);
\draw [line width=0.4pt] (1.5,-3.) -- (1.5,-1.);
\draw [->,line width=0.4pt] (3.,-3.) -- (3.,-1.);
\draw (0.25,0.5) node[anchor=north west] {$G$};
\draw (-1.8,-3) node[anchor=north west] {$1$};
\draw (2.1,-3) node[anchor=north west] {$k-2$};
\draw (-1.4,-2.2) node[anchor=north west] {$\ldots$};
\draw (0.1,-2.2) node[anchor=north west] {$\ldots$};
\draw (1.6,-2.2) node[anchor=north west] {$\ldots$};
\draw (-1.8,4.2) node[anchor=north west] {$1$};
\draw (2.1,4.2) node[anchor=north west] {$l-2$};
\draw (-1.4,2.) node[anchor=north west] {$\ldots$};
\draw (0.1,2.) node[anchor=north west] {$\ldots$};
\draw (1.6,2.) node[anchor=north west] {$\ldots$};
\draw [shift={(-1.5,-3.)},line width=0.4pt]  plot[domain=3.141592653589793:6.283185307179586,variable=\t]({1.*1.5*cos(\t r)+0.*1.5*sin(\t r)},{0.*1.5*cos(\t r)+1.*1.5*sin(\t r)});
\draw [shift={(-1.5,3.)},line width=0.4pt]  plot[domain=0.:3.141592653589793,variable=\t]({1.*1.5*cos(\t r)+0.*1.5*sin(\t r)},{0.*1.5*cos(\t r)+1.*1.5*sin(\t r)});
\draw [->,line width=0.4pt] (-3.,3.) -- (-3.,-3.);
\draw [shift={(3.,-3.)},line width=0.4pt]  plot[domain=3.141592653589793:6.283185307179586,variable=\t]({1.*1.5*cos(\t r)+0.*1.5*sin(\t r)},{0.*1.5*cos(\t r)+1.*1.5*sin(\t r)});
\draw [shift={(3.,3.)},line width=0.4pt]  plot[domain=0.:3.141592653589793,variable=\t]({1.*1.5*cos(\t r)+0.*1.5*sin(\t r)},{0.*1.5*cos(\t r)+1.*1.5*sin(\t r)});
\draw [->,line width=0.4pt] (4.5,3.) -- (4.5,-3.);
\end{tikzpicture}

\vspace{-2cm}
\end{center}
For Property 3.(b), let us consider $p=G$ a graph.
As the input edge indexed by $i$ in $\sigma \cdot G\cdot \tau$ is the input edge of $G$ indexed by $\tau(i)$
and the output edge indexed by $j$ in $\sigma \cdot G\cdot \tau$ is the output edge of $G$ indexed by $\sigma^{-1}(j)$,
$G_1=t_{i,j}(\sigma\cdot G\cdot \tau)$ is the graph obtained by gluing together
the input indexed by $\tau(j)$ and the output indexed by $\sigma^{-1}(j)$, reindexing the input
according to $\sigma_i$ and the output edges by $\tau_j$, so $G_1=\sigma_i\cdot t_{\tau(i),\sigma^{-1}(j)}(G)\cdot \tau_j$.

Let us prove Property 3.(c). By Lemma \ref{lemmeaxiomessimples}, it is enough to prove it for
$(p,p')=(G,G')$ a pair of graphs and $(i,j)=(1,1)$. In this case,
$e_i$ and $f_j$ are both edges of $G$, so $t_{1,1}(G*G')=t_{1,1}(G)*G'$. 

For Property 3.(d), let us consider the graph $I$ such that
\[V(I)=E(I)=O(I)=I(I)=L(I)=\emptyset,\]
and $IO(I)$ being reduced to a single element. Then for any graph $G$ with $|O(G)|\geq1$,
\[t_{1,2}(I*G)=G.\]
By Lemma \ref{lemmeaxiomessimples}, Property 3.(d) is satisfied, so $\Gr$ is a TraP. 
\end{proof}

\subsection{Free TraPs}

 \label{subsection:free_Traps}
\begin{theo} \label{theoFGlibre}
Let $P$ be a TraP and, for any $k,l\in  \N_0$, let $x_{k,l}\in P(k,l)$ such that:
\begin{align*}
&\forall \sigma \in \sym_l,\: \forall \tau\in \sym_k,&
\sigma\cdot x_{k,l}\cdot \tau&=x_{k,l}.
\end{align*}
There exists a unique TraP morphism $\Phi$ from $\Gr$ to $P$ sending $G_{k,l}$ to $x_{k,l}$ for any $k,l\geqslant 0$.
\end{theo}

\begin{proof} 
We provide here a sketch of the proof,  and refer the reader to the appendix for a full proof. We define $\Phi(G)$ for any graph 
$G\in \Gr(k,l)$ by induction on the number $N$  of internal edges of $G$. 

If $N=0$, then $G$ can be written as
\[G=\grapheo^{*p}*\sigma\cdot(I^{*q}*G_{k_1,l_1}*\ldots *G_{k_r,l_r})\cdot \tau,\]
(recall that $\grapheo$ is the graph with no vertex, and only one edge belonging to $L(G)$)
where $p,q,r\in  \N_0$, $(k_i,k_i) \in  \N_0^2$ for any $i$, and $\sigma \in \sym_{q+k_1+\ldots+k_r}$, $\tau\in \sym_{q+l_1+\ldots+l_r}$. 
We then put:
\[\Phi(G)=t_{1,1}(I)^{*p}*\sigma\cdot(I^{*q}*x_{k_1,l_1}*\ldots*x_{k_r,l_r})\cdot \tau.\]
We can prove that this does not depend on the choice of the decomposition of $G$, with the help of the TraP axioms applied to $P$ 
and the invariance of 
the $x_{k,l}$. 
Let us assume now that $\Phi(G')$ is defined for any graph with $N-1$ internal edges, for a given $N \geqslant 1$.
Let $G$ be a graph with $N$ internal edges and let $e$ be one of these edges. 
Let $G_e$ be a graph obtained by cutting this edge in two, such that $G=t_{1,1}(G_e)$.  We then set:
\[\Phi(G)=t_{1,1}\circ \Phi(G_e).\]
One can prove that this does not depend on the choice of $e$. It can then be shown that $\Phi$ defined as above is compatible with the
partial trace maps. \end{proof}

The following TraP counterpart of Theorem \ref{thm:freeness_GaX} can be proved in a similar way as Theorem \ref{theoFGlibre}:
\begin{theo}  \label{theo:freeTrapplanar}
	 
	Let $X$ be a $\sym\times \sym^{op}$-module, $P$  a TraP and $\varphi:X\longrightarrow P$ be a morphism
	of $\sym\times \sym^{op}$-modules. There exists a unique morphism of TraPs   $\Phi: \PGr(X)\longrightarrow P$,
which	extends $\varphi$  so that the following diagramme commutes:
\[\xymatrix{X\ar[d]_{\iota} \ar[r]^{\varphi}&P\\
\PGr(X)\ar[ru]_{\Phi}&}\]
	where $\iota: X\hookrightarrow \PGr(X)$ is the map that sends  an element $x$ of $X$ to the planar $X$-decorated graph   $G(x)= (PG_{k,l},d)$ with $d$ sending the unique vertex of $PG_{k,l}$ to $x$.

In other words, $\PGr(X)$ is the free TroP generated by the $\sym\times \sym^{op}$-module $X$. 
\end{theo}
\begin{remark}  
     The invariance condition of $x_{k,l}$ in Theorem \ref{theoFGlibre} is replaced here with the planar condition on graphs. 
     They play the same role, 
     namely to allow us to show that the map $\Phi$, defined inductively, is indeed well-defined.
    \end{remark}

\section{The functor $\Gacirc$ applied on TraPs} \label{sec:functor_Garcirc_TraP}

\subsection{The functor $\Gacirc$ as an endofunctor of $\Trap$}

\begin{prop}\label{prop::GaTraP}
Let $X$ be a $\sym\times \sym^{op}$-module. Then $\Gacirc(X)$ is a TraP.
\end{prop}

\begin{proof}
Similarly to the proof of  Proposition \ref{prop:TrapGG} concerning  $\Gr$, we can   prove that  $\PGr(X)$ is a TraP.

If $G$ and $G'$ are two $X$-decorated planar graphs such that $G'$ is obtained from $G$ 
by the action of permutations on the incoming and outgoing edges of a vertex of $G$,
then clearly, for any relevant $i$ and $j$, $t_{i,j}(G')$ is obtained from $G$ by the same operation.
So $t_{i,j}(G-G')\in I$, and the partial trace maps of $\PGr(X)$ induce partial trace maps on $\Gacirc(X)$. 
\end{proof}

Hence, $\Gacirc$ is a functor from the category $\catssm$ to the category $\Trap$. 
Combining with the forgetful functor $F:\Trap\longrightarrow \catssm$,
we obtain an endofunctor $\Gacirc\circ F:\Trap\longrightarrow \Trap$,
which we denote by $\Gacirc$, with a  slight  abuse of   notations.
As for ProPs (Corollary \ref{cor:PdecGrc} and Proposition \ref{prop:functGrc}), we have the following statement.

\begin{prop}\label{prop:PdecGrc}
 Given a  TraP $P$, there is a canonical morphism of TraPs \[ \alpha_P:\Gacirc(P)\longrightarrow P\] 
 induced by the decoration. 
These maps define a natural transformation 
from the endofunctor $\Gacirc$ to the identity endofunctor of $\Trap$,
that is to say: for any morphism of TraP $\varphi:P\longrightarrow Q$, the following diagram commutes:
\begin{align*}
&\xymatrix{\Gacirc(P)\ar[r]^{\Gacirc(\varphi)} \ar[d]_{\alpha_P}&\Gacirc(Q)\ar[d]^{\alpha_Q}\\
P\ar[r]_{\varphi}&Q}
\end{align*}
\end{prop}

\begin{proof}
Similar arguments as in the proofs of Corollary \ref{cor:PdecGrc} and Proposition \ref{prop:functGrc}.
\end{proof}

\subsection{The endofunctor $\Gacirc$ as a monad} \label{subsec:monad}

Let us now equip the endofunctor $\Gacirc$ with a  \textbf{monad} structure, a terminology we 
borrow from \cite[Definition 2.13]{Merkulov2009}.
\begin{defi} 
A monad $\Gamma$  (also called a triple)  on a category $\mathcal C$ is an associative and unital
		monoid $(\Gamma, \mu, \nu)$ in the the unital monoid\footnote{The terminology monoid is used in this definition with the obvious abuse of vocabulary since 
				$\Gamma$ and $\End(\cat)$ are not necessarily sets.} $\End(\cat)$ of endofunctors of $\cat$. This means that the multiplication 
		$\mu:\Gamma\circ \Gamma\longrightarrow \Gamma$ and the unit morphism $\nu:Id_\cat\longrightarrow \Gamma$
		should satisfy the axioms given by commutativity of the diagrams below for any object $P$ of the category $\cat$. 
	\begin{align}
	\label{axiomesmonade}
	&\xymatrix{\Gamma\circ \Gamma\circ \Gamma(P)\ar[r]^{\Gamma(\mu_P)}\ar[d]_{\mu_{\Gamma(P)}}
		&\Gamma\circ \Gamma(P)\ar[d]^{\mu_P}\\
		\Gamma\circ \Gamma(P)\ar[r]_{\mu_P}&\Gamma(E)}&
	\xymatrix{\Gamma(P)\ar[r]^{\Gamma(\nu_P)}\ar[rd]_{Id_\cat}&\Gamma\circ \Gamma(P)\ar[d]_{\mu_P}&
		\Gamma(P)\ar[l]_{\nu_{\Gamma(P)}}\ar[ld]^{Id_\cat}\\
		&\Gamma(P)&}
	\end{align}
\end{defi} We want to define a transformation   $\nu: Id_{\catssm}\longrightarrow \Gacirc$, i.e. maps  $\nu_P:P\longrightarrow \Gacirc(P)$   for any  $\sym\times \sym^{op}$-module $P$. The morphism $\nu_P$ sends
	an element $p\in P(k,l)$ to the class of the graph $PG_{k,l}(p)$ with one vertex $v$  decorated by $p$, 
	and  $k$  incoming edges indexed from left to right by $1,\ldots,k$,  $l$ outgoing edges indexed from left to right by $1,\ldots,l$.

\[\substack{\displaystyle \nu_P(p)=\\ \vspace{2cm}}
\begin{tikzpicture}[line cap=round,line join=round,>=triangle 45,x=0.7cm,y=0.7cm]
\clip(0.6,-2.1) rectangle (2.4,1.7);
\draw [line width=.4pt] (0.8,0.)-- (2.2,0.);
\draw [line width=.4pt] (2.2,0.)-- (2.2,-0.5);
\draw [line width=.4pt] (2.2,-0.5)-- (0.8,-0.5);
\draw [line width=.4pt] (0.8,-0.5)-- (0.8,0.);
\draw (1.2,0.1) node[anchor=north west] {\scriptsize $p$};
\draw [->,line width=.4pt] (1.,0.) -- (1.,1.);
\draw [->,line width=.4pt] (2.,0.) -- (2.,1.);
\draw [->,line width=.4pt] (1.,-1.5) -- (1,-0.5);
\draw [->,line width=.4pt] (2.,-1.5) -- (2.,-0.5);
\draw (0.7,-1.4) node[anchor=north west] {\scriptsize $1$};
\draw (1.05,-1.6) node[anchor=north west] {\scriptsize $\ldots$};
\draw (1.7,-1.4) node[anchor=north west] {\scriptsize $k$};
\draw (0.7,1.6) node[anchor=north west] {\scriptsize $1$};
\draw (1.05,1.4) node[anchor=north west] {\scriptsize $\ldots$};
\draw (1.7,1.6) node[anchor=north west] {\scriptsize $l$};
\end{tikzpicture}\substack{\displaystyle .\\ \vspace{2cm}}\]
\vspace{-1.5cm}

\noindent The morphism $\nu$ is a unit in $\End(\catssm)$ in the following sense: for any morphism $\phi:P\longrightarrow Q$,
the following diagram commutes:
\[\xymatrix{P\ar[r]^{\nu_P} \ar[d]_{\phi} &\Gacirc(P)\ar[d]^{\Gacirc(\phi)}\\
	Q\ar[r]_{\nu_Q}&\Gacirc(Q)}\]

The multiplication is given by  morphisms $\mu_P:\Gacirc\circ \Gacirc(P)\longrightarrow \Gacirc(P)$  attached to   $\sym\times \sym^{op}$-modules $P$.
Elements of $\Gacirc\circ \Gacirc(P)$ are graphs $G$  whose vertices $v$  are decorated by   graphs
$G_v$, consistently with  the number of incoming and outgoing edges. We denote by $\mu_P(G)$
the graph $H$ such that
\[V(H)=\bigsqcup_{v\in V(G)} V(G_v),\]
whose edges are obtained by identifying, for any vertex $v$, the $i$-th incoming edges of $v$
with the $i$-th incoming edge of $G_v$, and the $j$-th outgoing edge of $v$ with the $j$-th outgoing edge of $G_v$.

To illustrate this graphically, here is an example in which $\mu_P$ sends the graph on the left to the graph on the right:
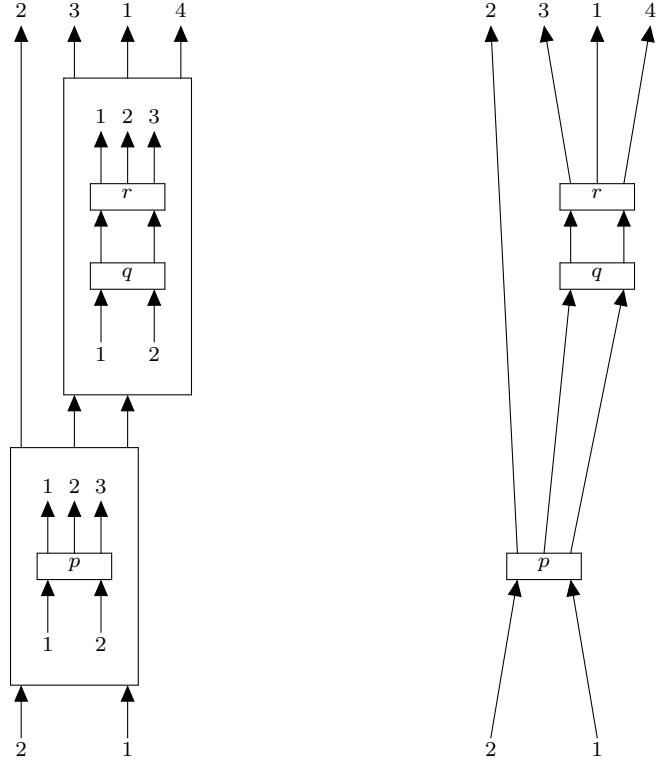
\begin{align*}
&\begin{tikzpicture}[line cap=round,line join=round,>=triangle 45,x=0.7cm,y=0.7cm]
\clip(0.2,-4.1) rectangle (3.8,11.);
\draw [line width=.4pt] (0.8,0.)-- (2.2,0.);
\draw [line width=.4pt] (2.2,0.)-- (2.2,-0.5);
\draw [line width=.4pt] (2.2,-0.5)-- (0.8,-0.5);
\draw [line width=.4pt] (0.8,-0.5)-- (0.8,0.);
\draw (1.2,0.1) node[anchor=north west] {\scriptsize $p$};
\draw [->,line width=.4pt] (1.,0.) -- (1.,1.);
\draw [->,line width=.4pt] (1.5,0.) -- (1.5,1.);
\draw [->,line width=.4pt] (2.,0.) -- (2.,1.);
\draw [->,line width=.4pt] (1.,-1.5) -- (1.,-0.5);
\draw [->,line width=.4pt] (2.,-1.5) -- (2.,-0.5);
\draw (0.7,-1.4) node[anchor=north west] {\scriptsize $1$};
\draw (1.7,-1.4) node[anchor=north west] {\scriptsize $2$};
\draw (0.7,1.6) node[anchor=north west] {\scriptsize $1$};
\draw (1.2,1.6) node[anchor=north west] {\scriptsize $2$};
\draw (1.7,1.6) node[anchor=north west] {\scriptsize $3$};
\draw [line width=.4pt]  (0.3,-2.5) -- (2.7,-2.5);
\draw [line width=.4pt]  (2.7,-2.5) -- (2.7,2.);
\draw [line width=.4pt]  (2.7,2.) -- (0.3,2.);
\draw [line width=.4pt]  (0.3,2.) -- (0.3,-2.5);
\draw [->,line width=.4pt] (0.5,-3.5) -- (0.5,-2.5);
\draw [->,line width=.4pt] (2.5,-3.5) -- (2.5,-2.5);
\draw (0.2,-3.4) node[anchor=north west] {\scriptsize $2$};
\draw (2.2,-3.4) node[anchor=north west] {\scriptsize $1$};
\draw [->,line width=.4pt] (1.5,2.) -- (1.5,3.);
\draw [->,line width=.4pt] (2.5,2.) -- (2.5,3.);
\draw [line width=.4pt]  (1.3,3.) -- (3.7,3.); 
\draw [line width=.4pt]  (3.7,3.) -- (3.7,9.); 
\draw [line width=.4pt]  (3.7,9.) -- (1.3,9.); 
\draw [line width=.4pt]  (1.3,9.) -- (1.3,3.); 
\draw [->,line width=.4pt] (2.,4.) -- (2.,5.);
\draw [->,line width=.4pt] (1.5,9.) -- (1.5,10.);
\draw [->,line width=.4pt] (2.5,9.) -- (2.5,10.);
\draw [->,line width=.4pt] (3.5,9.) -- (3.5,10.);
\draw [->,line width=.4pt] (0.5,2.) -- (0.5,10.);
\draw (0.2,10.6) node[anchor=north west] {\scriptsize $2$};
\draw (1.2,10.6) node[anchor=north west] {\scriptsize $3$};
\draw (2.2,10.6) node[anchor=north west] {\scriptsize $1$};
\draw (3.2,10.6) node[anchor=north west] {\scriptsize $4$};
\draw (1.7,4.1) node[anchor=north west] {\scriptsize $1$};
\draw (2.7,4.1) node[anchor=north west] {\scriptsize $2$};
\draw [->,line width=.4pt] (3.,4.) -- (3.,5.);
\draw [line width=.4pt]  (1.8,5.) -- (3.2,5.);
\draw [line width=.4pt]  (3.2,5.) -- (3.2,5.5);
\draw [line width=.4pt]  (3.2,5.5) -- (1.8,5.5);
\draw [line width=.4pt]  (1.8,5.5) -- (1.8,5.);
\draw (2.2,5.6) node[anchor=north west] {\scriptsize $q$};
\draw [->,line width=.4pt] (2.,5.5) -- (2.,6.5);
\draw [->,line width=.4pt] (3.,5.5) -- (3.,6.5);
\draw [line width=.4pt]  (1.8,6.5) -- (3.2,6.5);
\draw [line width=.4pt]  (3.2,6.5) -- (3.2,7);
\draw [line width=.4pt]  (3.2,7) -- (1.8,7);
\draw [line width=.4pt]  (1.8,7) -- (1.8,6.5);
\draw (2.2,7.1) node[anchor=north west] {\scriptsize $r$};
\draw [->,line width=.4pt] (2.,7.) -- (2.,8.);
\draw [->,line width=.4pt] (2.5,7.) -- (2.5,8.);
\draw [->,line width=.4pt] (3.,7.) -- (3.,8.);
\draw (1.7,8.6) node[anchor=north west] {\scriptsize $1$};
\draw (2.2,8.6) node[anchor=north west] {\scriptsize $2$};
\draw (2.7,8.6) node[anchor=north west] {\scriptsize $3$};
\end{tikzpicture}&&
\begin{tikzpicture}[line cap=round,line join=round,>=triangle 45,x=0.7cm,y=0.7cm]
\clip(0.2,-4.1) rectangle (3.8,11.);
\draw [line width=.4pt] (0.8,0.)-- (2.2,0.);
\draw [line width=.4pt] (2.2,0.)-- (2.2,-0.5);
\draw [line width=.4pt] (2.2,-0.5)-- (0.8,-0.5);
\draw [line width=.4pt] (0.8,-0.5)-- (0.8,0.);
\draw (1.2,0.1) node[anchor=north west] {\scriptsize $p$};
\draw [->,line width=.4pt] (1.,0.) -- (0.5,10.); 
\draw [->,line width=.4pt] (1.5,0.) -- (2.,5.); 
\draw [->,line width=.4pt] (2.,0.) -- (3.,5.);  
\draw [->,line width=.4pt] (0.5,-3.5) -- (1.,-0.5); 
\draw [->,line width=.4pt] (2.5,-3.5) -- (2.,-0.5); 
\draw (0.2,-3.4) node[anchor=north west] {\scriptsize $2$};
\draw (2.2,-3.4) node[anchor=north west] {\scriptsize $1$};
\draw [->,line width=.4pt] (2.,7.) -- (1.5,10.);
\draw [->,line width=.4pt] (2.5,7.) -- (2.5,10.);
\draw [->,line width=.4pt] (3.,7.) -- (3.5,10.);
\draw (0.2,10.6) node[anchor=north west] {\scriptsize $2$};
\draw (1.2,10.6) node[anchor=north west] {\scriptsize $3$};
\draw (2.2,10.6) node[anchor=north west] {\scriptsize $1$};
\draw (3.2,10.6) node[anchor=north west] {\scriptsize $4$};
\draw [line width=.4pt]  (1.8,5.) -- (3.2,5.);
\draw [line width=.4pt]  (3.2,5.) -- (3.2,5.5);
\draw [line width=.4pt]  (3.2,5.5) -- (1.8,5.5);
\draw [line width=.4pt]  (1.8,5.5) -- (1.8,5.);
\draw (2.2,5.6) node[anchor=north west] {\scriptsize $q$};
\draw [->,line width=.4pt] (2.,5.5) -- (2.,6.5);
\draw [->,line width=.4pt] (3.,5.5) -- (3.,6.5);
\draw [line width=.4pt]  (1.8,6.5) -- (3.2,6.5);
\draw [line width=.4pt]  (3.2,6.5) -- (3.2,7);
\draw [line width=.4pt]  (3.2,7) -- (1.8,7);
\draw [line width=.4pt]  (1.8,7) -- (1.8,6.5);
\draw (2.2,7.1) node[anchor=north west] {\scriptsize $r$};
\end{tikzpicture}
\end{align*}
where $p\in P(2,3)$, $q\in P(2,2)$ and $r\in P(2,3)$.

It is  clear from the combinatorics that the relations corresponding to the diagrams  (\ref{axiomesmonade}) are satisfied. Hence:
\begin{prop} \label{prop:monad}
	The triple $\Gacirc=(\Gacirc,\mu,\nu)$ is a monad in the category $\catssm$. 
\end{prop}
 
\section{TraPs versus ProPs}

We have built the free TraPs by means of graphs  discussed in Subsection \ref{subsection:generalised_graphs}. 
This, together with the functor 
$\Gacirc$ of Sections \ref{section:Hom_dec_graphs} and \ref{sec:functor_Garcirc_TraP} will allow us to show the equivalence of the categories  of 
TraPs and wheeled ProPs.

\subsection{TraPs are wheeled ProPs} \label{subsec:Trap_wProPs}

The  free  TraP we previously built from a given TraP  enables us to relate TraPs and Merkulov's notion of wheeled ProPs \cite{Merkulov2009}.
We now build algebras on the monad $\Gacirc$. Let us first recall the notion of  $\Gamma$-algebra  (see e.g. \cite[Definition 2.1.4]{Merkulov2009}).
\begin{defi} Let $\cat$ be a category. An   algebra over a monad  $\Gamma\in\End(\cat)$ or a \textbf{$\Gamma$-algebra} is an object $P$ of $\cat$ together with a
	structure morphism $\alpha: \Gamma(P)\to P$  such that
the following diagrams commute:
\begin{align}
\label{axiomesalgebres}
&\xymatrix{\Gamma\circ \Gamma(P)\ar[r]^{\Gamma(\alpha)} \ar[d]_{\mu_P}\ar[d]_{\mu_P}&\Gamma(P) \ar[d]^\alpha\\
\Gamma(P)\ar[r]_\alpha&P}&
&\xymatrix{P\ar[r]^{\nu_P}\ar[d]_{Id}&\Gamma(P)\ar[ld]^{\alpha}\\P&}
\end{align}
\end{defi}
\begin{prop} \label{propgammatotrop}
Any $\Gacirc$-algebra $(P,\alpha)$ defines a TraP defined as follows:
\begin{itemize}
\item For any $(p,p')\in P(k,l)\times P(k',l')$, $p*p'$ is obtained by  applying  $\alpha$ to the following graph:
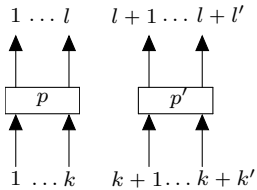
\begin{align*}
\begin{tikzpicture}[line cap=round,line join=round,>=triangle 45,x=0.7cm,y=0.7cm]
\clip(0.8,-2.1) rectangle (2.2,1.7);
\draw [line width=.4pt] (0.8,0.)-- (2.2,0.);
\draw [line width=.4pt] (2.2,0.)-- (2.2,-0.5);
\draw [line width=.4pt] (2.2,-0.5)-- (0.8,-0.5);
\draw [line width=.4pt] (0.8,-0.5)-- (0.8,0.);
\draw (1.2,0.1) node[anchor=north west] {\scriptsize $p$};
\draw [->,line width=.4pt] (1.,0.) -- (1.,1.);
\draw [->,line width=.4pt] (2.,0.) -- (2.,1.);
\draw [->,line width=.4pt] (1.,-1.5) -- (1,-0.5);
\draw [->,line width=.4pt] (2.,-1.5) -- (2.,-0.5);
\draw (0.7,-1.4) node[anchor=north west] {\scriptsize $1$};
\draw (1.1,-1.6) node[anchor=north west] {\scriptsize $\ldots$};
\draw (1.7,-1.4) node[anchor=north west] {\scriptsize $k$};
\draw (0.7,1.7) node[anchor=north west] {\scriptsize $1$};
\draw (1.05,1.5) node[anchor=north west] {\scriptsize $\ldots$};
\draw (1.7,1.7) node[anchor=north west] {\scriptsize $l$};
\end{tikzpicture}
\begin{tikzpicture}[line cap=round,line join=round,>=triangle 45,x=0.7cm,y=0.7cm]
\clip(-0.3,-2.1) rectangle (3.4,1.7);
\draw [line width=.4pt] (0.8,0.)-- (2.2,0.);
\draw [line width=.4pt] (2.2,0.)-- (2.2,-0.5);
\draw [line width=.4pt] (2.2,-0.5)-- (0.8,-0.5);
\draw [line width=.4pt] (0.8,-0.5)-- (0.8,0.);
\draw (1.2,0.2) node[anchor=north west] {\scriptsize $p'$};
\draw [->,line width=.4pt] (1.,0.) -- (1.,1.);
\draw [->,line width=.4pt] (2.,0.) -- (2.,1.);
\draw [->,line width=.4pt] (1.,-1.5) -- (1,-0.5);
\draw [->,line width=.4pt] (2.,-1.5) -- (2.,-0.5);
\draw (0.1,-1.4) node[anchor=north west] {\scriptsize $k+1$};
\draw (1.1,-1.6) node[anchor=north west] {\scriptsize $\ldots$};
\draw (1.7,-1.35) node[anchor=north west] {\scriptsize $k+k'$};
\draw (0.1,1.7) node[anchor=north west] {\scriptsize $l+1$};
\draw (1.05,1.5) node[anchor=north west] {\scriptsize $\ldots$};
\draw (1.7,1.75) node[anchor=north west] {\scriptsize $l+l'$};
\end{tikzpicture}\end{align*}
\item For any $p\in P(k,l)$, for any $(i,j)\in [k]\times [l]$, 
$t_{i,j}(p)$ is obtained by the application of $\alpha$ to the following graph:
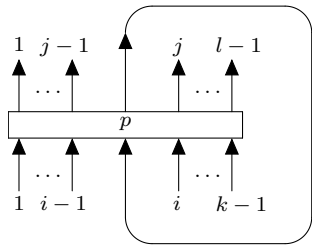
\begin{align*}
\begin{tikzpicture}[line cap=round,line join=round,>=triangle 45,x=0.7cm,y=0.7cm]
\clip(0.8,-2.6) rectangle (6.5,2.1);
\draw [line width=.4pt] (0.8,0.)-- (5.2,0.);
\draw [line width=.4pt] (5.2,0.)-- (5.2,-0.5);
\draw [line width=.4pt] (5.2,-0.5)-- (0.8,-0.5);
\draw [line width=.4pt] (0.8,-0.5)-- (0.8,0.);
\draw (2.7,0.1) node[anchor=north west] {\scriptsize $p$};
\draw [->,line width=.4pt] (1.,0.) -- (1.,1.);
\draw [->,line width=.4pt] (2.,0.) -- (2.,1.);
\draw [->,line width=.4pt] (3.,0.) -- (3.,1.5);
\draw [->,line width=.4pt] (4.,0.) -- (4.,1.);
\draw [->,line width=.4pt] (5.,0.) -- (5.,1.);
\draw [->,line width=.4pt] (1.,-1.5) -- (1,-0.5);
\draw [->,line width=.4pt] (2.,-1.5) -- (2.,-0.5);
\draw [->,line width=.4pt] (3.,-2) -- (3,-0.5);
\draw [->,line width=.4pt] (4.,-1.5) -- (4.,-0.5);
\draw [->,line width=.4pt] (5.,-1.5) -- (5.,-0.5);
\draw (0.7,-1.4) node[anchor=north west] {\scriptsize $1$};
\draw (1.1,-1.) node[anchor=north west] {\scriptsize $\ldots$};
\draw (1.2,-1.4) node[anchor=north west] {\scriptsize $i-1$};
\draw (3.7,-1.4) node[anchor=north west] {\scriptsize $i$};
\draw (4.1,-1.) node[anchor=north west] {\scriptsize $\ldots$};
\draw (4.5,-1.4) node[anchor=north west] {\scriptsize $k-1$};
\draw (0.7,1.6) node[anchor=north west] {\scriptsize $1$};
\draw (1.1,0.6) node[anchor=north west] {\scriptsize $\ldots$};
\draw (1.2,1.6) node[anchor=north west] {\scriptsize $j-1$};
\draw (3.7,1.6) node[anchor=north west] {\scriptsize $j$};
\draw (4.1,0.6) node[anchor=north west] {\scriptsize $\ldots$};
\draw (4.5,1.6) node[anchor=north west] {\scriptsize $l-1$};
\draw [shift={(3.5,-2.)},line width=0.4pt]  plot[domain=3.141592653589793:4.71238898,variable=\t]({1.*0.5*cos(\t r)+0.*0.5*sin(\t r)},{0.*0.5*cos(\t r)+1.*0.5*sin(\t r)});
\draw [line width=.4pt] (3.5,-2.5) -- (6.,-2.5);
\draw [shift={(6.,-2.)},line width=0.4pt]  plot[domain=4.71238898:6.283185307,variable=\t]({1.*0.5*cos(\t r)+0.*0.5*sin(\t r)},{0.*0.5*cos(\t r)+1.*0.5*sin(\t r)});
\draw [shift={(3.5,1.5)},line width=0.4pt]  plot[domain=1.570796327:3.141592653589793,variable=\t]({1.*0.5*cos(\t r)+0.*0.5*sin(\t r)},{0.*0.5*cos(\t r)+1.*0.5*sin(\t r)});
\draw [line width=.4pt] (3.5,2.) -- (6.,2.);
\draw [shift={(6,1.5)},line width=0.4pt]  plot[domain=0.:1.570796327,variable=\t]({1.*0.5*cos(\t r)+0.*0.5*sin(\t r)},{0.*0.5*cos(\t r)+1.*0.5*sin(\t r)});
\draw [line width=.4pt] (6.5,-2.) -- (6.5,1.5);
\end{tikzpicture}
\end{align*}
\end{itemize}
\end{prop}

\begin{proof} Let us prove some of the axioms of TraPs for $P$. The others can be proved in the same way and are left to the reader.

2. (a). let $(p,p',p'')\in P(k,l)\times P(k',l')\times P(k'',l'')$. Then $(p*p')*p''$ is obtained by the application
of $\alpha_P$ to the graph:
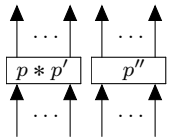
\begin{align*}
\begin{tikzpicture}[line cap=round,line join=round,>=triangle 45,x=0.7cm,y=0.7cm]
\clip(0.8,-1.5) rectangle (2.4,1.);
\draw [line width=.4pt] (0.8,0.)-- (2.2,0.);
\draw [line width=.4pt] (2.2,0.)-- (2.2,-0.5);
\draw [line width=.4pt] (2.2,-0.5)-- (0.8,-0.5);
\draw [line width=.4pt] (0.8,-0.5)-- (0.8,0.);
\draw (0.8,0.15) node[anchor=north west] {\scriptsize $p*p'$};
\draw [->,line width=.4pt] (1.,0.) -- (1.,1.);
\draw [->,line width=.4pt] (2.,0.) -- (2.,1.);
\draw [->,line width=.4pt] (1.,-1.5) -- (1,-0.5);
\draw [->,line width=.4pt] (2.,-1.5) -- (2.,-0.5);
\draw (1.1,-0.9) node[anchor=north west] {\scriptsize $\ldots$};
\draw (1.1,0.6) node[anchor=north west] {\scriptsize $\ldots$};
\end{tikzpicture}
\begin{tikzpicture}[line cap=round,line join=round,>=triangle 45,x=0.7cm,y=0.7cm]
\clip(0.8,-1.5) rectangle (2.4,1.);
\draw [line width=.4pt] (0.8,0.)-- (2.2,0.);
\draw [line width=.4pt] (2.2,0.)-- (2.2,-0.5);
\draw [line width=.4pt] (2.2,-0.5)-- (0.8,-0.5);
\draw [line width=.4pt] (0.8,-0.5)-- (0.8,0.);
\draw (1.2,0.15) node[anchor=north west] {\scriptsize $p''$};
\draw [->,line width=.4pt] (1.,0.) -- (1.,1.);
\draw [->,line width=.4pt] (2.,0.) -- (2.,1.);
\draw [->,line width=.4pt] (1.,-1.5) -- (1,-0.5);
\draw [->,line width=.4pt] (2.,-1.5) -- (2.,-0.5);
\draw (1.1,-0.9) node[anchor=north west] {\scriptsize $\ldots$};
\draw (1.1,0.6) node[anchor=north west] {\scriptsize $\ldots$};
\end{tikzpicture}
\end{align*}
(For the sake of simplicity, we delete the indices of the input and output edges of this graph: they are always indexed from left to right).
Hence, $(p*p')*p''$ is obtained by application of $\alpha\circ \Gacirc(\alpha)$ to the graph:
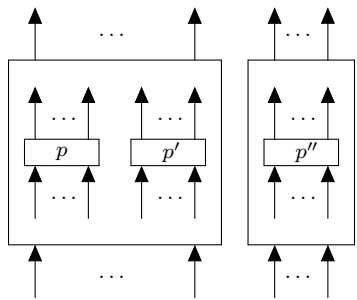
\begin{align*}
\begin{tikzpicture}[line cap=round,line join=round,>=triangle 45,x=0.7cm,y=0.7cm]
\clip(0.5,-3.) rectangle (7.,2.5);
\draw [line width=.4pt] (0.8,0.)-- (2.2,0.);
\draw [line width=.4pt] (2.2,0.)-- (2.2,-0.5);
\draw [line width=.4pt] (2.2,-0.5)-- (0.8,-0.5);
\draw [line width=.4pt] (0.8,-0.5)-- (0.8,0.);
\draw (1.2,0.05) node[anchor=north west] {\scriptsize $p$};
\draw [->,line width=.4pt] (1.,0.) -- (1.,1.);
\draw [->,line width=.4pt] (2.,0.) -- (2.,1.);
\draw [->,line width=.4pt] (1.,-1.5) -- (1,-0.5);
\draw [->,line width=.4pt] (2.,-1.5) -- (2.,-0.5);
\draw (1.1,-0.9) node[anchor=north west] {\scriptsize $\ldots$};
\draw (1.1,0.6) node[anchor=north west] {\scriptsize $\ldots$};
\draw [line width=.4pt] (2.8,0.)-- (4.2,0.);
\draw [line width=.4pt] (4.2,0.)-- (4.2,-0.5);
\draw [line width=.4pt] (4.2,-0.5)-- (2.8,-0.5);
\draw [line width=.4pt] (2.8,-0.5)-- (2.8,0.);
\draw (3.2,0.15) node[anchor=north west] {\scriptsize $p'$};
\draw [->,line width=.4pt] (3.,0.) -- (3.,1.);
\draw [->,line width=.4pt] (4.,0.) -- (4.,1.);
\draw [->,line width=.4pt] (3.,-1.5) -- (3,-0.5);
\draw [->,line width=.4pt] (4.,-1.5) -- (4.,-0.5);
\draw (3.1,-0.9) node[anchor=north west] {\scriptsize $\ldots$};
\draw (3.1,0.6) node[anchor=north west] {\scriptsize $\ldots$};
\draw [line width=.4pt] (5.3,0.)-- (6.7,0.);
\draw [line width=.4pt] (6.7,0.)-- (6.7,-0.5);
\draw [line width=.4pt] (6.7,-0.5)-- (5.3,-0.5);
\draw [line width=.4pt] (5.3,-0.5)-- (5.3,0.);
\draw (5.7,0.15) node[anchor=north west] {\scriptsize $p''$};
\draw [->,line width=.4pt] (5.5,0.) -- (5.5,1.);
\draw [->,line width=.4pt] (6.5,0.) -- (6.5,1.);
\draw [->,line width=.4pt] (5.5,-1.5) -- (5.5,-0.5);
\draw [->,line width=.4pt] (6.5,-1.5) -- (6.5,-0.5);
\draw (5.6,-0.9) node[anchor=north west] {\scriptsize $\ldots$};
\draw (5.6,0.6) node[anchor=north west] {\scriptsize $\ldots$};
\draw [line width=.4pt] (0.5,-2) -- (4.5,-2);
\draw [line width=.4pt] (4.5,-2) -- (4.5,1.5);
\draw [line width=.4pt] (4.5,1.5) -- (0.5,1.5);
\draw [line width=.4pt] (0.5,1.5) -- (0.5,-2.);
\draw [->,line width=.4pt] (1.,-3.) -- (1.,-2.);
\draw [->,line width=.4pt] (1.,1.5) -- (1.,2.5);
\draw [->,line width=.4pt] (4.,-3.) -- (4.,-2.);
\draw [->,line width=.4pt] (4.,1.5) -- (4.,2.5);
\draw (2.,-2.4) node[anchor=north west] {\scriptsize $\ldots$};
\draw (2.,2.2) node[anchor=north west] {\scriptsize $\ldots$};
\draw [line width=.4pt] (5,-2) -- (7,-2);
\draw [line width=.4pt] (7,-2) -- (7,1.5);
\draw [line width=.4pt] (7,1.5) -- (5,1.5);
\draw [line width=.4pt] (5,1.5) -- (5,-2);
\draw [->,line width=.4pt] (5.5,-3.) -- (5.5,-2.);
\draw [->,line width=.4pt] (6.5,1.5) -- (6.5,2.5);
\draw [->,line width=.4pt] (6.5,-3.) -- (6.5,-2.);
\draw [->,line width=.4pt] (5.5,1.5) -- (5.5,2.5);
\draw (5.5,-2.4) node[anchor=north west] {\scriptsize $\ldots$};
\draw (5.5,2.2) node[anchor=north west] {\scriptsize $\ldots$};
\end{tikzpicture}
\end{align*}
Note that for the second connected component of this graph, this comes from:
\[\alpha \circ \Gacirc(\alpha)\circ \Gacirc(\nu_P)(p'')=\alpha \circ \Gacirc(\alpha \circ \nu_P)(p'')
=\alpha \circ \Gacirc(Id_P)(p'')=\alpha(p'').\]
As $\alpha \circ \Gacirc(\alpha)=\alpha \circ \mu_P$, $(p*p')*p''$ is obtained by application of $\alpha$ to the graph:
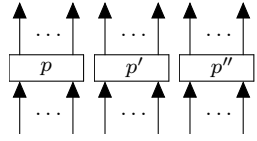
\begin{align*}
\begin{tikzpicture}[line cap=round,line join=round,>=triangle 45,x=0.7cm,y=0.7cm]
\clip(0.8,-1.5) rectangle (2.4,1.);
\draw [line width=.4pt] (0.8,0.)-- (2.2,0.);
\draw [line width=.4pt] (2.2,0.)-- (2.2,-0.5);
\draw [line width=.4pt] (2.2,-0.5)-- (0.8,-0.5);
\draw [line width=.4pt] (0.8,-0.5)-- (0.8,0.);
\draw (1.2,0.05) node[anchor=north west] {\scriptsize $p$};
\draw [->,line width=.4pt] (1.,0.) -- (1.,1.);
\draw [->,line width=.4pt] (2.,0.) -- (2.,1.);
\draw [->,line width=.4pt] (1.,-1.5) -- (1,-0.5);
\draw [->,line width=.4pt] (2.,-1.5) -- (2.,-0.5);
\draw (1.1,-0.9) node[anchor=north west] {\scriptsize $\ldots$};
\draw (1.1,0.6) node[anchor=north west] {\scriptsize $\ldots$};
\end{tikzpicture}
\begin{tikzpicture}[line cap=round,line join=round,>=triangle 45,x=0.7cm,y=0.7cm]
\clip(0.8,-1.5) rectangle (2.4,1.);
\draw [line width=.4pt] (0.8,0.)-- (2.2,0.);
\draw [line width=.4pt] (2.2,0.)-- (2.2,-0.5);
\draw [line width=.4pt] (2.2,-0.5)-- (0.8,-0.5);
\draw [line width=.4pt] (0.8,-0.5)-- (0.8,0.);
\draw (1.2,0.15) node[anchor=north west] {\scriptsize $p'$};
\draw [->,line width=.4pt] (1.,0.) -- (1.,1.);
\draw [->,line width=.4pt] (2.,0.) -- (2.,1.);
\draw [->,line width=.4pt] (1.,-1.5) -- (1,-0.5);
\draw [->,line width=.4pt] (2.,-1.5) -- (2.,-0.5);
\draw (1.1,-0.9) node[anchor=north west] {\scriptsize $\ldots$};
\draw (1.1,0.6) node[anchor=north west] {\scriptsize $\ldots$};
\end{tikzpicture}
\begin{tikzpicture}[line cap=round,line join=round,>=triangle 45,x=0.7cm,y=0.7cm]
\clip(0.8,-1.5) rectangle (2.4,1.);
\draw [line width=.4pt] (0.8,0.)-- (2.2,0.);
\draw [line width=.4pt] (2.2,0.)-- (2.2,-0.5);
\draw [line width=.4pt] (2.2,-0.5)-- (0.8,-0.5);
\draw [line width=.4pt] (0.8,-0.5)-- (0.8,0.);
\draw (1.2,0.15) node[anchor=north west] {\scriptsize $p''$};
\draw [->,line width=.4pt] (1.,0.) -- (1.,1.);
\draw [->,line width=.4pt] (2.,0.) -- (2.,1.);
\draw [->,line width=.4pt] (1.,-1.5) -- (1,-0.5);
\draw [->,line width=.4pt] (2.,-1.5) -- (2.,-0.5);
\draw (1.1,-0.9) node[anchor=north west] {\scriptsize $\ldots$};
\draw (1.1,0.6) node[anchor=north west] {\scriptsize $\ldots$};
\end{tikzpicture}
\end{align*}
The same computation can be done for $p*(p'*p'')$, which gives the associativity of $*$.

2. (b). The unit is $I_0=\alpha(\emptyset)$, where $\emptyset$ is the graph with no vertex and no edge.

3. (d). The unit $I$ is $\alpha(I_1)$, where $I_1$ is the graph with only one input-output edge.
Let $p\in P(k,l)$ and $2\leqslant j\leqslant l+1$. Then $t_{1,j}(I*p)$ is obtained by application
of $\alpha \circ \Gacirc(\alpha)$ to the graph:
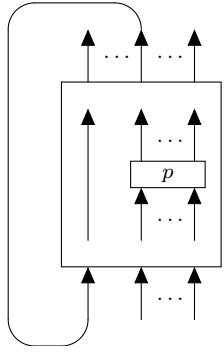
\begin{align*}
\begin{tikzpicture}[line cap=round,line join=round,>=triangle 45,x=0.7cm,y=0.7cm]
\clip(-1.7,-3.5) rectangle (2.7,3.2);
\draw [line width=.4pt] (0.8,0.)-- (2.2,0.);
\draw [line width=.4pt] (2.2,0.)-- (2.2,-0.5);
\draw [line width=.4pt] (2.2,-0.5)-- (0.8,-0.5);
\draw [line width=.4pt] (0.8,-0.5)-- (0.8,0.);
\draw (1.2,0.05) node[anchor=north west] {\scriptsize $p$};
\draw [->,line width=.4pt] (1.,0.) -- (1.,1.);
\draw [->,line width=.4pt] (2.,0.) -- (2.,1.);
\draw [->,line width=.4pt] (1.,-1.5) -- (1,-0.5);
\draw [->,line width=.4pt] (2.,-1.5) -- (2.,-0.5);
\draw (1.1,-0.9) node[anchor=north west] {\scriptsize $\ldots$};
\draw (1.1,0.6) node[anchor=north west] {\scriptsize $\ldots$};
\draw [->,line width=.4pt] (0.,-1.5) -- (0.,1.);
\draw [line width=.4pt] (-0.5,-2) -- (2.5,-2);
\draw [line width=.4pt] (2.5,-2) -- (2.5,1.5);
\draw [line width=.4pt] (2.5,1.5) -- (-0.5,1.5);
\draw [line width=.4pt] (-0.5,1.5) -- (-0.5,-2.);
\draw [->,line width=.4pt] (0.,-3.) -- (0.,-2.);
\draw [->,line width=.4pt] (1.,-3.) -- (1.,-2.);
\draw [->,line width=.4pt] (2.,-3.) -- (2.,-2.);
\draw (1.1,-2.4) node[anchor=north west] {\scriptsize $\ldots$};
\draw [->,line width=.4pt] (0.,1.5) -- (0.,2.5);
\draw [->,line width=.4pt] (1.,1.5) -- (1.,2.5);
\draw [->,line width=.4pt] (2.,1.5) -- (2.,2.5);
\draw (0.1,2.2) node[anchor=north west] {\scriptsize $\ldots$};
\draw (1.1,2.2) node[anchor=north west] {\scriptsize $\ldots$};
\draw [shift={(-0.5,-3)},line width=0.4pt]  plot[domain=4.71238898:6.283185307,variable=\t]({1.*0.5*cos(\t r)+0.*0.5*sin(\t r)},{0.*0.5*cos(\t r)+1.*0.5*sin(\t r)});
\draw [line width=.4pt] (-0.5,-3.5) -- (-1.,-3.5);
\draw [shift={(-1.,-3)},line width=0.4pt]  plot[domain=3.141592654:4.71238898,variable=\t]({1.*0.5*cos(\t r)+0.*0.5*sin(\t r)},{0.*0.5*cos(\t r)+1.*0.5*sin(\t r)});
\draw [line width=.4pt] (-1.5,-3.) -- (-1.5,2.5);
\draw [shift={(-1.,2.5)},line width=0.4pt]  plot[domain=1.570796327:3.141592654,variable=\t]({1.*0.5*cos(\t r)+0.*0.5*sin(\t r)},{0.*0.5*cos(\t r)+1.*0.5*sin(\t r)});
\draw [line width=.4pt] (-1.,3.) -- (0.5,3.);
\draw [shift={(0.5,2.5)},line width=0.4pt]  plot[domain=0.:1.570796327,variable=\t]({1.*0.5*cos(\t r)+0.*0.5*sin(\t r)},{0.*0.5*cos(\t r)+1.*0.5*sin(\t r)});
\end{tikzpicture}
\end{align*}
where the curved edge relate the first edge at the bottom to the $j$-th edge on the top.
As $\alpha \circ \Gacirc(\alpha)=\alpha \circ \mu_P$, $t_{1,j}(I*p)$ is obtained by application
of $\alpha$ to the graph:
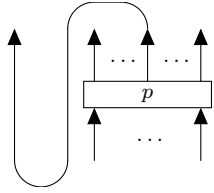
\begin{align*}
\begin{tikzpicture}[line cap=round,line join=round,>=triangle 45,x=0.7cm,y=0.7cm]
\clip(-0.7,-2) rectangle (3.2,1.5);
\draw [line width=.4pt] (0.8,0.)-- (3.2,0.);
\draw [line width=.4pt] (3.2,0.)-- (3.2,-0.5);
\draw [line width=.4pt] (3.2,-0.5)-- (0.8,-0.5);
\draw [line width=.4pt] (0.8,-0.5)-- (0.8,0.);
\draw (1.7,0.05) node[anchor=north west] {\scriptsize $p$};
\draw [->,line width=.4pt] (1.,0.) -- (1.,1.);
\draw [->,line width=.4pt] (2.,0.) -- (2.,1.);
\draw [->,line width=.4pt] (3.,0.) -- (3.,1.);
\draw [->,line width=.4pt] (1.,-1.5) -- (1,-0.5);
\draw [->,line width=.4pt] (3.,-1.5) -- (3.,-0.5);
\draw (1.6,-0.9) node[anchor=north west] {\scriptsize $\ldots$};
\draw (1.1,0.6) node[anchor=north west] {\scriptsize $\ldots$};
\draw (2.1,0.6) node[anchor=north west] {\scriptsize $\ldots$};
\draw [->,line width=.4pt] (-0.5,-1.5) -- (-0.5,1.);
\draw [shift={(0.,-1.5)},line width=0.4pt]  plot[domain=3.141592654:6.283185307,variable=\t]({1.*0.5*cos(\t r)+0.*0.5*sin(\t r)},{0.*0.5*cos(\t r)+1.*0.5*sin(\t r)});
\draw [line width=.4pt] (0.5,-1.5) -- (0.5,1);
\draw [shift={(1.,1.)},line width=0.4pt]  plot[domain=1.570796327:3.141592654,variable=\t]({1.*0.5*cos(\t r)+0.*0.5*sin(\t r)},{0.*0.5*cos(\t r)+1.*0.5*sin(\t r)});
\draw [line width=.4pt] (1.,1.5) -- (1.5,1.5);
\draw [shift={(1.5,1.)},line width=0.4pt]  plot[domain=0.:1.570796327,variable=\t]({1.*0.5*cos(\t r)+0.*0.5*sin(\t r)},{0.*0.5*cos(\t r)+1.*0.5*sin(\t r)});
\end{tikzpicture}
\end{align*}
where the curved edge relate the first edge on the bottom to the $j$-th edge on the top (note that this edge is also
the $(j-1)$-th outgoing the vertex decorated by $p$). As $\alpha$ is a $\sym\times \sym^{op}$ morphism,
we obtain that this is $(1,\ldots,j-1)\cdot \alpha \circ \nu_P(p)$, that is to say
$(1,\ldots,j-1)\cdot p$. \end{proof}

\begin{prop} \label{proptroptogamma}
Any TraP  is a $\Gacirc$-algebra.
\end{prop}
\begin{proof} 
Let $P$ be a TraP.
From Proposition \ref{prop:PdecGrc}, we obtain a unique TraP morphism $\alpha_P:\Gacirc(P)\longrightarrow P$,
such that for any $(k,l)\in  \N_0^2$, for any $p\in P(k,l)$, $\alpha_P$ sends the graph $\nu_P(P)$ to $P$.
The map $\alpha_P \circ \Gacirc(\alpha):\Gacirc\circ \Gacirc (P)\longrightarrow P$
is a TraP morphism, sending, for any graph $G\in \Gacirc(P)$, $\nu_{\Gacirc(P)}(G)$ to $\alpha(G)$. 
It is not difficult to see that $\mu_P:\Gacirc\circ \Gacirc(P)\longrightarrow \Gacirc(P)$ is a TraP morphism.
Hence, $\alpha_P \circ \mu_P:\Gacirc\circ \Gacirc (P)\longrightarrow P$
is a TraP morphism, sending, for any graph $G\in \Gacirc(P)$, $\nu_{\Gacirc(P)}(G)$ to $\alpha_P(G)$. 
As $\Gacirc\circ \Gacirc(P)$ is generated by the elements $\mu_P(G)$, both these morphisms coincide:
\[\alpha_P \circ \Gacirc(\alpha)=\alpha \circ \mu_P.\]
For any $p\in P$, by construction of $\alpha_P$, $\alpha_P \circ \nu_P(p)=p$, so:
\[\alpha_P\circ \nu_P=Id_P.\]
Therefore, $P$ is a $\Gacirc$-algebra.
\end{proof}

\begin{cor} \label{cor:lien_Trap_wProP}
The categories  $\Trap$ of TraPs and $\Gacirc-\mathbf{Alg}$ of $\Gacirc$-algebras are isomorphic.
\end{cor}

\begin{proof}
We defined in Propositions \ref{propgammatotrop} and \ref{proptroptogamma}  two functors
\begin{align*}
\mathcal{F}&:\Trap\longrightarrow\Gacirc-\mathbf{Alg},&
\mathcal{G}&:\Gacirc-\mathbf{Alg}\longrightarrow\Trap.
\end{align*}
Let $P$ be a TraP and $P'$ the TraP $\mathcal{G}\circ \mathcal{F}(P)$, with concatenation $*'$
and trace operators $t'_{i,j}$. 
We set $\mathcal{F}(P):=(P,\alpha_P)$: in other words, $\alpha_P$ is the TraP morphism of Proposition
\ref{prop:PdecGrc}. For any $p,q\in P$:
\[p*'q=\alpha_P (\nu_P(p)*\nu_P(q))=p*p',\]
where in the middle term $*$ is the concatenation in the TraP $\Gacirc(P)$. Therefore, $*=*'$.
If $p\in P(k,l)$, $(i,j)\in [k]\times [l]$, then $t'_{i,j}$ is obtained by the application of $\alpha_P$ to the graph:
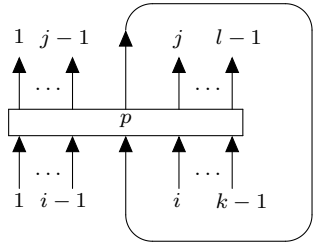
\[\begin{tikzpicture}[line cap=round,line join=round,>=triangle 45,x=0.7cm,y=0.7cm]
\clip(0.8,-2.6) rectangle (6.5,2.1);
\draw [line width=.4pt] (0.8,0.)-- (5.2,0.);
\draw [line width=.4pt] (5.2,0.)-- (5.2,-0.5);
\draw [line width=.4pt] (5.2,-0.5)-- (0.8,-0.5);
\draw [line width=.4pt] (0.8,-0.5)-- (0.8,0.);
\draw (2.7,0.1) node[anchor=north west] {\scriptsize $p$};
\draw [->,line width=.4pt] (1.,0.) -- (1.,1.);
\draw [->,line width=.4pt] (2.,0.) -- (2.,1.);
\draw [->,line width=.4pt] (3.,0.) -- (3.,1.5);
\draw [->,line width=.4pt] (4.,0.) -- (4.,1.);
\draw [->,line width=.4pt] (5.,0.) -- (5.,1.);
\draw [->,line width=.4pt] (1.,-1.5) -- (1,-0.5);
\draw [->,line width=.4pt] (2.,-1.5) -- (2.,-0.5);
\draw [->,line width=.4pt] (3.,-2) -- (3,-0.5);
\draw [->,line width=.4pt] (4.,-1.5) -- (4.,-0.5);
\draw [->,line width=.4pt] (5.,-1.5) -- (5.,-0.5);
\draw (0.7,-1.4) node[anchor=north west] {\scriptsize $1$};
\draw (1.1,-1.) node[anchor=north west] {\scriptsize $\ldots$};
\draw (1.2,-1.4) node[anchor=north west] {\scriptsize $i-1$};
\draw (3.7,-1.4) node[anchor=north west] {\scriptsize $i$};
\draw (4.1,-1.) node[anchor=north west] {\scriptsize $\ldots$};
\draw (4.5,-1.4) node[anchor=north west] {\scriptsize $k-1$};
\draw (0.7,1.7) node[anchor=north west] {\scriptsize $1$};
\draw (1.1,0.6) node[anchor=north west] {\scriptsize $\ldots$};
\draw (1.2,1.7) node[anchor=north west] {\scriptsize $j-1$};
\draw (3.7,1.7) node[anchor=north west] {\scriptsize $j$};
\draw (4.1,0.6) node[anchor=north west] {\scriptsize $\ldots$};
\draw (4.5,1.7) node[anchor=north west] {\scriptsize $l-1$};
\draw [shift={(3.5,-2.)},line width=0.4pt]  plot[domain=3.141592653589793:4.71238898,variable=\t]({1.*0.5*cos(\t r)+0.*0.5*sin(\t r)},{0.*0.5*cos(\t r)+1.*0.5*sin(\t r)});
\draw [line width=.4pt] (3.5,-2.5) -- (6.,-2.5);
\draw [shift={(6.,-2.)},line width=0.4pt]  plot[domain=4.71238898:6.283185307,variable=\t]({1.*0.5*cos(\t r)+0.*0.5*sin(\t r)},{0.*0.5*cos(\t r)+1.*0.5*sin(\t r)});
\draw [shift={(3.5,1.5)},line width=0.4pt]  plot[domain=1.570796327:3.141592653589793,variable=\t]({1.*0.5*cos(\t r)+0.*0.5*sin(\t r)},{0.*0.5*cos(\t r)+1.*0.5*sin(\t r)});
\draw [line width=.4pt] (3.5,2.) -- (6.,2.);
\draw [shift={(6,1.5)},line width=0.4pt]  plot[domain=0.:1.570796327,variable=\t]({1.*0.5*cos(\t r)+0.*0.5*sin(\t r)},{0.*0.5*cos(\t r)+1.*0.5*sin(\t r)});
\draw [line width=.4pt] (6.5,-2.) -- (6.5,1.5);
\end{tikzpicture}\]
which is $t_{i,j}(\nu_P(p))$, where here $t_{i,j}$ is the trace operator of $\Gacirc(P)$. As $\alpha_P$ is a TraP morphism:
\[t'_{i,j}(p)=\alpha_P \circ t_{i,j}\circ \nu_P(p)=t_{i,j}\circ \alpha_P \circ \nu_P(p)=t_{i,j}(p),\]
so $P'=P$ and $\mathcal{G}\circ \mathcal{F}$ is the identity functor of $\Trap$.

Let now $(P,\alpha)$ be a $\Gacirc$-algebra and let us consider $(P',\alpha')$ be the $\Gacirc$-algebra
$ \mathcal{F}\circ  \mathcal{G}(P)$. Both $\alpha$ and $\alpha'$ are TraP morphisms from
$\Gacirc(P)$ to $\mathcal{G}(P)$; for any $p\in P$,
\[\alpha \circ \nu_P(p)=\alpha'\circ \nu_P(p)=p.\]
As $\Gacirc(P)$ is generated, as a TraP, by the elements $\nu_P(p)$, $\alpha=\alpha'$,
so  $\mathcal{F}\circ \mathcal{G}$ is the identity functor of $\Gacirc-\mathbf{Alg}$.
\end{proof}

\begin{remark}
$\Gacirc$-algebras appear in the literature \cite{Merkulov2006,Merkulov2009,Merkulov2010,Merkulov2010-2} 
under the name of unitary wheeled props; see \cite{Merkulov2009} for the description of the monad of graphs
used for wheeled props, and \cite{Merkulov2010,Merkulov2010-2} for applications of wheeled props.

We defined the structure of TraPs having their application to Feynman graphs in QFT in mind.
 Since our focus in this paper is on traces for which we need an explicit realisation of the structures under consideration, we choose to 
 keep here the terminology TraP.
\end{remark}

\subsection{TraPs are ProPs}
 
	TraPs can be equipped with a ProP structure   as a result of the fact that both the trace  and composition of morphisms   can be expressed in terms of a dual pairing. 
  Corollary \ref{cor:lien_Trap_wProP}, 
		  yields an isomorphism between the categories of TraPs and wheeled ProPs. It is known that wheeled ProPs are ProPs, and we give here a 
		  detailed  construction of the ProP structure on our TraPs, showing how the partial 
		 trace maps (referred to as contractions  by Merkulov) of wheeled ProPs 
	give rise to a vertical composition, and therefore to a ProP structure, a fact readily observed in \cite[Remarks 2.1.1]{Merkulov2009}.

\begin{prop} \label{prop:trop_prop}
Let $P$ be a TraP. We define a vertical composition in the following way:
\begin{align*}
&\forall p\in P(k,l),\:\forall q\in P(l,m),&
q\circ p&=t_{k+1,1}\circ \ldots \circ t_{k+l-1,l-1}\circ t_{k+l,l}(p*q).
\end{align*}
Then $P$ is a ProP.
\end{prop}

\begin{example}
In the TraP of graphs $\Gr$:
\end{example}

\begin{center}
\begin{tikzpicture}[line cap=round,line join=round,>=triangle 45,x=0.5cm,y=0.5cm]
\clip(-2.5,-4.) rectangle (1.,4.);
\draw [line width=0.4pt] (-2.,1.)-- (0.5,1.);
\draw [line width=0.4pt] (0.5,1.)-- (0.5,-1.);
\draw [line width=0.4pt] (0.5,-1.)-- (-2.,-1.);
\draw [line width=0.4pt] (-2.,-1.)-- (-2.,1.);
\draw [->,line width=0.4pt] (-1.5,1.) -- (-1.5,3.);
\draw [->,line width=0.4pt] (0.,1.) -- (0.,3.);
\draw [->,line width=0.4pt] (-1.5,-3.) -- (-1.5,-1.);
\draw [->,line width=0.4pt] (0.,-3.) -- (0.,-1.);
\draw (-1.25,0.5) node[anchor=north west] {$H$};
\draw (-1.8,-3) node[anchor=north west] {$1$};
\draw (-0.3,-3) node[anchor=north west] {$l$};
\draw (-1.4,-2.2) node[anchor=north west] {$\ldots$};
\draw (-1.8,4.2) node[anchor=north west] {$1$};
\draw (-0.3,4.) node[anchor=north west] {$m$};
\draw (-1.4,2.) node[anchor=north west] {$\ldots$};
\end{tikzpicture}
$\substack{\displaystyle \circ\\ \vspace{3cm}}$
\begin{tikzpicture}[line cap=round,line join=round,>=triangle 45,x=0.5cm,y=0.5cm]
\clip(-2.5,-4.) rectangle (0.7,4.);
\draw [line width=0.4pt] (-2.,1.)-- (0.5,1.);
\draw [line width=0.4pt] (0.5,1.)-- (0.5,-1.);
\draw [line width=0.4pt] (0.5,-1.)-- (-2.,-1.);
\draw [line width=0.4pt] (-2.,-1.)-- (-2.,1.);
\draw [->,line width=0.4pt] (-1.5,1.) -- (-1.5,3.);
\draw [->,line width=0.4pt] (0.,1.) -- (0.,3.);
\draw [->,line width=0.4pt] (-1.5,-3.) -- (-1.5,-1.);
\draw [->,line width=0.4pt] (0.,-3.) -- (0.,-1.);
\draw (-1.25,0.5) node[anchor=north west] {$G$};
\draw (-1.8,-3) node[anchor=north west] {$1$};
\draw (-0.3,-3) node[anchor=north west] {$k$};
\draw (-1.4,-2.2) node[anchor=north west] {$\ldots$};
\draw (-1.8,4.2) node[anchor=north west] {$1$};
\draw (-0.3,4.2) node[anchor=north west] {$l$};
\draw (-1.4,2.) node[anchor=north west] {$\ldots$};
\end{tikzpicture}
$\substack{\displaystyle =\\ \vspace{3cm}}$
\begin{tikzpicture}[line cap=round,line join=round,>=triangle 45,x=0.5cm,y=0.5cm]
\clip(-2.5,-4.) rectangle (7.,5.);
\draw [line width=0.4pt] (-2.,1.)-- (0.5,1.);
\draw [line width=0.4pt] (0.5,1.)-- (0.5,-1.);
\draw [line width=0.4pt] (0.5,-1.)-- (-2.,-1.);
\draw [line width=0.4pt] (-2.,-1.)-- (-2.,1.);
\draw [line width=0.4pt] (-1.5,1.) -- (-1.5,2.);
\draw [line width=0.4pt] (0.,1.) -- (0.,2.);
\draw [->,line width=0.4pt] (-1.5,-3.) -- (-1.5,-1.);
\draw [->,line width=0.4pt] (0.,-3.) -- (0.,-1.);
\draw (-1.25,0.5) node[anchor=north west] {$G$};
\draw (-1.8,-3) node[anchor=north west] {$1$};
\draw (-0.3,-3) node[anchor=north west] {$k$};
\draw (-1.4,-2.2) node[anchor=north west] {$\ldots$};
\draw (-1.4,2.) node[anchor=north west] {$\ldots$};
\draw [shift={(0.,2.)},line width=0.4pt]  plot[domain=0.:3.141592653589793,variable=\t]({1.*1.5*cos(\t r)+0.*1.5*sin(\t r)},{0.*1.5*cos(\t r)+1.*1.5*sin(\t r)});
\draw [shift={(3.,-2.)},line width=0.4pt]  plot[domain=3.141592653589793:6.283185307179586,variable=\t]({1.*1.5*cos(\t r)+0.*1.5*sin(\t r)},{0.*1.5*cos(\t r)+1.*1.5*sin(\t r)});
\draw [->,line width=0.4pt] (1.5,2.) -- (1.5,-2.);
\draw [shift={(1.5,2.)},line width=0.4pt]  plot[domain=0.:3.141592653589793,variable=\t]({1.*1.5*cos(\t r)+0.*1.5*sin(\t r)},{0.*1.5*cos(\t r)+1.*1.5*sin(\t r)});
\draw [shift={(4.5,-2.)},line width=0.4pt]  plot[domain=3.141592653589793:6.283185307179586,variable=\t]({1.*1.5*cos(\t r)+0.*1.5*sin(\t r)},{0.*1.5*cos(\t r)+1.*1.5*sin(\t r)});
\draw [->,line width=0.4pt] (3.,2.) -- (3.,-2.);
\draw [line width=0.4pt] (4.,1.)-- (6.5,1.);
\draw [line width=0.4pt] (6.5,1.)-- (6.5,-1.);
\draw [line width=0.4pt] (6.5,-1.)-- (4.,-1.);
\draw [line width=0.4pt] (4.,-1.)-- (4.,1.);
\draw [->,line width=0.4pt] (4.5,1.) -- (4.5,3.);
\draw [->,line width=0.4pt] (6.,1.) -- (6.,3.);
\draw [line width=0.4pt] (4.5,-2.) -- (4.5,-1.);
\draw [line width=0.4pt] (6.,-2.) -- (6.,-1.);
\draw (4.75,0.5) node[anchor=north west] {$H$};
\draw (4.2,4.2) node[anchor=north west] {$1$};
\draw (5.7,4) node[anchor=north west] {$m$};
\draw (4.4,-2.2) node[anchor=north west] {$\ldots$};
\draw (4.6,2.) node[anchor=north west] {$\ldots$};
\end{tikzpicture}
$\substack{\displaystyle =\\ \vspace{3cm}}$
\begin{tikzpicture}[line cap=round,line join=round,>=triangle 45,x=0.5cm,y=0.5cm]
\clip(-2.5,-4.) rectangle (1.,8.);
\draw [line width=0.4pt] (-2.,1.)-- (0.5,1.);
\draw [line width=0.4pt] (0.5,1.)-- (0.5,-1.);
\draw [line width=0.4pt] (0.5,-1.)-- (-2.,-1.);
\draw [line width=0.4pt] (-2.,-1.)-- (-2.,1.);
\draw [->,line width=0.4pt] (-1.5,1.) -- (-1.5,3.);
\draw [->,line width=0.4pt] (0.,1.) -- (0.,3.);
\draw [->,line width=0.4pt] (-1.5,-3.) -- (-1.5,-1.);
\draw [->,line width=0.4pt] (0.,-3.) -- (0.,-1.);
\draw (-1.25,0.5) node[anchor=north west] {$G$};
\draw (-1.8,-3) node[anchor=north west] {$1$};
\draw (-0.3,-3) node[anchor=north west] {$k$};
\draw (-1.4,-2.2) node[anchor=north west] {$\ldots$};
\draw (-1.4,2.) node[anchor=north west] {$\ldots$};
\draw [line width=0.4pt] (-2.,5.)-- (0.5,5.);
\draw [line width=0.4pt] (0.5,5.)-- (0.5,3.);
\draw [line width=0.4pt] (0.5,3.)-- (-2.,3.);
\draw [line width=0.4pt] (-2.,3.)-- (-2.,5.);
\draw [->,line width=0.4pt] (-1.5,5.) -- (-1.5,7.);
\draw [->,line width=0.4pt] (0.,5.) -- (0.,7.);
\draw (-1.25,4.5) node[anchor=north west] {$H$};
\draw (-1.8,8.2) node[anchor=north west] {$1$};
\draw (-0.3,8.) node[anchor=north west] {$m$};
\draw (-1.4,6.) node[anchor=north west] {$\ldots$};
\end{tikzpicture}

\vspace{-1.5cm}
\end{center}

\begin{proof}
It is enough to prove it for a free TraP $\PGr(X)$, as any TraP is the quotient of such an object. 
If $G \in \PGr(X)(k,l)$ and $H\in \PGr(X)(l,m)$ are two $X$-decorated  planar graphs, then by definition 
of the partial trace maps,
$G*H$ is the $X$-decorated  planar graph obtained by grafting together the output edge $i$ of $G$
with the input edge $j$ of $H$ for any $i\in [k]$; this is precisely the vertical concatenation of  graphs,
adapted to $X$-decorated  planar graphs. So it is indeed a ProP.
\end{proof}

\begin{example}
\begin{enumerate}
\item For graphs, we recover the composition defined in Section \ref{sectiongraphes},
extended to graphs.
\item For the $\Hom_V$ TraP, for any $F=f_1\ldots f_k \otimes v_1\ldots v_l\in V^{*\otimes k}\otimes V^{\otimes l}\approx
\Hom(V^{\otimes k},V^{\otimes l})$
and $G=g_1\ldots g_l \otimes w_1\ldots w_n \in \in V^{*\otimes l}\otimes V^{\otimes m}\approx
\Hom(V^{\otimes l},V^{\otimes m})$:
\[F\circ G=g_1(v_1)\ldots g_l(v_l)f_1\ldots f_k \otimes  w_1\ldots w_n.\]
This is the composition of $\Hom_V$.
\end{enumerate}
\end{example}
Applied to the TraP $\Hom_V^c$ of Proposition \ref{lem:TrHom}, this method allows to recover the ProP $\Hom_V^c$ of Theorem 
\ref{thm:Hom_V_generalised}.
\begin{prop}\label{prop:TropPropHomc}
	Let $V$ be a Fr\'echet nuclear space. The ProP built from the TraP  $(\Hom_V^c(k, l))_{k,l\geq 0}$ as in Proposition \ref{prop:trop_prop} is 
	isomorphic, as a ProP, to the ProP $\Hom_V^c$ of Theorem \ref{thm:Hom_V_generalised}.
\end{prop}
\begin{proof}
	It is enough to check that the composition of two homomorphisms will give the right object. Let $f=\Hom_V^c(k,l)$ and $g=\Hom_V^c(l,m)$. 
	By Equation \eqref{eq:E_prime_otimes_F} we can write 
	\begin{equation*}
	f = \sum_\alpha \left((v_1^\alpha)^*\otimes\cdots\otimes(v_1^\alpha)^*\right)\otimes\left(w_1^\alpha\otimes\cdots\otimes w_k^\alpha\right), \qquad g = \sum_\beta \left((u_1^\beta)^*\otimes\cdots\otimes(u_m^\beta)^*\right)\otimes\left(r_1^\beta\otimes\cdots\otimes r_l^\beta\right).
	\end{equation*}
	Then the definition of the composition product of Proposition \ref{prop:trop_prop} implies
	\begin{equation*}
	f\circ g = \sum_\alpha\sum_\beta\left[\prod_{i=1}^l(v_i^\alpha)^*(r_i^\beta)\right]\left((u_1^\beta)^*\otimes\cdots\otimes(u_m^\beta)^*\right)\otimes\left(w_1^\alpha\otimes\cdots\otimes w_k^\alpha\right).
	\end{equation*}
	Using Equation \eqref{eq:echange_dual_prod}, we can apply Lemma \ref{lem:compo_dual_pairing} to the case $E_1=V^{\widehat\otimes  m}$, $E_2=V^{\widehat\otimes  l}$, 
	$E_3=V^{\widehat\otimes  k}$. The result then follows from this lemma and the observation that 
	\begin{equation*}
	\prod_{i=1}^l(v_i^\alpha)^*(r_i^\beta) = (v_1^\alpha\otimes\cdots\otimes v_1^\alpha)^*\otimes(r_1^\beta\otimes\cdots\otimes r_l^\beta)
	\end{equation*}
	for the duality pairing in $E_2$.
\end{proof}
We end this Subsection with a Corollary to Proposition \ref{prop:trop_prop}.
\begin{cor} \label{coro:generalised_traces}
Let $P$ be a TraP. For any $p\in P(k,k)$, we set:
\[\mathrm{Tr}(p)=t_{1,1}\circ \ldots \circ t_{k,k}(p).\]
\begin{enumerate}
\item For any $(k,l)\in  \N_0^2$, for any $(p,q)\in P(k,l)\times P(l,k)$,
\begin{align*}
\mathrm{Tr}(p\circ q)&=\mathrm{Tr}(q\circ p).
\end{align*}
\item For any $(k,l)\in  \N_0^2$, for any $(p,q)\in P(k,k)\times P(l,l)$,
\begin{align*}
\mathrm{Tr}(p*q)&=\mathrm{Tr}(p)\mathrm{Tr}(q).
\end{align*}\end{enumerate}\end{cor}

\begin{example}
\begin{enumerate}
\item In $\Gr$, for any graph $G\in \Gr(k,k)$, $\mathrm{Tr}(G)$ is obtained by
gluing together the $i$-th output edge with the $i$-th output edge of $G$. 
In particular, $\grapheo=\mathrm{Tr}(I)$. Graphically:

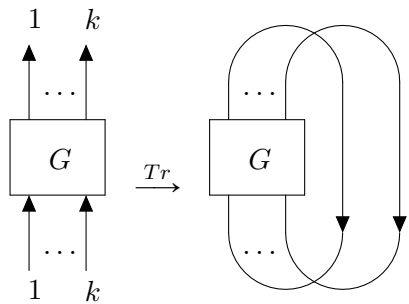
\begin{center}
\begin{tikzpicture}[line cap=round,line join=round,>=triangle 45,x=0.5cm,y=0.5cm]
\clip(-2.5,-4.) rectangle (1.,4.);
\draw [line width=0.4pt] (-2.,1.)-- (0.5,1.);
\draw [line width=0.4pt] (0.5,1.)-- (0.5,-1.);
\draw [line width=0.4pt] (0.5,-1.)-- (-2.,-1.);
\draw [line width=0.4pt] (-2.,-1.)-- (-2.,1.);
\draw [->,line width=0.4pt] (-1.5,1.) -- (-1.5,3.);
\draw [->,line width=0.4pt] (0.,1.) -- (0.,3.);
\draw [->,line width=0.4pt] (-1.5,-3.) -- (-1.5,-1.);
\draw [->,line width=0.4pt] (0.,-3.) -- (0.,-1.);
\draw (-1.25,0.5) node[anchor=north west] {$G$};
\draw (-1.8,-3) node[anchor=north west] {$1$};
\draw (-0.3,-3) node[anchor=north west] {$k$};
\draw (-1.4,-2.2) node[anchor=north west] {$\ldots$};
\draw (-1.8,4.2) node[anchor=north west] {$1$};
\draw (-0.3,4.2) node[anchor=north west] {$k$};
\draw (-1.4,2.) node[anchor=north west] {$\ldots$};
\end{tikzpicture}
$\substack{\displaystyle \stackrel{Tr}{\longrightarrow}\\ \vspace{3cm}}$
\begin{tikzpicture}[line cap=round,line join=round,>=triangle 45,x=0.5cm,y=0.5cm]
\clip(-2.5,-4.) rectangle (3.5,4.);
\draw [line width=0.4pt] (-2.,1.)-- (0.5,1.);
\draw [line width=0.4pt] (0.5,1.)-- (0.5,-1.);
\draw [line width=0.4pt] (0.5,-1.)-- (-2.,-1.);
\draw [line width=0.4pt] (-2.,-1.)-- (-2.,1.);
\draw [line width=0.4pt] (-1.5,1.) -- (-1.5,2.);
\draw [line width=0.4pt] (0.,1.) -- (0.,2.);
\draw [line width=0.4pt] (-1.5,-2.) -- (-1.5,-1.);
\draw [line width=0.4pt] (0.,-2.) -- (0.,-1.);
\draw (-1.25,0.5) node[anchor=north west] {$G$};
\draw (-1.4,-2.2) node[anchor=north west] {$\ldots$};
\draw (-1.4,2.) node[anchor=north west] {$\ldots$};
\draw [shift={(1.5,-2.)},line width=0.4pt]  plot[domain=3.141592653589793:6.283185307179586,variable=\t]({1.*1.5*cos(\t r)+0.*1.5*sin(\t r)},{0.*1.5*cos(\t r)+1.*1.5*sin(\t r)});
\draw [shift={(0.,-2.)},line width=0.4pt]  plot[domain=3.141592653589793:6.283185307179586,variable=\t]({1.*1.5*cos(\t r)+0.*1.5*sin(\t r)},{0.*1.5*cos(\t r)+1.*1.5*sin(\t r)});
\draw [shift={(1.5,2.)},line width=0.4pt]  plot[domain=0.:3.141592653589793,variable=\t]({1.*1.5*cos(\t r)+0.*1.5*sin(\t r)},{0.*1.5*cos(\t r)+1.*1.5*sin(\t r)});
\draw [shift={(0.,2.)},line width=0.4pt]  plot[domain=0.:3.141592653589793,variable=\t]({1.*1.5*cos(\t r)+0.*1.5*sin(\t r)},{0.*1.5*cos(\t r)+1.*1.5*sin(\t r)});
\draw [->,line width=0.4pt] (1.5,2.) -- (1.5,-2.);
\draw [->,line width=0.4pt] (3.,2.) -- (3.,-2.);
\end{tikzpicture}

\vspace{-1.5cm}
\end{center}

\item Let $V$ be a finite dimensional vector space of dimension $n$. In the TraP $\Hom_V$ introduced in Proposition 
\ref{prop:Trap_fin_dim_VS} we obtain a trace for morphisms $F:V^{\otimes k}\mapsto V^{\otimes k}$. Specialising to the case $k=1$, we recover the 
usual trace of linear endomorphisms: choose $(e_1,\cdots,e_n)$ a basis of $V$. Any morphism $f:V\mapsto V$ can be represented in this basis by 
	$\sum_{i,j=1}^na^f_{ij}e_i^*\otimes e_j$ for some complex numbers $a^f_{ij}$. Then 
	$\mathrm{Tr}(f)=\sum_{i,j=1}^na^f_{ij}e_i^*(e_j)=\sum_{i=1}^na^f_{ii}$.
$\mathrm{Tr}(f)$   lies in $\K$, is viewed here as an element of $\Hom_V(0,0)$ via the identification of a constant $\lambda$ in $\K$ to a linear map 
$x\longmapsto \lambda \, x$ on $\K$.

The vertical composition  $f\circ g= t_{2,1}(f * g)$ of two morphisms $f$ and $g$,  defined   according to Proposition \ref{prop:trop_prop}  is  
indeed represented by   the usual  matrix product:   
\[\sum_{i,j=1}^n\sum_{k,l=1}^na^f_{ik}\,a^g_{lj} \,e_i^*\otimes e_k^*(e_l)\otimes e_j=\sum_{i,j=1}^n\left(\sum_{k=1}^na^f_{ik}\,a^g_{kj} \right)\,e_i^*\otimes  e_j,\]
where $(a_{ij}^f)_{i, j}$, $(a_{ij}^g)_{i, j}$ are the matrix representations of $f$ and $g$ respectively.
\end{enumerate}
\end{example}

\begin{proof} Again, it is enough to prove the result for a free TraP $\PGr(X)$.

Let  $G\in \PGr(X)(k,l)$ and $H \in \PGr(X)(l,k)$ be two graphs.
Then $\mathrm{Tr}(H\circ G)$ is graphically represented by each of the graphs:

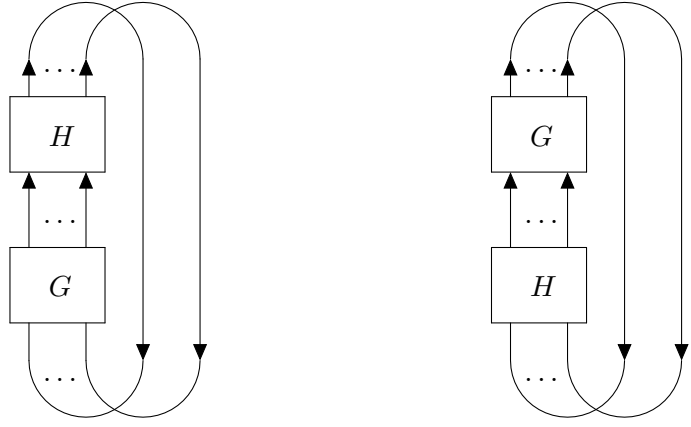
\begin{align*}
&\begin{tikzpicture}[line cap=round,line join=round,>=triangle 45,x=0.5cm,y=0.5cm]
\clip(-2.5,-4.) rectangle (3.5,8.);
\draw [line width=0.4pt] (-2.,1.)-- (0.5,1.);
\draw [line width=0.4pt] (0.5,1.)-- (0.5,-1.);
\draw [line width=0.4pt] (0.5,-1.)-- (-2.,-1.);
\draw [line width=0.4pt] (-2.,-1.)-- (-2.,1.);
\draw [->,line width=0.4pt] (-1.5,1.) -- (-1.5,3.);
\draw [->,line width=0.4pt] (0.,1.) -- (0.,3.);
\draw [line width=0.4pt] (-1.5,-2.) -- (-1.5,-1.);
\draw [line width=0.4pt] (0.,-2.) -- (0.,-1.);
\draw (-1.25,0.5) node[anchor=north west] {$G$};
\draw (-1.4,-2.2) node[anchor=north west] {$\ldots$};
\draw (-1.4,2.) node[anchor=north west] {$\ldots$};
\draw [line width=0.4pt] (-2.,5.)-- (0.5,5.);
\draw [line width=0.4pt] (0.5,5.)-- (0.5,3.);
\draw [line width=0.4pt] (0.5,3.)-- (-2.,3.);
\draw [line width=0.4pt] (-2.,3.)-- (-2.,5.);
\draw [->,line width=0.4pt] (-1.5,5.) -- (-1.5,6.);
\draw [->,line width=0.4pt] (0.,5.) -- (0.,6.);
\draw (-1.25,4.5) node[anchor=north west] {$H$};
\draw (-1.4,6.) node[anchor=north west] {$\ldots$};
\draw [shift={(1.5,-2.)},line width=0.4pt]  plot[domain=3.141592653589793:6.283185307179586,variable=\t]({1.*1.5*cos(\t r)+0.*1.5*sin(\t r)},{0.*1.5*cos(\t r)+1.*1.5*sin(\t r)});
\draw [shift={(0.,-2.)},line width=0.4pt]  plot[domain=3.141592653589793:6.283185307179586,variable=\t]({1.*1.5*cos(\t r)+0.*1.5*sin(\t r)},{0.*1.5*cos(\t r)+1.*1.5*sin(\t r)});
\draw [shift={(1.5,6.)},line width=0.4pt]  plot[domain=0.:3.141592653589793,variable=\t]({1.*1.5*cos(\t r)+0.*1.5*sin(\t r)},{0.*1.5*cos(\t r)+1.*1.5*sin(\t r)});
\draw [shift={(0.,6.)},line width=0.4pt]  plot[domain=0.:3.141592653589793,variable=\t]({1.*1.5*cos(\t r)+0.*1.5*sin(\t r)},{0.*1.5*cos(\t r)+1.*1.5*sin(\t r)});
\draw [->,line width=0.4pt] (1.5,6.) -- (1.5,-2.);
\draw [->,line width=0.4pt] (3.,6.) -- (3.,-2.);
\end{tikzpicture}& 
\begin{tikzpicture}[line cap=round,line join=round,>=triangle 45,x=0.5cm,y=0.5cm]
\clip(-2.5,-4.) rectangle (3.5,8.);
\draw [line width=0.4pt] (-2.,1.)-- (0.5,1.);
\draw [line width=0.4pt] (0.5,1.)-- (0.5,-1.);
\draw [line width=0.4pt] (0.5,-1.)-- (-2.,-1.);
\draw [line width=0.4pt] (-2.,-1.)-- (-2.,1.);
\draw [->,line width=0.4pt] (-1.5,1.) -- (-1.5,3.);
\draw [->,line width=0.4pt] (0.,1.) -- (0.,3.);
\draw [line width=0.4pt] (-1.5,-2.) -- (-1.5,-1.);
\draw [line width=0.4pt] (0.,-2.) -- (0.,-1.);
\draw (-1.25,0.5) node[anchor=north west] {$H$};
\draw (-1.4,-2.2) node[anchor=north west] {$\ldots$};
\draw (-1.4,2.) node[anchor=north west] {$\ldots$};
\draw [line width=0.4pt] (-2.,5.)-- (0.5,5.);
\draw [line width=0.4pt] (0.5,5.)-- (0.5,3.);
\draw [line width=0.4pt] (0.5,3.)-- (-2.,3.);
\draw [line width=0.4pt] (-2.,3.)-- (-2.,5.);
\draw [->,line width=0.4pt] (-1.5,5.) -- (-1.5,6.);
\draw [->,line width=0.4pt] (0.,5.) -- (0.,6.);
\draw (-1.25,4.5) node[anchor=north west] {$G$};
\draw (-1.4,6.) node[anchor=north west] {$\ldots$};
\draw [shift={(1.5,-2.)},line width=0.4pt]  plot[domain=3.141592653589793:6.283185307179586,variable=\t]({1.*1.5*cos(\t r)+0.*1.5*sin(\t r)},{0.*1.5*cos(\t r)+1.*1.5*sin(\t r)});
\draw [shift={(0.,-2.)},line width=0.4pt]  plot[domain=3.141592653589793:6.283185307179586,variable=\t]({1.*1.5*cos(\t r)+0.*1.5*sin(\t r)},{0.*1.5*cos(\t r)+1.*1.5*sin(\t r)});
\draw [shift={(1.5,6.)},line width=0.4pt]  plot[domain=0.:3.141592653589793,variable=\t]({1.*1.5*cos(\t r)+0.*1.5*sin(\t r)},{0.*1.5*cos(\t r)+1.*1.5*sin(\t r)});
\draw [shift={(0.,6.)},line width=0.4pt]  plot[domain=0.:3.141592653589793,variable=\t]({1.*1.5*cos(\t r)+0.*1.5*sin(\t r)},{0.*1.5*cos(\t r)+1.*1.5*sin(\t r)});
\draw [->,line width=0.4pt] (1.5,6.) -- (1.5,-2.);
\draw [->,line width=0.4pt] (3.,6.) -- (3.,-2.);
\end{tikzpicture}
\end{align*}
which are the same. So $\mathrm{Tr}(H\circ G)=\mathrm{Tr}(G\circ H)$. 
Moreover, the graph $\mathrm{Tr}(G*H)$ is represented by the 
graph

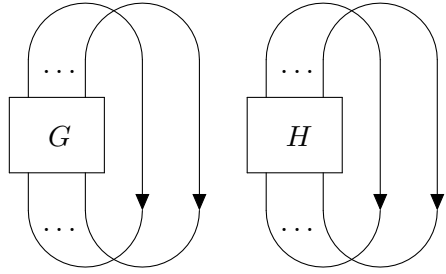
\begin{center}
\begin{tikzpicture}[line cap=round,line join=round,>=triangle 45,x=0.5cm,y=0.5cm]
\clip(-2.5,-4.) rectangle (3.5,4.);
\draw [line width=0.4pt] (-2.,1.)-- (0.5,1.);
\draw [line width=0.4pt] (0.5,1.)-- (0.5,-1.);
\draw [line width=0.4pt] (0.5,-1.)-- (-2.,-1.);
\draw [line width=0.4pt] (-2.,-1.)-- (-2.,1.);
\draw [line width=0.4pt] (-1.5,1.) -- (-1.5,2.);
\draw [line width=0.4pt] (0.,1.) -- (0.,2.);
\draw [line width=0.4pt] (-1.5,-2.) -- (-1.5,-1.);
\draw [line width=0.4pt] (0.,-2.) -- (0.,-1.);
\draw (-1.25,0.5) node[anchor=north west] {$G$};
\draw (-1.4,-2.2) node[anchor=north west] {$\ldots$};
\draw (-1.4,2.) node[anchor=north west] {$\ldots$};
\draw [shift={(1.5,-2.)},line width=0.4pt]  plot[domain=3.141592653589793:6.283185307179586,variable=\t]({1.*1.5*cos(\t r)+0.*1.5*sin(\t r)},{0.*1.5*cos(\t r)+1.*1.5*sin(\t r)});
\draw [shift={(0.,-2.)},line width=0.4pt]  plot[domain=3.141592653589793:6.283185307179586,variable=\t]({1.*1.5*cos(\t r)+0.*1.5*sin(\t r)},{0.*1.5*cos(\t r)+1.*1.5*sin(\t r)});
\draw [shift={(1.5,2.)},line width=0.4pt]  plot[domain=0.:3.141592653589793,variable=\t]({1.*1.5*cos(\t r)+0.*1.5*sin(\t r)},{0.*1.5*cos(\t r)+1.*1.5*sin(\t r)});
\draw [shift={(0.,2.)},line width=0.4pt]  plot[domain=0.:3.141592653589793,variable=\t]({1.*1.5*cos(\t r)+0.*1.5*sin(\t r)},{0.*1.5*cos(\t r)+1.*1.5*sin(\t r)});
\draw [->,line width=0.4pt] (1.5,2.) -- (1.5,-2.);
\draw [->,line width=0.4pt] (3.,2.) -- (3.,-2.);
\end{tikzpicture}
\begin{tikzpicture}[line cap=round,line join=round,>=triangle 45,x=0.5cm,y=0.5cm]
\clip(-2.5,-4.) rectangle (6.,4.);
\draw [line width=0.4pt] (-2.,1.)-- (0.5,1.);
\draw [line width=0.4pt] (0.5,1.)-- (0.5,-1.);
\draw [line width=0.4pt] (0.5,-1.)-- (-2.,-1.);
\draw [line width=0.4pt] (-2.,-1.)-- (-2.,1.);
\draw [line width=0.4pt] (-1.5,1.) -- (-1.5,2.);
\draw [line width=0.4pt] (0.,1.) -- (0.,2.);
\draw [line width=0.4pt] (-1.5,-2.) -- (-1.5,-1.);
\draw [line width=0.4pt] (0.,-2.) -- (0.,-1.);
\draw (-1.25,0.5) node[anchor=north west] {$H$};
\draw (-1.4,-2.2) node[anchor=north west] {$\ldots$};
\draw (-1.4,2.) node[anchor=north west] {$\ldots$};
\draw [shift={(1.5,-2.)},line width=0.4pt]  plot[domain=3.141592653589793:6.283185307179586,variable=\t]({1.*1.5*cos(\t r)+0.*1.5*sin(\t r)},{0.*1.5*cos(\t r)+1.*1.5*sin(\t r)});
\draw [shift={(0.,-2.)},line width=0.4pt]  plot[domain=3.141592653589793:6.283185307179586,variable=\t]({1.*1.5*cos(\t r)+0.*1.5*sin(\t r)},{0.*1.5*cos(\t r)+1.*1.5*sin(\t r)});
\draw [shift={(1.5,2.)},line width=0.4pt]  plot[domain=0.:3.141592653589793,variable=\t]({1.*1.5*cos(\t r)+0.*1.5*sin(\t r)},{0.*1.5*cos(\t r)+1.*1.5*sin(\t r)});
\draw [shift={(0.,2.)},line width=0.4pt]  plot[domain=0.:3.141592653589793,variable=\t]({1.*1.5*cos(\t r)+0.*1.5*sin(\t r)},{0.*1.5*cos(\t r)+1.*1.5*sin(\t r)});
\draw [->,line width=0.4pt] (1.5,2.) -- (1.5,-2.);
\draw [->,line width=0.4pt] (3.,2.) -- (3.,-2.);
\end{tikzpicture}
\end{center}
which is also a graphical representation of $\mathrm{Tr}(G)*\mathrm{Tr}(H)$. So $\mathrm{Tr}(G*H)=\mathrm{Tr}(G)*Tr(H)$. 
\end{proof}

\subsection{Quasi-TraPs} \label{subsec:partial}

The partial trace maps $t_{i,j}$ arising in the definition of a TraP  might not be defined on every operator. To circumvent this difficulty, 
we work with a $\sym\times \sym^{op}$-module $(P(k,l))_{k,l\geqslant 0}$ with a horizontal concatenation $\star$,
satisfying all the required axioms, and for any $k,l\geqslant 1$, for any $i\in [k]$, $j\in [l]$,
a map $T_{i,j}:P'(k,l) \longrightarrow P(k-1,l-1)$ defined on a submodule of $P(k,l)$; we assume that it satisfies all the required axioms
as soon as all the maps they imply are defined. 

We can then embed such a quasi-TraP 
in a "complete" TraP: consider the TraP $\Gacirc(P)$,
and quotient it by the TraP ideal generated by the elements:
\begin{enumerate}
\item $\nu_P(p)*\nu_P(q)-\nu_P(p\star q)$, where $p,q \in P$.
\item $t_{i,j}\circ \nu_P(p)-\nu_P\circ T_{i,j}(p)$, where $p\in P$ such that $T_{i,j}(p)$ is defined.
\end{enumerate}
We obtain in this way a TraP $\overline{P}$, with partial trace maps $t_{i,j}$ induced on the quotient by the partial trace maps of $\Gacirc(P)$.
It contains a $\sym\times \sym$-module isomorphic to $P$ and formed by graphs with only one vertex,
which we identify with $P$ itself. Then,  if $T_{i,j}(p)$ is defined,  $T_{i,j}(p)=t_{i,j}(p)$.

\begin{example}
Let $V=\K[X]$, $(X^n)_{n\geqslant 0}$ its canonical basis and $(\delta_n)_{n\geqslant 0}$ the dual basis.
Let us denote by $E^+$ the subspace of $\Hom(V)$ generated by the endomorphisms  of the form
\[f_{i,j}:\left\{\begin{array}{rcl}
\K[X]&\longrightarrow&\K[X]\\
X^k&\longrightarrow&\delta_{i,k}X^j,
\end{array}\right.\]
where $i,j\geqslant 0$ (i.e. $f_{i,j}(X^k)=X^j$ if $k=i$, and $f_{i,j}(X^k)=0$ otherwise). This is the subspace of endomorphisms of $V$ with a 
finite support when applied on monomials.
Note that $E^+$ does not contains $\mathrm{Id}_V$: we put $E=E^+\oplus \K \mathrm{Id}_V$. For any $k,l\geqslant 0$,
let $P(k,l)$ be the submodule of $\Hom (V^{\otimes k},V^{\otimes k})$ generated by $E^{\otimes k}$ if $k=l$,
and $\{0\}$ otherwise. This is stable under the horizontal concatenation of $\Hom_V$.  

The elements of $P(k,k)$ are linear spans of  terms:
\[\sigma \cdot (f_1\otimes \ldots \otimes f_k) \cdot \tau,\]
where $\sigma,\tau \in \sym_k$, and for any $p$, $f_p$ is one of the $f_{i,j}$ or is $\mathrm{Id}_V$. 
We define a partial trace map on $P$ by putting $T_{1,1}(f_{i,j})=\delta_{i,j}$; but $T_{1,1}(\mathrm{Id}_V)$ is not defined.
This is extended to $P$ using the axioms of a TraP. For example:
\begin{align*}
T_{1,1}(f_{i,j}\otimes f_{k,l})&=\delta_{i,j} f_{k,l},&T_{1,1}(f_{i,j}\otimes \mathrm{Id}_V)&=\delta_{i,j}\mathrm{Id}_V,\\
T_{2,2}(f_{i,j}\otimes f_{k,l})&=\delta_{k,l} f_{i,j},&T_{2,2}(f_{i,j}\otimes \mathrm{Id}_V)&\mbox{ is not defined},\\
T_{1,2}(f_{i,j}\otimes f_{k,l})&=\delta_{i,l} f_{k,j},&T_{1,2}(f_{i,j}\otimes \mathrm{Id}_V)&=f_{i,j},\\
T_{2,1}(f_{i,j}\otimes f_{k,l})&=\delta_{j,k} f_{i,l},&T_{2,1}(f_{i,j}\otimes \mathrm{Id}_V)&=f_{i,j}.
\end{align*}
Denoting by $\grapheo$ the graph with only one loop,
we obtain that for any $k\geqslant 0$,
\[\overline{P}(k,k)=\K[\grapheo]\otimes P(k,k),\] 
and $t_{1,1}(\mathrm{Id}_V)=\grapheo$.  Any $p\in P(k,k)$ is identified with $1\otimes p\in \overline{P}(k,k)$. For example, in $\overline{P}$:
\begin{align*}
t_{1,1}(f_{i,j}\otimes f_{k,l})&=\delta_{i,j} f_{k,l},&t_{1,1}(f_{i,j}\otimes \mathrm{Id}_V)&=\delta_{i,j}\mathrm{Id}_V,\\
t_{2,2}(f_{i,j}\otimes f_{k,l})&=\delta_{k,l} f_{i,j},&t_{2,2}(f_{i,j}\otimes \mathrm{Id}_V)&=f_{i,j}\otimes\grapheo,\\
t_{1,2}(f_{i,j}\otimes f_{k,l})&=\delta_{i,l} f_{k,j},&t_{1,2}(f_{i,j}\otimes \mathrm{Id}_V)&=f_{i,j},\\
t_{2,1}(f_{i,j}\otimes f_{k,l})&=\delta_{j,k} f_{i,l},&t_{2,1}(f_{i,j}\otimes \mathrm{Id}_V)&=f_{i,j}.
\end{align*}
Choosing for any $k\geqslant 1$ an element $f_k\in P(k,l)$,
any graph $G$ such that $L(G)=\emptyset$ is sent to an element of $P$ by $\Phi$.
\end{example}

\section{The TraP $\overline{\mathcal{K}}_{{X}}^\infty$ of smoothing pseudo-differential operators}

We apply our results on TraPs to  tensor products of a class of  of Fr\'echet nuclear spaces introduced in Section \ref{sec:propHom}, 
 namely Fr\'echet spaces ${\mathcal E}(X)$  of smooth sections of $X$. Recall from Proposition 
\ref{prop:prod_function}   that such spaces are stable under tensor products and morphisms in  ${\rm Hom}^c( {\mathcal E}^\prime(X), {\mathcal E}(Y))$ are determined by    smoothing kernels in ${\mathcal E}(X\times Y)$.

\subsection{Trace of smoothing pseudo-differential operators}

Let $X$   be a smooth finite dimensional closed manifold.
Let us set $E=\mathcal{E}({X})$, and 
 $F=\mathcal{E}'({X})$, which is not Fr\'echet, in which case Lemma \ref{lem:compo_dual_pairing} does not apply. 
 
 Instead, we  restrict ourselves to   smooth kernels
 which stabilise ${\mathcal E}({X})$. 
 We  set, for $(k,l)\neq(1,1)$:
 \[{\mathcal K}_{{X}}^\infty(k, l):={\mathcal E}({{X}}^k\times {{X}}^l)\simeq{\mathcal E}({{X}})^{\widehat\otimes  k}\, \hat \otimes \,  {\mathcal E}({{X}})^{\widehat\otimes  l},\]
 where the identification holds by Proposition \ref{prop:prod_function}. For $(k,l)=(1,1)$ we set 
 \begin{equation*}
  {\mathcal K}_{{X}}^\infty(1, 1):={\mathcal E}(X\times X)\bigcup \{\delta\}\simeq{\mathcal E}(X)\, \hat \otimes \,  {\mathcal E}(X)\bigcup\{\delta\}
 \end{equation*}
 with $\delta$ the (singular) kernel of the identity operator on $\mathcal{E}(X)$. With the notations of Definition \ref{defi:Trap}, we will 
 have $I=\delta$.

For a closed  Riemannian manifold ${{X}}$ equipped with a volume measure $\mu$, the 
   canonical embedding 
   ${\mathcal E}({{X}})\hookrightarrow {\mathcal E}^\prime({{X}})$, $f\longmapsto \left(\varphi\mapsto \int_{{X}}f(x)\, \varphi(x)\, d\mu(x)\right)$ induces an 
   embedding 
\[{\mathcal K}_{{X}}^\infty(k, l)\hookrightarrow  {\mathcal K}_{{X}} (k, l)\simeq{\mathcal E}'({{X}})^{\widehat\otimes  k}\, \hat \otimes \,  {\mathcal E}({{X}})^{\widehat\otimes  l}.\]
\begin{prop}
The family of topological vector spaces $\left({\mathcal K}_{{X}}^\infty(k, l)\right)_{k,l\geq 0}$ equipped with the partial traces
\begin{align*}
t_{i,j}&:\left\{\begin{array}{rcl}
{\mathcal K}_{{X}}^\infty(k, l)&\longrightarrow&{\mathcal K}_{{X}}^\infty(k-1,l-1)\\
K_1\otimes K_2 &\longmapsto&  {t_{i,j}(K_1\otimes K_2)}
\end{array}\right.
\end{align*}
 with, for $K_1\otimes K_2\neq\delta$, $t_{i,j}(K_1\otimes K_2)$ defined by
\begin{align*}
 t_{i,j}(K_1\otimes K_2) & (x_1, \cdots, x_{k-1}, y_1, \cdots,  y_{l-1}):= \\
 & \int_{X} K_1(x_1 ,  \cdots,x_{i-1},z,x_{i}\cdots, x_{k-1})\, K_2(y_1,\cdots,y_{j-1},z,y_{j}\cdots, y_k) \,d\mu(z)
\end{align*}
(with an obvious abuse of notation in the cases where $i$ (or $j$) is equal to 1 or $k$ (or to 1 or $l$))

defines a TraP, written $\mathcal{K}_{{X}}^\infty$. 
\end{prop}  
\begin{remark}\label{rk:smoothkernelquasitrop}
 Technically, $\mathcal{K}_{{X}}^\infty$ is a quasi-TraP in the sense of Subsection \ref{subsec:partial} since $t_{1,1}(I)=t_{1,1}(\delta)$ is not 
 defined. Following Subsection \ref{subsec:partial}, this quasi-TraP can be completed to a full TraP $\overline{\mathcal{K}}_{{X}}^\infty$.
\end{remark}
\begin{proof}
The unit $I_0\in\mathcal{K}_{{X}}^\infty(0,0)\simeq \C\otimes\C$ of the vertical concatenation $*=\otimes$ is the constant map defined by 
$f(x)=1$. It is the unit of $\otimes$ by bilinearity of the tensor product.

The unit $I\in\mathcal{K}_{{X}}^\infty(0,0)$ is $\delta$ by definition of the action of Dirac's distribution on smooth kernels.

 It suffices to show that $t_{i,j}(K_1\otimes K_2) (x_1, \cdots, x_{k-1}, y_1, \cdots,  y_{l-1})$ lies in  
 ${\mathcal K}_X^\infty(k-1,l-1)$.
 The axioms of the 
 TraP will then hold since they are in ${\mathcal K}_{{X}} (k, l)$ (Example \ref{ex:KM}).
 
 The existence of the integral comes from the smoothness of $K_1$ and $K_2$ and the 
 closedness of $X$. It is enough to show that the function 
 $t_{i,j}(K_1\otimes K_2):X^{k-1}\times X^{l-1}\longrightarrow \C$ is smooth. Since
 $K_1$ and $K_2$ are smooth, the map 
 \begin{equation*}
  (x_1 ,\cdots, x_{k-1},y_1,\cdots, y_k) \mapsto K_1(x_1 ,  \cdots,x_{i-1},z,x_{i}\cdots, x_{k-1})\, K_2(y_1,\cdots,y_{j-1},z,y_{j}\cdots, y_k)
 \end{equation*}
 is infinitely differentiable for any $z\in X$. Since $X$ is compact, the partial derivatives
 \begin{equation*}
  \partial_{\vec x}^{\vec\alpha}\partial_{\vec y}^{\vec \beta}K_1(x_1 ,  \cdots,x_{i-1},z,x_{i}\cdots, x_{k-1})\, K_2(y_1,\cdots,y_{j-1},z,y_{j}\cdots, y_k)
 \end{equation*}
 are bounded uniformly in $z$. We can therefore use the dominated convergence theorem to get that
 \begin{align*}
 & \int_{X} \partial_{\vec x}^{\vec\alpha}\partial_{\vec y}^{\vec \beta}K_1(x_1 ,  \cdots,x_{i-1},z,x_{i}\cdots, x_{k-1})\, K_2(y_1,\cdots,y_{j-1},z,y_{j}\cdots, y_k) \,d\mu(z) \\
  = & \partial_{\vec x}^{\vec\alpha}\partial_{\vec y}^{\vec \beta}\int_{X}K_1(x_1 ,  \cdots,x_{i-1},z,x_{i}\cdots, x_{k-1})\, K_2(y_1,\cdots,y_{j-1},z,y_{j}\cdots, y_k) \,d\mu(z) \\
  = & \partial_{\vec x}^{\vec\alpha}\partial_{\vec y}^{\vec \beta} t_{i,j}(K_1\otimes K_2) (x_1, \cdots, x_{k-1}, y_1, \cdots,  y_{l-1}).
\end{align*}
Therefore the map $t_{i,j}(K_1\otimes K_2) (x_1, \cdots, x_{k-1}, y_1, \cdots,  y_{l-1})$ is smooth.
\end{proof}
In view of the fact that the trace of a smoothing pseudodifferential operator $P$ with kernel $K$ is 
\[\mathrm{Tr}(P)= \int_X K(x,x) \, d\mu(x)\]
Corollary \ref{coro:generalised_traces} yields a generalised trace 
\begin{equation*}
 \mathrm{Tr}:\bigsqcup_{k\in  \N_0} \mathcal{K}_{{X}}^\infty(k,k)\longrightarrow \C
\end{equation*}
on smoothing pseudo-differential operators on a closed smooth finite dimensional manifold. This trace is indeed cyclic for the horizontal and vertical 
composition products of $\overline{\mathcal{K}}_{{X}}^\infty$ in the sense of Corollary \ref{coro:generalised_traces}.

\subsection{Generalised convolution of smoothing operators}

Let ${{X}}$ be a smooth finite dimensional closed Riemannian manifold. Set $X_{k,l}:=\mathcal{K}_X^\infty(k,l)$. 
Recall from Proposition \ref{prop:PdecGrc} that there exists a TraP map $\Phi:\Gacirc(P)\longrightarrow X$,
as $X$ is a TraP.
\begin{defi}
 Let $G$ be a graph decorated by $X=(\mathcal{K}_X^\infty(k,l))_{kl,\in  \N_0}$. 
 The \textbf{generalised convolution} associated to $G$ is the 
 smoothing operator $\Phi(G)\in\overline{\mathcal{K}}_{{X}}^\infty$ given by the image of $G$ under 
 $\Phi$.
\end{defi}

The name generalised convolution is justified by the following remark.
\begin{remark}
 Let $G$ be a ladder graph decorated by $X=(\mathcal{K}_X^\infty(k,l))_{kl,\in  \N_0}$  i.e., a graph such that $I(G)=O(G)=[1]$, 
 $IO(G)=L(G)=\emptyset$,  $V(G)=\{v_1,\cdots,v_n\}$, $E(G)=\{e_1,\cdots,e_{n-1}\}$ and the source and target maps defined by
\begin{align*}
&& s_G(1)&=v_n,&t_G(1)&=v_1,\\
&\forall i\in [n-1],&s_G(e_i)&=v_i,&t_G(e_i)&=v_{i+1}.
 \end{align*}
Here is a graphical representation of this graph:
 
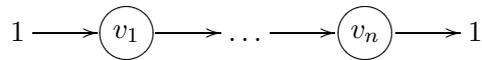
\[\xymatrix{1\ar[r]& \rond{v_1}\ar[r]&\ldots \ar[r]&\rond{v_n}\ar[r]&1}\]
  Let $O_i$ be the smoothing pseudo-differential operator defined by the kernel $K_i$ that decorates the vertex $v_i$: 
 $K_i:=\mathrm{dec}(v_i)$ for any 
 $v_i\in[n]$. Then the generalised convolution associated to the graph $G$ is the convolution of the kernels $K_i,\cdots,K_n$, which is the 
 kernel of the smoothing pseudo-differential 
 operator $O_1\circ\cdots\circ O_n$. 
\end{remark}
The previous remark  leads to the following statement.
\begin{cor}
 The convolution of smoothing pseudo-differential operators is well-defined and associative.
\end{cor}
\begin{proof}
 Well-definedness follows from the definition. The associativity follows from the fact that the vertical composition build from the 
 TraP structure of graphs is associative, 
 together with the fact that $\phi_{\mathrm{Id}}$ is a morphism of TraP. 
\end{proof}

\appendix

\section{Appendix: topologies on tensor products} \label{section:topologies_tens_prod}

Tensor products ot topological spaces can be equipped with various topologies. A first possibility is the so-called \textbf{$\epsilon$-topology}; \cite[Definition 43.1]{Treves67}. For two   topological vector spaces $E$ and $F$, one 
can show (\cite[Proposition 42.4]{Treves67}) the isomorphism  of vector spaces
$E\otimes F\simeq \mathcal{B}^c(E'_\sigma\times F'_\sigma,\K)$ where 
$\mathcal{B}^c(E'_\sigma\times F'_\sigma,\K)$ denotes the space of continuous bilinear maps from $E'_\sigma\times F'_\sigma$ to $\K$ and $E'_\sigma$ (resp. 
$F'_\sigma$)  the topological dual of $E$ (resp. $F$) for $\sigma$, the weak topology. 

Recall that a bilinear map $f:E\times F\longrightarrow K$ is called separately continuous if, 
for any pair $(x,y)\in E\times F$, the maps $z\to f(x,z)$ and $z\to f(z,y)$ are continuous. We then clearly have that continuous bilinear maps 
build a linear subspace of the space $\mathcal{B}^{sc}(E\times F,\K)$ of separately continuous bilinear maps.

The space  $\mathcal{B}^{sc}(E\times F,\K)$ can be equipped with the topology of uniform convergence on products of equicontinuous subsets of $E'_\sigma$ with 
equicontinuous subsets of $F'_\sigma$. Recall that, for a topological space $X$ and a topological vector space $G$, a set $S$ of maps from $X$ to $G$ is 
said to be equicontinuous at $x_0\in X$ if, for any $V\subseteq G$ neighbourhood of zero, there is some neighbourhood  $V(x_0)\subseteq X$   of $x_0$, such 
that 
\begin{equation*}
 \forall f\in S,~x\in V(x_0) \Rightarrow f(x)-f(x_0)\in V.
\end{equation*}
In our case, $G$ is $\K$ and $X$ is $E_\sigma$ (resp. $F_\sigma$). This topology induces a topology on the subspace 
$\mathcal{B}^c(E'_\sigma\times F'_\sigma,\K)$ and thus on $E\otimes F$. We denote by $E\otimes_\epsilon F$ the topological vector space obtained by 
endowing $E\otimes F$ with this topology.

There is another  topology on $E\otimes F$   called the \textbf{projective topology}; 
\cite[Definition 43.2]{Treves67}.  The projective topology is defined as the strongest 
locally convex topology on $E\otimes F$ such that the canonical map $\phi:E\times F\longrightarrow E\otimes F$ is continuous. 
We write $E\otimes_\pi F$ the topological vector space obtained by 
endowing $E\otimes F$ with this topology.

The neighbourhoods of zero of the projective topology can be simply described in terms of neighbourhoods of zero in $E$ and $V$. A convex subset $S$ of 
$E\otimes F$ containing zero is a neighbourhood of zero if it exist a neighbourhood $U$ (resp. V) of zero in $E$ (resp. $F$) such that 
$U\otimes V:=\{u\otimes v|u\in U\wedge v\in V\}\subseteq S$.

\section{Appendix: definition of the partial trace maps on $\Gr$}

We give a rigorous definition of the partial trace maps on the space of graphs $\Gr$, which were only loosely defined in the bulk of the article.

Let $G\in \Gr(k,l)$ with $k,l\geqslant 1$, $i\in [k]$ and $j\in [l]$. We put $e_i=\alpha_G^{-1}(i)$
and $f_j=\beta_G^{-1}(j)$. We define the graph $G'=t_{i,j}(G)$ in the following way:
\begin{enumerate}
\item If $e_i\in I(G)$ and $f_j\in O(G)$, then:
\begin{align*}
V(G')&=V(G),&E(G')&=E(G)\sqcup \{(e_i,f_j)\},\\
I(G')&=I(G)\setminus\{e_i\},& O(G')&=O(G)\setminus\{f_j\},\\
IO(G')&=IO(G),&L(G')&=L(G),\\
s_{G'}(e)&=\begin{cases}
s_G(f_j)\mbox{ if }e=(e_i,f_j),\\
s_G(e)\mbox{ otherwise},
\end{cases}&
t_{G'}(e)&=\begin{cases}
t_G(e_i)\mbox{ if }e=(e_i,f_j),\\
t_G(e)\mbox{ otherwise},
\end{cases}\\
\alpha_{G'}(e)&=\begin{cases}
\alpha_G(e)\mbox{ if }\alpha_G(e)<i,\\
\alpha_G(e)-1\mbox{  if }\alpha_G(e)\geqslant i,
\end{cases}&
\beta_{G'}(e)&=\begin{cases}
\beta_G(e)\mbox{ if }\beta_G(e)<j,\\
\beta_G(e)-1\mbox{ if }\beta_G(e)\geqslant j.
\end{cases}
\end{align*}

\item If $e_i\in IO(G)$ and $f_j\in O(G)$, then:
\begin{align*}
V(G')&=V(G),&E(G')&=E(G),\\
I(G')&=I(G),& O(G')&=O(G)\setminus\{f_j\}\sqcup \{(e_i,f_j,)\},\\
IO(G')&=IO(G)\setminus\{e_i\},&L(G')&=L(G),\\
s_{G'}(e)&=\begin{cases}
s_G(f_j)\mbox{ if }e=(e_i,f_j),\\
s_G(e)\mbox{ otherwise},
\end{cases}&
t_{G'}(e)&=t_G(e),\\
\alpha_{G'}(e)&=\begin{cases}
\alpha_G(e)\mbox{ if }\alpha_G(e)<i,\\
\alpha_G(e)-1\mbox{  if }\alpha_G(e)\geqslant i,
\end{cases}&
\beta_{G'}(e)&=\begin{cases}
\beta_G(e_i)\mbox{ if }e=(e_i,f_j)\mbox{ and }\beta_G(e_i)<j,\\
\beta_G(e_i)-1\mbox{ if }e=(e_i,f_j)\mbox{ and }\beta_G(e_i)\geqslant j,\\
\beta_G(e)\mbox{ if }e\neq(e_i,f_j)\mbox{ and }\beta_G(e)<j,\\
\beta_G(e)-1\mbox{ if }e\neq(e_i,f_j)\mbox{ and }\beta_G(e)\geqslant j.
\end{cases}
\end{align*}

\item If $e_i\in I(G)$ and $f_j\in IO(G)$, then:
\begin{align*}
V(G')&=V(G),&E(G')&=E(G),\\
I(G')&=I(G)\setminus\{e_i\}\sqcup \{(e_i,f_j)\},& O(G')&=O(G),\\
IO(G')&=IO(G)\setminus\{f_j\},&L(G')&=L(G),\\
s_{G'}(e)&=s_G(e),&
t_{G'}(e)&=\begin{cases}
t_G(e_i)\mbox{ if }e=(e_i,f_j),\\
t_G(e)\mbox{ otherwise},
\end{cases}\\
\alpha_{G'}(e)&=\begin{cases}
\alpha_G(f_i)\mbox{ if }e=(e_i,f_j)\mbox{ and }\alpha_G(f_j)<i,\\
\alpha_G(f_i)-1\mbox{ if }e=(e_i,f_j)\mbox{ and }\alpha_G(f_j)\geqslant i,\\
\alpha_G(e)\mbox{ if }e\neq(e_i,f_j)\mbox{ and }\alpha_G(e)<i,\\
\alpha_G(e)-1\mbox{ if }e\neq(e_i,f_j)\mbox{ and }\alpha_G(e)\geqslant i,
\end{cases}&
\beta_{G'}(e)&=\begin{cases}
\beta_G(e)\mbox{ if }\beta_G(e)<j,\\
\beta_G(e)-1\mbox{ if }\beta_G(e)\geqslant j.
\end{cases}
\end{align*}

\item If $e_i\in IO(G)$, $f_j\in IO(G)$ and $e_i\neq f_j$, then:
\begin{align*}
V(G')&=V(G),&E(G')&=E(G),\\
I(G')&=I(G),& O(G')&=O(G),\\
IO(G')&=\{(e_i,f_j)\}\sqcup IO(G)\setminus\{e_i,f_j\},&L(G')&=L(G),\\
s_{G'}(e)&=s_G(e),&
t_{G'}(e)&=t_G(e),\\
\alpha_{G'}(e)&=\begin{cases}
\alpha_G(f_i)\mbox{ if }e=(e_i,f_j)\mbox{ and }\alpha_G(f_j)<i,\\
\alpha_G(f_i)-1\mbox{ if }e=(e_i,f_j)\mbox{ and }\alpha_G(f_j)\geqslant i,\\
\alpha_G(e)\mbox{ if }e\neq(e_i,f_j)\mbox{ and }\alpha_G(e)<i,\\
\beta_G(e)-1\mbox{ if }e\neq(e_i,f_j)\mbox{ and }\alpha_G(e)\geqslant i,
\end{cases}\\
\beta_{G'}(e)&=\begin{cases}
\beta_G(e_i)\mbox{ if }e=(e_i,f_j)\mbox{ and }\beta_G(e_i)<j,\\
\beta_G(e_i)-1\mbox{ if }e=(e_i,f_j)\mbox{ and }\beta_G(e_i)\geqslant j,\\
\beta_G(e)\mbox{ if }e\neq(e_i,f_j)\mbox{ and }\beta_G(e)<j,\\
\beta_G(e)-1\mbox{ if }e\neq(e_i,f_j)\mbox{ and }\beta_G(e)\geqslant j.
\end{cases}
\end{align*}

\item If $e_i\in IO(G)$, $f_j\in IO(G)$ and $e_i=f_j$, then:
\begin{align*}
V(G')&=V(G),&E(G')&=E(G),\\
I(G')&=I(G),& O(G')&=O(G),\\
IO(G')&=IO(G)\setminus\{e_i,f_j\},&L(G')&=L(G)\sqcup \{(e_i,f_j)\},\\
s_{G'}(e)&=s_G(e),&
t_{G'}(e)&=t_G(e),\\
\alpha_{G'}(e)&=\begin{cases}
\alpha_G(e)\mbox{ if }\alpha_G(e)<i,\\
\alpha_G(e)-1\mbox{  if }\alpha_G(e)\geqslant i,
\end{cases}&
\beta_{G'}(e)&=\begin{cases}
\beta_G(e)\mbox{ if }\beta_G(e)<j,\\
\beta_G(e)-1\mbox{ if }\beta_G(e)\geqslant j.
\end{cases}
\end{align*}
\end{enumerate}

\section{Appendix: full proofs}

\subsection{Proof of Theorem \ref{thm:freeness_Gr}} \label{appendix:proof_freeness_Gr}

\begin{proof}
Let us define $\Phi(G)$ for any graph $G$ by induction on $n=|V(G)|$, such that for any permutation
$\sigma\in \sym_{i(G)}$, $\tau \in \sym_{o(G)}$,
\[\Phi(\sigma\cdot G\cdot \tau)=\sigma\cdot \Phi(G)\cdot \tau.\]
If $n=0$, there exists a unique permutation $\gamma\in \sym_k$ such that $G=\gamma\cdot I_k$. We put
\[\Phi(G)=\gamma\cdot I_k,\]
where we used the same notation $I_k$ for the units of $\textbf{Gr}$ and $P$.

If $\sigma,\tau\in \sym_k$:
\begin{align*}
\Phi(\sigma \cdot G\cdot \tau)&=\Phi((\sigma\gamma)\cdot I_k\cdot \tau)\\
&=\Phi((\sigma\gamma \tau)\cdot I_k)\\
&=(\sigma\gamma \tau)\cdot I_k\\
&=\sigma\cdot(\gamma\cdot I_k)\cdot \tau\\
&=\sigma \cdot \Phi(G)\cdot \tau.
\end{align*}
Let us assume that $\Phi(G')$ is defined for any graph $G'$ such that $|V(G')|<n$.
Let \[G=\gamma \cdot (G_1*\ldots *G_k*I_p)\circ G_0*\grapheo^{*\ell}\]
be a  minimal decomposition of $G$.  {If $G$ is indecomposable, we set $\Phi(G)=\phi(G)$. Otherwise,} as $V(G_1)\neq \emptyset$, $|V(G_0)|<n$. We put:
\[\Phi(G)=\gamma \cdot  (\phi(G_1)*\ldots *\phi(G_k)*I_p)\circ\Phi(G_0)*\phi(\grapheo)^{*\ell}.\]
Let us first prove that this does not depend on the choice of the minimal decomposition of $G$.
Starting from a minimal decomposition of $G$, one obtains all possible minimal decompositions
of $G$ by a finite sequence of operations of type A and B:
\begin{itemize}
\item Type A: changing the indexations of the input and output edges of the graphs $G_i$.
We obtain a minimal decomposition $G=\gamma' \cdot (G'_1*\ldots*G'_k*I_p)\circ G'_0*\grapheo^{*\ell}$,
such that there exists permutations $\alpha_i$, $\beta_i$, with:
\begin{align*}
G'_i&=\alpha_i\cdot G_i\cdot \beta_i,\\
G'_0&=(\beta_1^{-1}\otimes \ldots \otimes \beta_k^{-1}\otimes \mathrm{Id}_p)\cdot G_0,\\
\alpha'&=\alpha(\alpha_1^{-1}\otimes \ldots \otimes \alpha_k^{-1}\otimes \mathrm{Id}_p).
\end{align*}
\item Type B: permuting $G_l$ and $G_{l+1}$  for  $l\in [k-1]$. 
We obtain another minimal decomposition $G=\gamma' \cdot (G'_1*\ldots*G'_k*I_p)\circ G'_0*\grapheo^{*\ell}$, with:
\begin{align*}
G'_i&=\begin{cases}
G_{l+1}\mbox{ if }i=l,\\
G_l\mbox{ if }j=l+1,\\
G_i\mbox{ otherwise};
\end{cases}\\
G'_0&= (\mathrm{Id}_{i(G_1)+\ldots+i(G_{l-1})}\otimes c_{i(G_{l+1}),i(G_l)}\otimes \mathrm{Id}_{i(G_{l+2})+\ldots+i(G_k)+p})
\circ G_0,\\
\gamma'&=\gamma(\mathrm{Id}_{o(G_1)+\ldots+o(G_{l-1})}\otimes c_{o(G_l),o(G{l+1})}\otimes
 \mathrm{Id}_{o(G_{l+2})+\ldots+o(G_k)+p}).
\end{align*}
\end{itemize}
Let $G=\gamma \cdot (G'_1*\ldots*G'_k*I_p)\circ G'_0*\grapheo^{*\ell'}$ be another minimal decomposition of $G$.
Then $\ell=\ell'$ is the number of loops of $G$. It is enough to prove that
\[\gamma \cdot  (\phi(G_1)*\ldots *\phi(G_k)*I_p)\circ\Phi(G_0)
=\gamma' \cdot  (\phi(G_1')*\ldots *\phi(G_k')*I_p)\circ\Phi(G_0').\]
We can assume that $G'$ is obtained from $G$ by a single operation of type A or of type B.
If it is of type A:
\begin{align*}
&\gamma' \cdot (\phi(G'_1)*\ldots*\phi(G'_k)*I_p)\circ \Phi(G'_0)\\
&=\gamma\cdot(\alpha_1^{-1}\otimes \ldots \otimes \alpha_k^{-1}\otimes \mathrm{Id}_p)\cdot
(\phi(\alpha_1\cdot G_1\cdot \beta_1)\otimes \ldots \otimes \phi(\alpha_k\cdot G_k\cdot \beta_k)*I_p)\\
&\circ \Phi((\beta_1^{-1}\otimes \ldots \otimes \beta_k^{-1}\otimes \mathrm{Id}_p)\cdot G_0)\\
&=\gamma\cdot(\alpha_1^{-1}\otimes \ldots \otimes \alpha_k^{-1}\otimes \mathrm{Id}_p)\cdot
(\alpha_1\cdot \phi(G_1)\cdot \beta_1\otimes \ldots \otimes\alpha_k\cdot  \phi(G_k)\cdot \beta_k*I_p)\\
&\circ((\beta_1^{-1}\otimes \ldots \otimes \beta_k^{-1}\otimes \mathrm{Id}_p)\cdot \Phi( G_0))\\
&=\gamma(\alpha_1^{-1}\otimes \ldots \otimes \alpha_k^{-1}\otimes \mathrm{Id}_p)
(\alpha_1\otimes \ldots \otimes \alpha_k\otimes \mathrm{Id}_p)\cdot
(\phi(G_1)\otimes \ldots \otimes\phi(G_k)*I_p)\\
&\circ  (\beta_1\otimes \ldots \otimes \beta_k\otimes \mathrm{Id}_p)(\beta_1^{-1}\otimes \ldots \otimes \beta_k^{-1}\otimes \mathrm{Id}_p)\cdot \Phi(G_0)\\
&=\gamma\cdot(\phi(G_1)\otimes \ldots \otimes\phi(G_k)*I_p)\circ \Phi(G_0).
\end{align*}
If it is of type B:
\begin{align*}
&\gamma' \cdot (\phi(G'_1)*\ldots*\phi(G'_k)*I_p)\circ \Phi(G'_0)\\
&=\gamma (\mathrm{Id}_{o(G_1)+\ldots+o(G_{l-1})}\otimes c_{o(G_l),o(G_{l+1})}
\otimes \mathrm{Id}_{o(G_{l+2})+\ldots+o(G_k)+p})\\
& \cdot(\phi(G_1)*\ldots*\phi(G_{l+1})*\phi(G_l)*\ldots*\phi(G_k)*I_p)\\
&\circ\Phi((\mathrm{Id}_{i(G_1)+\ldots+i(G_{l-1})}\otimes c_{i(G_{l+1}),i(G_l)}\otimes \mathrm{Id}_{i(G_{l+2})+\ldots+i(G_k)+p})
\cdot G_0)\\
&\gamma' \cdot (\phi(G'_1)*\ldots*\phi(G'_k)*I_p)\circ \Phi(G'_0)\\
&=\gamma (\mathrm{Id}_{o(G_1)+\ldots+o(G_{l-1})}\otimes c_{o(G_l),o(G_{l+1})}
\otimes \mathrm{Id}_{o(G_{l+2})+\ldots+o(G_k)+p})\\
& \cdot(\phi(G_1)*\ldots*\phi(G_{l+1})*\phi(G_l)*\ldots*\phi(G_k)*I_p)\\
&\circ(\mathrm{Id}_{i(G_1)+\ldots+i(G_{l-1})}\otimes c_{i(G_{l+1}),i(G_l)}\otimes \mathrm{Id}_{i(G_{l+2})+\ldots+i(G_k)+p})
\cdot\Phi(G_0)\\
&=\gamma \cdot (\phi(G_1)*\ldots*c_{o(G_l),o(G_{l+1})}\cdot(\phi(G_{l+1})*\phi(G_l))\cdot c_{i(G_{l+1}),i(G_l)}
*\ldots*\phi(G_k)*I_p)\circ\Phi(G_0)\\
&=\gamma \cdot (\phi(G_1)*\ldots*\phi(G_l)*\phi(G_{l+1}))*\ldots*\phi(G_k)*I_p)\circ\Phi(G_0)\\
&=\gamma\cdot(\phi(G_1)\otimes \ldots \otimes\phi(G_k)*I_p)\circ \Phi(G_0).
\end{align*}
So $\Phi(G)$ is well-defined. Let $\sigma \in \sym_{o(G)}$ and $\tau\in \sym_{i(G)}$. We put $H=\sigma\cdot G\cdot \tau$.
A minimal decomposition of $H$ is given by:
\begin{align*}
H_0&=G_0\cdot \tau,&H_i&=G_i\mbox{ if }i\in [k],&\gamma'&=\sigma\gamma.
\end{align*}
Hence:
\begin{align*}
\Phi(H)&=\sigma\gamma\cdot (\phi(G_1)*\ldots*\phi(G_k)*I_p)\circ \Phi(G_0\cdot \tau)*\phi(\grapheo)^{*\ell}\\
&=\sigma\gamma\cdot (\phi(G_1)*\ldots*\phi(G_k)*I_p)\circ \Phi(G_0)\cdot \tau*\phi(\grapheo)^{*\ell}\\
&=\sigma\cdot(\gamma\cdot (\phi(G_1)*\ldots*\phi(G_k)*I_p)\circ \Phi(G_0))\cdot \tau*\phi(\grapheo)^{*\ell}\\
&=\sigma\cdot \Phi(G)\cdot \tau.
\end{align*}
Consequently, we have defined a map $\Phi:\Gr\longrightarrow P$, extending the morphism  $\phi$ of $\sym\times \sym^{op}$-modules.
Let us prove that it is compatible with both concatenations.

Let $G$ and $G'$ be two graphs, both  with no loop. Let us prove that $\Phi(G*G')=\Phi(G)*\Phi(G')$ by induction on $n'=|V(G')|$.
 If $n'=0$, there exists $\tau' \in \sym_q$ and $\ell'\in  \N_0$, such that $G'=\sigma'\cdot I_q$. 
We proceed by induction on $n=|V(G)|$. If $n=0$, there exists $\tau\in \sym_p$, such that 
$G=\sigma\cdot I_p$. Then $G*G'=(\sigma\otimes \sigma')\cdot I_{p+q}$, and:
\begin{align*}
\Phi(G*G')&=(\sigma\otimes \sigma')\cdot I_{p+q}\\
&=(\sigma\otimes \sigma')\cdot (I_p*I_q)\\
&=(\sigma \cdot I_p)*(\sigma'\cdot I_q)\\
&=\Phi(G)*\Phi(G').
\end{align*} 
Otherwise, let $G=\gamma\cdot (G_1*\ldots*G_k*_p)\circ G_0$ be a minimal decomposition of $G$.
A minimal decomposition of $G*G'$ is:
\[G*G'=(\alpha\otimes \sigma')\cdot (G_1*\ldots* G_k*I_{p+q}) \circ (G_0*I_q),\]
so, using the induction hypothesis on $G_0$:
\begin{align*}
\Phi(G*G')&=(\gamma\otimes \sigma')\cdot (\phi(G_1)*\ldots*\phi(G_k)*I_{p+q})
\circ \Phi(G_0*I_q)\\
&=(\gamma\otimes \sigma')\cdot\phi(G_1)*\ldots*\phi(G_k)*I_p*I_q)
\circ (\Phi(G_0)*I_q)\\
&=(\gamma\cdot (\phi(G_1)*\ldots*\phi(G_k)*I_p)\circ \Phi(G_0))*(\sigma'\cdot I_q)\\
&=\Phi(G)*\Phi(G').
\end{align*}
So the result holds at rank $n'=0$.

Let us assume the results hold at any rank  $<n'$. Let us consider minimal decompositions of $G$ and $G'$:
\begin{align*}
G&=\gamma \cdot (G_1*\ldots*G_k*I_p)\circ G_0,&G'&=\gamma' \cdot  (G'_1*\ldots*G'_l*I_q)\circ G'_0,
\end{align*}
with the convention $k=0$ if $V(G)=\emptyset$. We obtain a minimal decomposition of $G*G'$:
\begin{align*}
G*G'&=(\gamma\otimes \gamma')(\mathrm{Id}_{o(G_1)+\ldots+o(G_k)}\otimes c_{p,o(G'_1)+\ldots+o(G'_l)}\otimes \mathrm{Id}_q)\\
&\cdot (G_1*\ldots*G_k*G'_1*\ldots*G'_l*I_{p+q})\\
&\circ ((\mathrm{Id}_{i(G_1)+\ldots+i(G_k)}\otimes c_{i(G'_1)+\ldots+i(G'_l),p}\otimes \mathrm{Id}_q)) \cdot (G_0*G'_0)).
\end{align*}
We apply the induction assumption $\Phi(G*G')=\Phi(G)*\Phi(G')$ for  $|V(G')|<n'$  to  $G'_0$ whose number of vertices is smaller than that of 
$G'$ and hence smaller than $n'$.
\begin{align*}
\Phi(G*G')
&=(\gamma\otimes \gamma')(\mathrm{Id}_{o(G_1)+\ldots+o(G_k)}\otimes c_{p,o(G'_1)+\ldots+o(G'_l)}\otimes \mathrm{Id}_q)\\
&\cdot (\phi(G_1)*\ldots*\phi(G_k)*\phi(G'_1)*\ldots*\phi(G'_l)*I_{p+q})\\
&\circ \Phi((\mathrm{Id}_{i(G_1)+\ldots+i(G_k)}\otimes c_{i(G'_1)+\ldots+i(G'_l),p}\otimes \mathrm{Id}_q)) \cdot (G_0*G'_0))\\
&=(\gamma\otimes \gamma')(\mathrm{Id}_{o(G_1)+\ldots+o(G_k)}\otimes c_{p,o(G'_1)+\ldots+o(G'_l)}\otimes \mathrm{Id}_q)\\
&\cdot (\phi(G_1)*\ldots*\phi(G_k)*\phi(G'_1)*\ldots*\phi(G'_l)*I_{p+q})\\
&\circ (\mathrm{Id}_{i(G_1)+\ldots+i(G_k)}\otimes c_{i(G'_1)+\ldots+i(G'_l),p}\otimes \mathrm{Id}_q)) \cdot\Phi(G_0*G'_0)\\
&=(\gamma\otimes \gamma')(\mathrm{Id}_{o(G_1)+\ldots+o(G_k)}\otimes c_{p,o(G'_1)+\ldots+o(G'_l)}\otimes \mathrm{Id}_q)\\
&\cdot (\phi(G_1)*\ldots*\phi(G_k)*\phi(G'_1)*\ldots*\phi(G'_l)*I_p*I_q)\\
&\circ (\mathrm{Id}_{i(G_1)+\ldots+i(G_k)}\otimes c_{i(G'_1)+\ldots+i(G'_l),p}\otimes \mathrm{Id}_q)) \cdot(\Phi(G_0)*\Phi(G'_0))\\
&=(\gamma\otimes \gamma')\cdot
(\phi(G_1)*\ldots*\phi(G_k)*I_p*\phi(G'_1)*\ldots*\phi(G'_l)*I_q)\circ(\Phi(G_0)*\Phi(G'_0))\\
&=(\gamma\cdot(\phi(G_1)*\ldots*\phi(G_k)*I_p)\circ \Phi(G_0))
*(\gamma' \cdot(\phi(G'_1)*\ldots*\phi(G'_l)*I_q)\circ(\Phi(G'_0))\\
&=\Phi(G)*\Phi(G').
\end{align*}
So if $G$ and $G'$ are both graphs with no loop, $\Phi(G*G')=\Phi(G)*\Phi(G')$. \\

Let $G$, $G'$ be two graphs, both with no loop. Let us prove that $\Phi(G'\circ G)=\Phi(G')\circ \Phi(G)$. 
We proceed by induction on $n=|V(G)|+|V(G')|$. If $V(G')=\emptyset$,
there exists a permutation $\sigma\in \sym_p$ such that $G'=\sigma\cdot I_k$. Then:
\begin{align*}
\Phi(G'\circ G)&=\Phi(\sigma\cdot G)=\sigma\cdot \Phi(G)=\sigma \cdot (I_p\circ \Phi(G))=(\sigma\cdot I_p)\circ \Phi(G)
=\Phi(G')\circ \Phi(G).
\end{align*}
Similarly, if $V(G)=\emptyset$, $\Phi(G'\circ G)=\Phi(G')\circ \Phi(G)$. Thus we have proved the cases $n=0$ and $1$.

Let us assume it holds up to rank $N$ and take $G$ and $G'$ such that $n=N+1$. By the previous argument, 
$\Phi(G\circ G')=\Phi(G)\circ \Phi(G')$ if $V(G)=\emptyset$ or $V(G')=\emptyset$.
We now assume that $V(G)$ and $V(G')$ are nonempty. 
Let us consider minimal decompositions of $G$ and $G'$:
\begin{align*}
G&=\gamma \cdot(G_1*\ldots*G_k*I_p)\circ G_0&
G'&=\gamma' \cdot  (G'_1*\ldots*G'_l*I_q)\circ G'_0.
\end{align*}
In $G'\circ G$, the output edges of $G$ are glued with an input or an input-output edge of $G'$.
In particular, for any $i$,  output edges of $G_i$ are glued with input edges or input-output edges of $G'$.
Up to a change of indexation we assume that there is some $r$ such that:
\begin{itemize}
\item For all $i\leq r$, at least one output edge of $G_i$ is glued with an input edge of $G'$.
\item If $i>r$, all output edges of $G_i$ are glued with input-output edges of $G'$. 
\end{itemize}
\textit{A particular sub-case}. We assume that the input-output edges of $G'$ glued with an output of one of the 
$G_i$ are the input edges of $G'$ with the greatest indices. 
Then $G'_0=G''_0*I_{s+o(G_{r+1})+\ldots+o(G_k)}$
for a certain $s$. Moreover, $\gamma$ can be written as $\gamma=\gamma_1\otimes \gamma_2$, such that a minimal
decomposition of $H=G'\circ G$ is given by:
\begin{align*}
H_0&=(\mathrm{Id}_{i(G'_1)+\ldots+i(G'_l)}\otimes c_{i(G_{r+1})+\ldots+i(G_k)+p,s})\cdot G'_0\\
&\circ (\gamma_1\cdot (G_1*\ldots*G_r *I_{i(G_{r+1})+\ldots+i(G_k)+p})\cdot G_0,\\
(H_1,\ldots,H_m)&=(G'_1,\ldots,G'_l,G_{r+1},\ldots,G_k),\\
\gamma''&=\gamma'(\mathrm{Id}_{o(G'_1)+\ldots+o(G'_l)}\otimes c_{s,o(G_{r+1})+\ldots+o(G_k)+p})
(\mathrm{Id}_{o(G'_1)+\ldots+o(G'_l)+s}\otimes \gamma_2).
\end{align*}
Applying the induction hypothesis on $G_0$ and $G'_0$:
\begin{align*}
\Phi(H)&=\gamma'(\mathrm{Id}_{o(G'_1)+\ldots+o(G'_l)}\otimes c_{s,o(G_{r+1})+\ldots+o(G_k)+p})
(\mathrm{Id}_{o(G'_1)+\ldots+o(G'_l)+s}\otimes \gamma_2)\\
&\cdot((\phi(G'_1)*\ldots *\phi(G'_l)*\phi(G_{r+1}*\ldots*\phi(G_k))\\
&\circ \Phi((\mathrm{Id}_{i(G'_1)+\ldots+i(G'_l)}\otimes c_{i(G_{r+1})+\ldots+i(G_k)+p,s})\cdot G'_0\\
&\circ (\gamma_1\cdot (\phi(G_1)*\ldots*\phi(G_r) *I_{i(G_{r+1})+\ldots+i(G_k)+p})\cdot G_0)\\
&=\gamma'(\mathrm{Id}_{o(G'_1)+\ldots+o(G'_l)}\otimes c_{s,o(G_{r+1})+\ldots+o(G_k)+p})
(\mathrm{Id}_{o(G'_1)+\ldots+o(G'_l)+s}\otimes \gamma_2)\\
&\cdot((\phi(G'_1)*\ldots *\phi(G'_l)*\phi(G_{r+1}*\ldots*\phi(G_k))\\
&\circ (\mathrm{Id}_{i(G'_1)+\ldots+i(G'_l)}\otimes c_{i(G_{r+1})+\ldots+i(G_k)+p,s})\cdot \Phi(G'_0)\\
&\circ (\gamma_1\cdot (\phi(G_1)*\ldots*\phi(G_r) *I_{i(G_{r+1})+\ldots+i(G_k)+p})\cdot \Phi(G_0))\\
&=(\gamma\cdot (\phi(G_1)*\ldots*\phi(G_k)*I_p)\circ\Phi(G_0))
\circ(\gamma' \cdot (\phi(G'_1)*\ldots*\phi(G'_l)*I_q)\circ\Phi(G'_0))\\
&=\Phi(G)\circ \Phi(G').
\end{align*}

\textit{General case}. There exists a permutation $\sigma$, such that if $H'=G'\cdot \sigma^{-1}$
and $H=\sigma\cdot G$, then the condition of the particular sub-case holds for $(H,H')$. Then:
\begin{align*}
\Phi(G'\circ G)&=\Phi((G'\cdot \sigma^{-1}\sigma)\circ G)\\
&=\Phi((G'\cdot \sigma^{-1})\circ (\sigma\cdot G))\\
&=\Phi(G'\cdot \sigma^{-1})\circ\Phi(\sigma\cdot G)\quad\text{since the subcase holds}\\
&=(\Phi(G')\cdot \sigma^{-1})\circ(\sigma\cdot\Phi(G))\\
&=(\Phi(G')\cdot \sigma^{-1})\sigma)\cdot\Phi(G)\\
&=\Phi(G')\cdot \Phi(G).
\end{align*}
Finally, if $G$ and $G'$ are both graphs with no loop, $\Phi(G\circ G')=\Phi(G)\circ \Phi(G')$.\\

Let us finish this proof by considering loops. First, if $H$ is a graph, there exist a (unique) graph with no loop
and a (unique) integer $\ell$, such that $H=G*\grapheo^{*\ell}$.  
Let \[G=\gamma \cdot (G_1*\ldots *G_k*I_p)\circ G_0\]
be a  minimal decomposition of $G$. Then a minimal decomposition of $H$ is:
 \[H=\gamma \cdot (G_1*\ldots *G_k*I_p)\circ G_0*\grapheo^{*\ell}\],
 so
 \[\Phi(H)=\gamma \cdot (\phi(G_1)*\ldots \phi(G_k)*I_p)\circ \Phi(G_0)*\phi(\grapheo)^{\ell}
 =\Phi(G)*\phi(\grapheo)^{\ell}.\]
Hence, if $H$ and $H'$ are two graphs, let us consider graphs $G$ and $G'$ with no loop and integers $\ell$ and $\ell'$,
such that
\begin{align*}
H&=G*\grapheo^{*\ell},&H'&=G'*\grapheo^{*\ell'}.
\end{align*} 
Then $H*H'=G*G'*\grapheo^{*(\ell+\ell')}$ and $G*G'$ is a graph with no loop. Hence, by commutativity of the horizontal
concatenation of the product of $P$:
\begin{align*}
\Phi(H*H')&=\Phi(G*G')*\phi(\grapheo)^{*(\ell+\ell')}\\
&=\Phi(G)*\Phi(G')*\phi(\grapheo)^{*\ell}*\phi(\grapheo)^{*\ell'}\\
&=\Phi(G)*\phi(\grapheo)^{*\ell}*\Phi(G')*\phi(\grapheo)^{*\ell'}\\
&=\Phi(H)*\Phi(H').
\end{align*}
 So $\Phi$ is compatible with the horizontal concatenation. 
 
 If moreover, $H\in \Gr(l,m)$ and $H'\in \Gr(k,l)$, then $H\circ H'=(G\circ G')*\grapheo^{*(\ell+\ell')}$,
 and $G\circ G'$ is a graph with no loop.
 By the compatibility of the two concatenations of $P$:
 \begin{align*}
 \Phi(H\circ H')&=\Phi(G\circ G')*\phi(\grapheo)^{*(\ell+\ell')}\\
 &=(\Phi(G) \circ \Phi(G'))*\phi(\grapheo)^{*\ell}*\phi(\grapheo)^{*\ell'}\\
 &=(\Phi(G)*\phi(\grapheo)^{*\ell}) \circ( \Phi(G')*\phi(\grapheo)^{*\ell'})\\
&=\Phi(H)\circ \Phi(H').
 \end{align*}
 So $\Phi$ is compatible with the vertical concatenation.
\end{proof}

\subsection{Proof of Theorem \ref{theoFGlibre}}

\begin{proof}
We first define $\Phi(G)$ for any graph such that, if $G\in \Gr(k,l)$, for any $(\sigma,\tau)\in \sym_l\times \sym_k$, 
$\Phi(\sigma\cdot G\cdot \tau)=\sigma\cdot\Phi(G)\cdot \tau$. We proceed by induction 
on the number $N$  of internal edges of $G$. If $N=0$, then $G$ can be written (non uniquely) as
\[G=\grapheo^{*p}*\sigma\cdot(I^*q*G_{k_1,l_1}*\ldots *G_{k_r,l_r})\cdot \tau,\]
where $p,q,r\in  \N_0$ are unique, $(k_i,k_i) \in  \N_0^2$ for any $i$, unique up to a permutation,
and $\sigma \in \sym_{q+k_1+\ldots+k_r}$, $\tau\in \sym_{q+l_1+\ldots+l_r}$. 
We then put:
\[\Phi(G)=t_{1,1}(I)^{*p}*\sigma\cdot(I^{*q}*x_{k_1,l_1}*\ldots*x_{k_r,l_r})\cdot \tau.\]
Let us prove that this does not depend of the choice of the writing of $G$. As this is up to a permutation of the vertices
and of the choice of $\sigma$ and $\tau$, we can go from one decomposition of $G$ to any other one in a finite steps among the following two 
cases:
\begin{enumerate}
\item We consider two writing of $G$ of the form
\begin{align*}
G&=\grapheo^{*p}*\sigma\cdot(I^{*q}*G_{k_1,l_1}*\ldots* G_{k_i,l_i}*G_{k_{i+1},l_{i+1}} *\ldots *G_{k_r,l_r})\cdot \tau,\\
G&=\grapheo^{*p}*\sigma'\cdot(I^{*q}*G_{k_1,l_1}*\ldots*G_{k_{i+1},l_{i+1}}* G_{k_i,l_i} *\ldots *G_{k_r,l_r})\cdot \tau',
\end{align*}
with
\begin{align*}
\sigma'&=\sigma(\mathrm{Id}_{q+l_1+\ldots+l_{i-1}}\otimes c_{l_i,l_{i+1}}\otimes \mathrm{Id}_{l_{i+2}+\ldots+l_r}),\\
\tau'&=(\mathrm{Id}_{q+k_1+\ldots+k_{i-1}}\otimes c_{k_{i+1},k_i}\otimes \mathrm{Id}_{k_{i+2}+\ldots+k_r})\tau.
\end{align*}
Then, by commutativity of $*$:
\begin{align*}
&\sigma'\cdot(I^{*q}*x_{k_1,l_1}*\ldots*x_{k_r,l_r})\cdot \tau'\\
&=\sigma\cdot (I^{*q}*x_{k_1,l_1}*\ldots * c_{l_i,l_{i+1}}
\cdot(x_{k_{i+1},l_{i+1}}*x_{k_i,l_i})\cdot c_{k_{i+1},k_i}*\ldots*x_{k_r,l_r})\cdot \tau\\
&=\sigma\cdot (I^{*q}*x_{k_1,l_1}*\ldots *x_{k_i,l_i}*x_{k_{i+1},l_{i+1}}*\ldots *x_{k_r,l_r})\cdot \tau.
\end{align*}
\item We consider two writings of $G$ of the form
\begin{align*}
G&=\grapheo^{*p}*\sigma\cdot(I^{*q}*G_{k_1,l_1}*\ldots*G_{k_r,l_r})\cdot \tau,\\
G&=\grapheo^{*p}*\sigma'\cdot(I^{*q}*G_{k_1,l_1}*\ldots*G_{k_r,l_r})\cdot \tau',
\end{align*}
with
\begin{align*}
\sigma'&=\sigma (\sigma_0\otimes \sigma_1\otimes \ldots \otimes \sigma_r),&
\tau'&=(\sigma_0^{-1}\otimes \tau_1\otimes \ldots \otimes \tau_r)\tau',
\end{align*}
with $\sigma_0\in \sym_q$, $\sigma_i\in \sym_{k_i}$ and $\tau_i\in \sym_{l_i}$ if $i\geqslant 1$.
Using the commutativity of $*$ and the invariance of the $x_{k,l}$:
\begin{align*}
&\sigma'\cdot (I^{*q}*x_{k_1,l_1}*\ldots*x_{k_r,l_r})\cdot \tau'\\
&=\sigma \cdot (\sigma_0\cdot I^{*q}\cdot \sigma_0^{-1}
*\sigma_1\cdot x_{k_1,l_1}\cdot \tau_1*\ldots*\sigma_r\cdot x_{k_r,l_r}\cdot \tau_r)\cdot \tau\\
&=\sigma\cdot (I^{*q}*x_{k_1,l_1}*\ldots*x_{k_r,l_r})\cdot \tau.
\end{align*}
\end{enumerate}
Hence, $\Phi(G)$ is well-defined. Moreover, of $\tau' \in \sym_k$, $\sigma'\in \sym_l$, choosing a writing of $G$ of the form
\[G=\grapheo^{*p}*\sigma\cdot(I^{*q}*G_{k_1,l_1}*\ldots* G_{k_i,l_i}*G_{k_{i+1},l_{i+1}} *\ldots *G_{k_r,l_r})\cdot \tau,\]
a writing of $G'=\sigma'\cdot G\cdot \tau'$ is
\[\grapheo^{*p}*\sigma'\sigma\cdot(I^{*q}*G_{k_1,l_1}*\ldots*G_{k_r,l_r})\cdot \tau\tau',\]
and, by definition of $\Phi(G')$:
\begin{align*}
\Phi(G')&=t_{1,1}(I)^{*p}*\sigma'\sigma\cdot(I^{*q}*x_{k_1,l_1}*\ldots*x_{k_r,l_r})\cdot \tau\tau'\\
&=\sigma'\cdot( t_{1,1}(I)^*p*\sigma\cdot(I^{*q}*x_{k_1,l_1}*\ldots*x_{k_r,l_r})\cdot \tau)*\tau'\\
&=\sigma'\cdot \Phi(G)\cdot \tau'.
\end{align*}
Let us assume now that $\Phi(G')$ is defined for any graph with $N-1$ internal edges, for a given $N \geqslant 1$.
Let $G$ be a graph with $N$ internal edges and let $e$ be one of these edges. 
Let $G_e$ be a graph obtained by cutting this edge in two:
\begin{enumerate}
\item $V(G_e)=V(G)$.
\item $E(G_e)=E(G)\setminus \{e\}$, $I(G_e)=I(G)\sqcup \{e\}$, $O(G_e)=O(G)\sqcup \{e\}$, $IO(G_e)=IO(G)$, $L(G_e)=L(G)$.
\item $s_{G_e}=s_G$ and $t_{G_e}=t_G$.
\item For any $e'\in I(G_e)\sqcup IO(G_e)$, for any $f'\in O(G_e)\sqcup IO(G_e)$:
\begin{align*}
\alpha_{G_e}(e')&=\begin{cases}
1\mbox{ if }e'=e,\\
\alpha_{G}(e')+1\mbox{ if }e'\neq e,
\end{cases}
&\beta_{G_e}(f')&=\begin{cases}
1\mbox{ if }f'=e,\\
\beta_{G}(f')+1\mbox{ if }f'\neq e.
\end{cases}
\end{align*}
\end{enumerate}
Then $G=t_{1,1}(G_e)$ and $G_e$ has $N-1$ internal edges. We then put:
\[\Phi(G)=t_{1,1}\circ \Phi(G_e).\]
Let us prove that this does not depend of the choice of $e$. If $e'$ is another internal edge of $G$,
then:
\[(G_e)_{e'}=(12)\cdot (G_{e'})_e\cdot (12),\]
which implies, by definition of $\Phi(G_e)$ and $\Phi(G_{e'})$:
\begin{align*}
t_{1,1}\circ \Phi(G_e)&=t_{1,1}\circ t_{1,1}\circ \Phi((G_e)_{e'})\\
&=t_{1,1}\circ t_{1,1} \circ ((12)\cdot \Phi((G_{e'})_e)\cdot (12))\\
&=t_{1,1}\circ t_{2,2}\circ \Phi((G_{e'})_e)\\
&=t_{1,1}\circ t_{1,1}\circ \Phi((G_{e'})_e)\\
&=t_{1,1}\circ \Phi(G_{e'}).
\end{align*}
So $\Phi(G)$ is well-defined. Let $\sigma \in \sym_k$ and $\tau\in \sym_l$. Then:
\[(\sigma\cdot G\cdot \tau)_e=((1)\otimes \sigma)\cdot (G_e)\cdot ((1)\otimes \tau),\]
so:
\begin{align*}
\Phi(\sigma \cdot G\cdot \tau)&=t_{1,1}\circ \Phi((\sigma\cdot G\cdot \tau)_e)\\
&=t_{1,1}((1)\otimes \sigma)\cdot \Phi(G_e)\cdot ((1)\otimes \tau)\\
&=((1)\otimes \sigma)_1\cdot t_{1,1}\circ \Phi(G_e)\cdot ((1)\otimes \tau)_1\\
&=\sigma\cdot \Phi(G)\cdot \tau.
\end{align*}
where, for $\sigma\in\sym_k$ we use $\sigma_i$ for the permutation in $\sym_{k-1}$ defined by
\begin{equation*}
  \sigma_i(j) = \begin{cases}
               & \sigma(j) \quad\text{if }j\leq i-1, \\
               & \sigma(j-1)\quad\text{if }j\geq i. 
              \end{cases}
\end{equation*}
where $((1)\otimes \tau)_1$ is defined by (\ref{defalphak}).

We have therefore defined a map $\Phi:\textbf{GGr}\longrightarrow P$, compatible with the action of the symmetric groups.
Let us prove that for any graphs $G$, $G'$,
\[\Phi(G*G')=\Phi(G)*\Phi(G').\]
We proceed by induction on the number $N$ of internal edges of $G*G'$. If $N=0$, we put:
\begin{align*}
G&=\grapheo^{*p}*\sigma\cdot (I^{*q}*G_{k_1,l_1}*\ldots*G_{k_r,l_r})\cdot \tau,\\
G'&=\grapheo^{*p'}*\sigma'\cdots (I^{*q'}*G_{k'_1,l'_1}*\ldots*G_{k'_{r'},l'_{r'}})\cdot \tau'.
\end{align*}
We obtain:
\begin{align*}
G*G'&=\grapheo^{*(p+p')}*(\sigma \otimes \sigma')* (\mathrm{Id}_q\otimes c_{k_1+\ldots+k_r,q'}\otimes \mathrm{Id}_{k'_1+\ldots+k'_{r'}})\\
&\cdot (I^{q+q'}*G_{k_1,l_1}*\ldots*G_{k'_{r'},l'_{r'}})\cdot
(\mathrm{Id}_q\otimes c_{q',l_1+\ldots+l_r}\otimes \mathrm{Id}_{l'_1+\ldots+l'_{r'}}),
\end{align*}
which gives, by commutativity of $*$:
\begin{align*}
\Phi(G*G')&=t_{1,1}(I)^{*(p+p')}*(\sigma \otimes \sigma')* (\mathrm{Id}_q\otimes c_{l_1+\ldots+l_r,q'}
\otimes \mathrm{Id}_{l'_1+\ldots+l'_{r'}})\\
&\cdot (I^{q+q'}*x_{k_1,l_1}*\ldots*x_{k'_{r'},l'_{r'}})\cdot
(\mathrm{Id}_q\otimes c_{q',k_1+\ldots+k_r}\otimes \mathrm{Id}_{k'_1+\ldots+k'_{r'}})\\
&=t_{1,1}(I)^{*p}*\sigma\cdot(I^{*q}*x_{k_1,l_1}*\ldots*x_{k_r,l_r})\cdot \tau\\
&*t_{1,1}(I)^{*p'}*\sigma'\cdot(I^{*q'}*x_{k'_1,l'_1}*\ldots*x_{k'_{r'},l'_{r'}})\cdot \tau'\\
&=\Phi(G)*\Phi(G').
\end{align*}
If $N\geqslant 1$, let us take an internal edge $e$ of $G*G'$. If $e$ is an internal edge of $G$, then
$(G*G')_e=G_e*G'$, and:
\begin{equation*}
 \Phi(G*G')=t_{1,1}\circ \Phi((G*G')_e)=t_{1,1}\circ \Phi(G_e*G')=t_{1,1}(\Phi(G_e)*G')=t_{1,1}\circ \Phi(G_e)*\Phi(G')=\Phi(G)*\Phi(G').
\end{equation*}
If $e$ is an internal edge of $G'$, we obtain similarly that $\Phi(G'*G)=\Phi(G')*\Phi(G)$.
The result then 
follows from the commutativity of $*$ (axiom $2.(d)$ of Definition \ref{defi:Trap}).
So $\Phi$ is compatible with $*$. \\

It remains to prove the compatibility of $\Phi$ with the partial trace maps. By Lemma \ref{lemmemorphismes},
it is enough to prove that $\Phi$ is compatible with $t_{1,1}$. 
Let $G\in \Gr(k,l)$ be a graph, $e_1=\alpha^{-1}(1)$, $f_1=\beta^{-1}(1)$. We put $G'=t_{1,1}(G)$
and $e=\{e_1,f_1\}$ be the edge of $G'$ created in the process. There are five different cases:
\begin{enumerate}
\item If $e_1\in I(G)$ and $f_1\in O(G)$, then $e\in E(G')$ and $G'_e=G$. By construction of $\Phi(G')$:
\[\Phi\circ t_{1,1}(G)=\Phi(G')=t_{1,1}\circ \Phi(G'_e)=t_{1,1}\circ \Phi(G).\]
\item If $e_1\in IO(G)$ and $f_1\in O(G)$, let us put $j=\beta(e_1)$. Then there exists a graph $H$
such that $(1,j)\cdot G=I*H$. Then:
\begin{align*}
t_{1,1}(G)&=t_{1,1}((1,j)\cdot(I*H))=(1,\ldots,j)\cdot (t_{1,i}(I*H))=
(1,\ldots,j)\cdot H,
\end{align*}
so:
\begin{align*}
t_{1,1}\circ \Phi(G)&=t_{1,1}((1,j)\cdot (I*\Phi(H))\\
&=  (1,j)(1,\ldots,j-1)\cdot \Phi(H)\\
&=(1,\ldots,j)\cdot \Phi(H)\\
&=\Phi((1,\ldots,j)\cdot H)\\
&=\Phi\circ t_{1,1,}(G).
\end{align*}
\item If $e_1\in I(G)$ and $f_1\in IO(G)$: similar computation.
\item If $e_1,f_1\in IO(G)$, with $e_1\neq f_1$: similar computation.
\item If $e_1=f_1$ in $IO(G)$, then $G=I*H$ for a certain graph $G$ and $t_{1,1}(G)=\grapheo*H$.
Then:
\[\Phi\circ t_{1,1}(G)=\Phi(\grapheo)*\Phi(H)=t_{1,1}\circ \Phi(I)*\Phi(H)=t_{1,1}(\Phi(I)*\Phi(H)
=t_{1,1}\circ \Phi(G).\]
\end{enumerate}
So $\Phi$ is compatible with the partial trace maps. \end{proof}

\bibliographystyle{alpha}
\addcontentsline{toc}{section}{References}
\bibliography{biblio_new}

\end{document}